\documentclass[reqno,11pt]{amsart}

\usepackage{amsmath}
\usepackage{amsthm}
\usepackage{amssymb}
\usepackage{amscd}
\usepackage{xypic}
\usepackage{verbatim}
\usepackage[plainpages=false,colorlinks,hyperindex,pdfpagemode=None,bookmarksopen,linkcolor=red,citecolor=blue,urlcolor=blue]{hyperref}
\usepackage{pdflscape}
\usepackage{stmaryrd}

\usepackage{yfonts} 
\usepackage{palatino}
\usepackage{enumerate}
\usepackage{graphicx}

\usepackage{bbm}

\def\beq{\begin{equation}}
\def\eeq{\end{equation}}

\swapnumbers
\theoremstyle{plain}
\newtheorem{theorem}[subsubsection]{Theorem}

\newtheorem*{theorem*}{Theorem}
\newtheorem{proposition}[subsubsection]{Proposition}
\newtheorem*{proposition*}{Proposition}
\newtheorem{lemma}[subsubsection]{Lemma}
\newtheorem*{lemma*}{Lemma}
\newtheorem*{fact*}{Fact}
\newtheorem{corollary}[subsubsection]{Corollary}
\newtheorem*{corollary*}{Corollary}

\theoremstyle{definition}
\newtheorem{definition}[subsubsection]{Definition}

\theoremstyle{remark}
\newtheorem{remark}[subsubsection]{Remark}

\newtheorem{example}[subsubsection]{Example}

\renewcommand{\comment}[1] {  }

\DeclareFontFamily{OT1}{rsfs}{}
\DeclareFontShape{OT1}{rsfs}{n}{it}{<-> rsfs10}{}
\DeclareMathAlphabet{\mathscr}{OT1}{rsfs}{n}{it}

\newcommand{\from}{\leftarrow}

\newcommand{\codim}{\mathrm{codim}}

\newcommand{\Res}{\mathrm{Res}}

\newcommand{\Z}{\mathbb{Z}}

\newcommand{\adele}{\mathbb{A}_k}

\newcommand{\CC}{\mathbb{C}}

\newcommand{\RR}{\mathbb{R}}

 \newcommand{\xx}{\mathbf{x}}

\newcommand{\QQ}{\mathbb{Q}}

\newcommand{\Hom}{\operatorname{Hom}}

\newcommand{\Aut}{{\operatorname{Aut}}}

\newcommand{\Gm}{\mathbb{G}_m}
\newcommand{\Ga}{\mathbb{G}_a}

\newcommand{\GL}{\operatorname{GL}}

\newcommand{\PGL}{\operatorname{PGL}}

\newcommand{\SL}{\operatorname{SL}}

\newcommand{\SO}{{\operatorname{SO}}}

\newcommand{\spec}{\operatorname{spec}}
\newcommand{\Vol}{\operatorname{Vol}}
\newcommand{\diag}{{\operatorname{diag}}}
\newcommand{\ev}{\operatorname{ev}}
\newcommand{\im}{\operatorname{im}}

\newcommand{\std}{{\operatorname{std}}}

\newcommand{\RTF}{{\operatorname{RTF}}}

\newcommand{\Ob}{{\operatorname{ob}}}

\newcommand{\conj}{{\textrm{\small{-conj}}}}

\newcommand{\ob}{{\operatorname{ob}}}
\newcommand{\Isom}{{\operatorname{Isom}}}
\newcommand{\pt}{{\operatorname{pt}}}
\newcommand{\naive}{{\operatorname{naive}}}
\newcommand{\ps}{{\operatorname{pre-stack}}}
\newcommand{\PresX}{\operatorname{Pres}_{\mathcal X}}
\newcommand{\cc}{\mathfrak c}

\newcommand{\red}[1]{#1}
\newcommand{\Ff}{\mathcal F}
\newcommand{\F}{{\mathfrak F}} 

\newcommand{\id}{\operatorname{id}}
\newcommand{\op}{{\operatorname{op}}}

\newcommand{\x}{\mathbf{x}}

\makeatletter
\let\@cleartopmattertags\relax
\newcommand\articleend
 {\enddoc@text
  \let\authors\@empty
  \let\contribs\@empty
  \let\xcontribs\@empty
  \let\toccontribs\@empty
  \let\addresses\@empty
  \let\thankses\@empty
  \newpage
 }
\let\@wraptoccontribs\wraptoccontribs
\makeatother

\begin{document}
\numberwithin{equation}{section}
\setcounter{tocdepth}{1}
\title{The Schwartz space of a smooth \linebreak semi-algebraic stack}
\author{Yiannis Sakellaridis}
\email{sakellar@rutgers.edu}
\address{Department of Mathematics and Computer Science, Rutgers University -- Newark, 101 Warren Street, Smith Hall 216, Newark, NJ 07102, USA. \medskip  \linebreak  and \medskip \linebreak
\hspace{1cm} Department of Mathematics,
School of Applied Mathematical and Physical Sciences,
National Technical University of Athens,
Heroon Polytechneiou 9,
Zografou 15780, Greece.}

\begin{abstract}
Schwartz functions, or measures, are defined on any smooth semi-algebraic (``Nash'') manifold, and are known to form a cosheaf for the semi-algebraic restricted topology. We extend this definition to smooth semi-algebraic stacks, which are defined as geometric stacks in the category of Nash manifolds. 

Moreover, when those are obtained from algebraic quotient stacks of the form $X/G$, with $X$ a smooth affine variety and $G$ a reductive group defined over a number field $k$, we define, whenever possible, an ``evaluation map'' at each semisimple $k$-point of the stack, without using truncation methods. This corresponds to a regularization of the sum of those orbital integrals whose semisimple part corresponds to the chosen $k$-point.

These evaluation maps produce, in principle, a distribution which generalizes the Arthur--Selberg trace formula and Jacquet's relative trace formula, although the former, and many instances of the latter, cannot actually be defined by the purely geometric methods of this paper. In any case, the stack-theoretic point of view provides an explanation for the pure inner forms that appear in many versions of the Langlands, and relative Langlands, conjectures.
\end{abstract}

   \subjclass{Primary 58H05; Secondary 22A22, 11F70}

   \keywords{stacks, Schwartz spaces, orbital integrals}

\maketitle

\begin{flushright}
 \emph{To Joseph Bernstein, \\ in admiration and gratitude.}
\end{flushright}

\tableofcontents

\section{Introduction} 

\subsection{Overview}

It is an insight of Joseph Bernstein \cite{B-stacks} that several problems in representation theory and, in particular, the Arthur--Selberg trace formula and its generalizations (such as Jacquet's relative trace formula) should be understood in terms of quotient stacks of the form $\mathcal X=X/G$, where $X$ is an affine variety and $G$ a reductive group. For the Arthur--Selberg trace formula this stack is the adjoint quotient of a reductive group; for the relative trace formula it can be a more general stack of the form $H_1\backslash G/H_2$ (specializing to the adjoint quotient of $H$ when $G=H\times H$ and $H_1=H_2=$the diagonal copy of $H$ in $G$).

One piece of evidence for the relevance of algebraic stacks is the appearance, in all nice formulations of the Langlands conjectures (and their relative variants), of pure inner forms of the groups that one originally wants to consider. Compare, for example, with the local Langlands conjectures as stated by Vogan \cite{Vogan}, or the local Gan--Gross--Prasad conjectures \cite{GGP}, which have now been proven in most cases \cite{W-GP, MW-GP, BP, GI, Atobe}. 

Eventually, one should have a local and a global Langlands conjecture attached to (at least) stacks of the form $(X\times X)/G^\diag$, where $X$ is some ``nice'' (e.g., spherical) $G$-variety. (In the group case $X=H$, $G=H\times H$, this specializes again to the adjoint quotient of $H$.) Both conjectures should describe, in terms of Langlands parameters, the ``spectral decomposition'' of a distinguished functional on the local and global Schwartz space of the stack. Locally, this functional should be, essentially, the $L^2$-inner product of two functions on $X$, while globally it should be the pertinent ``trace formula''. It is not my intention in this article to provide a formulation of such conjectures, but to start building the theoretical framework necessary for such a formulation.
It should also be mentioned that there is a more advanced version of the local Langlands conjecture due to Kaletha \cite{Kaletha} which includes non-pure inner forms. It is clear that one needs to go beyond the point of view of the present article to incorporate Kaletha's conjecture.

Let $\mathcal X$ denote a \emph{smooth} algebraic stack over a number field $k$ or a local, locally compact field $F$. For non-smooth stacks we \emph{do not know} the appropriate definitions of Schwartz spaces etc., even when the stack is a variety --- however, this is an important problem, cf.\ \cite{BNS}. The goal of this article is to define:

\begin{itemize}
\item locally, the \emph{Schwartz space} $\mathcal S(\mathcal X(F))$, generalizing the well-known Schwartz space of smooth, rapidly decaying (or of compact support, in the non-Archimedean case) measures on the $F$-points of a smooth variety;
\item globally, a notion of ``evaluation'' of a Schwartz function at a closed $k$-point $x:\spec k \to \mathcal X$, where the space of Schwartz functions is by definition a restricted tensor product of the local Schwartz spaces, divided by a Tamagawa measure. I do this under a certain set of restrictions on the algebraic stack.
\end{itemize}

For smooth stacks of the form $\mathcal X= X/G$, an equivalent definition of the local Schwartz space $\mathcal S(\mathcal X(F))$ was suggested to me by Joseph Bernstein: without loss of generality, $G\simeq \GL_n$, and then we can set
\begin{equation}\label{firstdef}
\mathcal S(\mathcal X(F)):= \mathcal S(X(F))_{G(F)},
\end{equation}
the coinvariant space of the space of Schwartz measures on $X(F)$. (For topological spaces, we define coinvariants by taking the quotient by the \emph{closure} of the span of vectors of the form $g\cdot v-v$; a more refined notion is provided by derived categories that we will encounter below.) The requirement $G\simeq \GL_n(F)$ is in order to have trivial Galois cohomology, i.e., so that the $F$-points of $X$ surject onto the $F$-points of $\mathcal X$ --- otherwise the definition would depend on the presentation of the stack. This explains the appearance of pure inner forms: for another presentation $\mathcal X = W/H$, where $H$ is an arbitrary reductive group, choose a faithful representation $H\hookrightarrow G=\GL_n$ and set $X=W\times^H G$. Then $\mathcal X= X/G = W/H$, but in terms of the quotient $W/H$, an $F$-point of $\mathcal X$ corresponds to a torsor $T$ (equivalently: a pure inner form) of $H$ over $F$, together with an $H$-equivariant map: $T\to W$. As we will see, the $G(F)$-coinvariants of $\mathcal S(X(F))$ can be identified with
$$ \bigoplus_{T} \mathcal S(W^T(F))_{H^T(F)},$$
where $T$ runs over all isomorphism classes of $T$-torsors, $W^T = W\times^H T$ and $H^T$ is the inner form $\Aut^H(T)$.

In this article I follow a more general approach for the definition of Schwartz spaces, which turns out to specialize to Bernstein's in the case of quotient stacks. To define the Schwartz space for a general smooth algebraic stack, I extend and exploit the following well-known fact:
$$\mbox{\emph{Schwartz spaces form a cosheaf.}}$$
This is true in the semi-algebraic (restricted) topology on the $F$-points of a smooth variety or, more generally, on a smooth semi-algebraic (``Nash'') manifold, and has been systematically exploited, e.g., in \cite{DuCloux, AGSchwartz}. It is, however, true in a stronger sense, namely: in the \emph{smooth} Grothendieck topology on the category of Nash manifolds. This is the observation that allows us to define spaces of Schwartz measures as a cosheaf on the smooth topology over a smooth semi-algebraic stack. 

These Schwartz spaces are nuclear Fr\'echet spaces in the Archimedean case, and the push-forward maps with respect to smooth morphisms of varieties are \emph{strict}, i.e., have closed image. Taking the functor of global sections over our stack $\mathcal X(F)$, we obtain a well-defined element $\mathcal S_\bullet(\mathcal X(F))$ in the derived category of nuclear Fr\'echet spaces. (For the notions of derived categories of non-abelian categories see Appendix \ref{app:homology}.) For a smooth quotient stack $\mathcal X=X/G$, the zeroth homology essentially coincides with Bernstein's definition that we saw above, and of course higher homology groups correspond to the higher derived functors of coinvariants.

In the non-Archimedean case, the same statements are true in the derived category of vector spaces without topology. It should be mentioned, however, that in positive characteristic there are serious issues with semi-algebraic geometry. On the other hand, one does not need the semi-algebraic structure to define Schwartz spaces in the non-Archimedean case: the usual notion of compactly supported smooth measures or functions makes sense on $F$-analytic manifolds. 
Thus, in positive characteristic one should read any mention of ``Nash'' manifolds in this paper as referring to $F$-analytic manifolds, and any mention of semi-algebraic restricted topology as referring to the usual topology on $F$-points. In the non-Archimedean case in characteristic zero, the reader can choose either of the two approaches. Although the definitions make sense for $F$-analytic manifolds, there is something to be gained by restricting to the semi-algebraic case, namely: one preserves a notion of ``polynomial growth''. Moreover, one might wish to replace the usual Schwartz spaces of compactly supported smooth measures by the Fr\'echet spaces of ``almost smooth'' measures defined in \cite[Appendix A]{SaBE1}.

The second goal, defining an ``evaluation map'', is performed only for smooth quotient stacks of the form $\mathcal X=X/G$ and corresponds to defining regularized orbital integrals for the $G(\adele)$-action on $\mathcal S(X(\adele))$. There is little stacky here, although a lot of work goes into showing that the constructions are independent of choices and give rise to well-defined functionals on the Schwartz space of a stack. Once we choose a closed (``semisimple'') point $x: \spec k\to \mathcal X$, corresponding to a $G(k)$-orbit on $X(k)$ (by appropriately choosing the presentation $\mathcal X=X/G$ among pure inner twists), the goal is to define a regularization for the sum of $G(\adele)$-orbital integrals corresponding to $G(k)$-orbits on $X(k)$ whose ``semisimple'' part is isomorphic to $x$. 
In the literature, this has been done almost invariably by generalizing the methods of truncation \cite{A-truncation} in the monumental work of Arthur (which, in turn, is a vast generalization of Selberg's). 

There is a very general approach to truncated orbital integrals by J.\ Levy \cite{Levy} (who, however, does not produce invariant distributions), and several variants of the truncation method which are better suited for applications to the relative trace formula by Jacquet, Lapid, Rogawski, Ichino, Yamana, Zydor and others \cite{JLR,IY, Zydor1, Zydor2, Zydor3}.  Joseph Bernstein sometimes expressed the wish to see this done without truncation, and this is what I do here, in a restricted setting that does not include the adjoint quotient of the group (but for reasons that cannot be addressed in a purely geometric way). It does include, however, many other interesting examples of relative trace formulas, such as those related to the Gross--Prasad conjectures (see Example \ref{exampleGP}). Its scope is not restricted to the trace formula and its variants; see Example \ref{exampleKR} for an application to the Kudla-Rallis regularized period of the theta correspondence.

It is well-known, of course, that truncation at first does not produce an invariant distribution. What I do here, before attempting to integrate, consists of two steps. First, replace the problem with a linear one, namely replace the quotient $X/G$, by the \'etale neighborhood $V/H$ of the given semisimple point $x$, provided by Luna's \'etale slice theorem. Here $H = \Aut(x)$ and $V$ is the $H$-representation obtained by considering the action of $H$ on the normal bundle to the $G$-orbit corresponding to $x$. In the usual trace formula, this is nothing other than replacing the adjoint quotient of the group by its Lie algebra version (or rather, the adjoint Lie algebra quotient for the centralizer of the given semisimple point). 

The second step is the crucial one, and it consists in analyzing the asymptotic behavior of the function:
$$\Sigma_V f: h\mapsto \sum_{\gamma \in V(k)} f(\gamma h),$$
where $f$ is a Schwartz function on the vector space $V(\adele)$. The above function is an ``asymptotically finite'' function on $[H]=H(k)\backslash H(\adele)$, i.e., it coincides, up to a Schwartz function, with functions which have specified behavior of multiplicative type close to infinity. To describe this asymptotic behavior one needs a quite natural class of compactifications of $[H]$ which I term ``equivariant toroidal'', because they are described by data similar to those of toric varieties.\footnote{However, these compactifications are not related to the toroidal compactifications of Shimura varieties but, rather, to the ``reductive Borel--Serre'' compactification.} The ``asymptotically finite'' functions are then described as sections of a certain cosheaf on such a compactification, which depends on some characters, or \emph{exponents}, that can easily be read off from the weights of the representation $V$.
Although the method and arguments here are not fundamentally different from those used in truncation methods, the approach is completely general and provides a conceptually clear answer for when one should not expect to have an invariant, regularized orbital integral out of purely geometric considerations: when the ``exponents'' of $\Sigma_V f$ are critical, i.e., equal to the modular character whose inverse describes the decay of volumes at infinity.

In cases with critical exponents, no reasonable invariant distribution can be defined, to the best of my understanding, by purely geometric methods: this would amount to finding a continuous invariant extension, to the Fr\'echet space of functions on $\RR^\times_+$ that are of rapid decay in a neighborhood of $\infty$ but constant (up to a rapidly decaying function) in a neighborhood of zero, of the distribution defined by a multiplicative Haar measure. Such a continuous invariant extension does not exist as follows, e.g., from Tate's thesis.
The critical case includes the adjoint quotient of a group; as we know from Arthur \cite{A-invariant1, A-invariant2}, a combination of geometric and spectral considerations can give rise to an invariant distribution on the Schwartz space in this case.

\subsection{Outline of the paper}

In Section \ref{sec:stacks} I introduce ``Nash stacks'' over a local field $F$ in characteristic zero.   They are generalizations of Nash manifolds, i.e., $F$-analytic manifolds with a restricted topology and a finite open cover by analytic submanifolds of $F^n$ described by polynomial equalities and inequalities (and some generalizations of those). The main advantage of Nash, over arbitrary $F$-analytic, manifolds is that they come equipped with a notion of ``polynomial growth'' for functions, which allows one to define spaces of Schwartz functions and measures. Nash stacks are a generalization, which allows one to talk of quotients $X/G$ of Nash manifolds by Nash groups, as if $G$ were acting freely and properly on $X$. As mentioned above, in positive characteristic we simply work with $F$-analytic manifolds, losing this notion of polynomial growth, because some statements of semi-algebraic geometry that we need are not known. Rather than considering two different cases when the arguments are identical, I leave it to the reader to translate all statements from the semi-algebraic to the analytic setting.

The reader who is not familiar with stacks should be reminded of the quotient $X/G$ of algebraic varieties, where $X  = \Res_{E/F} \Gm$, the restriction of scalars of the multiplicative group from a quadratic extension $E$ of a field $F$ to $F$, and $G$ is the kernel of the norm map from $E^\times$ to $F^\times$. The quotient is an algebraic variety isomorphic to $\Gm$ (over $F$), in such a way that the map $X\to X/G$ is the norm map: $\Res_{E/F} \Gm \to \Gm$. However, an $F$-point of the quotient does not necessarily correspond to a $G(F)$-orbit on $X(F)=E^\times$, since the norm map is not surjective on $F$-points: $E^\times \to F^\times$; rather, it corresponds to a subvariety of $X$, defined over $F$, which is a \emph{$G$-torsor} ($G$-principal homogeneous space), although it may not have any $F$-points.
 
Generalizing this observation, for any group variety $G$ acting on a variety $X$ over $F$, the $Y$-points of the algebraic stack $X/G$ (where $Y$ is $\spec F$, or any scheme over it) correspond to (\'etale) $G$-torsors $T$ over $Y$, together with a $G$-equivariant map: $T\to X$. This is not a set but a category fibered over the category of $F$-schemes $Y$, and objects in the fiber over $\spec F$ (i.e., ``$F$-points'') may have a non-trivial group of automorphisms when the map $T\to X$ is not an immersion. 

In \S \ref{ssdefstacks} I define Nash stacks by transferring this idea to the category $\mathfrak N$ of Nash manifolds. Again, one should not try to describe ``points'' of $\mathfrak X$, since they do not form a set; instead, one should describe what is a morphism from a Nash $F$-manifold $Y$ to $\mathfrak X$ --- thus obtaining a category fibered over $\mathfrak N$.
For stacks of the form $\mathfrak X=X/G$ with $X$ a Nash manifold and $G$ a Nash group acting on it, a morphism from $Y$ to $\mathfrak X$ corresponds to a $G$-torsor $T$ over $Y$ (in the category of Nash manifolds), together with a $G$-equivariant morphism $T\to X$ of Nash manifolds.

General Nash stacks are defined as ``geometric stacks in the category $\mathfrak N$ of Nash manifolds'', that is: stacks fibered over this category (endowed with the semi-algebraic, equivalently the smooth, Grothendieck topology) and admitting a smooth (i.e., submersive) surjective map (``presentation'') from a Nash stack.

I also discuss how to obtain a Nash stack from a smooth algebraic stack $\mathcal X$ over $F$ --- this is completely achieved in the case of stacks which are quotients of smooth varieties by linear groups, Proposition \ref{groupquotient}. The idea here is to use the groupoid $R_X(F)\rightrightarrows X(F)$, where $X\to \mathcal X$ is a presentation by a smooth scheme and $R_X = X\times_{\mathcal X} X$, but there is a problem with this approach: the Nash groupoid obtained depends on the choice of algebraic presentation. The reason is easy to see and lies at the heart of the motivation for this paper: as we saw above in the example of the norm map, a quotient algebraic stack $X/G$ has more $F$-points than can be accounted for by the $F$-points of $X$. An $F$-point of $X/G$ represented by a torsor $T$ and an equivariant map $T\to X$ can be accounted for by the $F$-points of the ``pure inner'' twist: $X^T = X\times^G T$, which is a space with an action of the inner form $G^T=\Aut^G(T)$. Thus, we only obtain a Nash stack when there is an ``$F$-surjective presentation'' $X\to \mathcal X$ of the algebraic stack $\mathcal X$, i.e., one for which the map of isomorphism classes $X(F)\to \mathcal X(F)$ is surjective. A more general and presentation-free approach to construct a stack is sketched in Appendix \ref{app:Nash}, but I do not know if it always produces a Nash stack.

In Section \ref{sec:localSchwartz}, I introduce the Schwartz space of a Nash stack, generalizing the space of (complex-valued) Schwartz measures on a Nash manifold. Measures are better suited than functions for this purpose, because they have natural push-forwards, and push-forwards of Schwartz measures under smooth (submersive) semi-algebraic maps are also Schwartz. As mentioned above, the main observation that makes the definition possible is that Schwartz measures form a cosheaf for the smooth topology (Proposition \ref{cosheaf-smooth}). Thus, we can postulate that the collection of all Schwartz spaces $\mathcal S(X)$, for all smooth covers $X\to \mathfrak X$ of a given Nash stack $\mathfrak X$ ``is'' its Schwartz cosheaf $\mathfrak G_{\mathfrak X}$. By the Schwartz space $\mathcal S(\mathfrak X)$ we mean the ``global sections'' of this cosheaf, i.e., the (coarse) zeroth homology of the canonical element $\mathcal S_\bullet(\mathfrak X)$ in the derived category of nuclear Fr\'echet spaces (vector spaces without topology in the non-Archimedean case) that one obtains from \v Cech homology of this cosheaf. A main result of this section is Theorem \ref{independent}, which says that this ``complex'' can be computed from any presentation of the Nash stack, which is a consequence of the fact that the Schwartz cosheaves of Nash manifolds are acyclic for the smooth topology, Proposition \ref{Poincare}.
 Proposition \ref{groupquotient-Schwartz} and its corollary confirm that, for quotient stacks, this definition coincides with the one suggested by Bernstein.
 
 As the referee has pointed out, there is a need for extending the notion of  Schwartz space and, in particular, Theorem \ref{independent} to non-trivial ``bundles'' over a Nash stack in order to include, for example, $G$-coinvariants (where $G$ is a Nash group) of the Schwartz sections of a $G$-equivariant vector bundle over a Nash manifold $X$. Such ``quotients'' with $X/G$ a (real or $p$-adic) Nash manifold appear, for example, in Bernstein's ``Frobenius reciprocity'' for distributions \cite{Be-distr}, and extensions of this idea by Baruch \cite{Baruch} and Aizenbud and Gourevitch \cite{AGdeRham}. While clearly very important for applications, the study of Schwartz sections of non-trivial ``bundles'' on stacks is beyond the scope of the present article; I hope to return to it in the near future.

In the following sections the goal is to define the global ``evaluation maps'' for algebraic stacks of the form $X/G$ defined over a global field, where $X$ is smooth affine and $G$ is reductive. When $X$ is a smooth algebraic variety defined over a global field $k$, one has a global Schwartz space $\mathcal F(X(\adele))$ of functions on the adelic points of $X$. The difference between measures and functions is not very significant globally, where we have Tamagawa measures, although it can only be bridged without extra data at the level of ``stalks'' over a $k$-point of the invariant-theoretic quotient $\cc = X\sslash G$. Ignoring this technical difficulty for now (for which we will need the whole Section \ref{sec:stalks} to carefully develop the appropriate notions of stalks), we have a global Schwartz space $\mathcal S(\mathcal X(\adele))$ for a smooth quotient stack $\mathcal X= X/G$, which should be the right space of ``test functions'' for the trace formula and its generalizations. The ``relative trace formula'' for $\mathcal X$ is defined in \S \ref{ssRTF} as the sum
\begin{equation} \RTF_{\mathcal X}: f\mapsto \sum_x \ev_x,
\end{equation}
where $x$ runs over isomorphism classes of closed (``semisimple'') $k$-points into $\mathcal X$, and $\ev_x$ is a certain ``evaluation'' functional at $x$. 

The evaluation functional is the analog of (regularized) orbital integrals: We may assume without loss of generality that the point $x$ corresponds to a semisimple $G(k)$-orbit on $X(k)$, and our evaluation map is a regularization of the sum of $G(\adele)$-orbital integrals corresponding to all $G(k)$-orbits on $X(k)$ whose semisimple part (cf.\ Proposition \ref{Jordan}) is isomorphic to $x$. In the present paper I only define these evaluation maps when the (reductive) stabilizer $H$ of the given semisimple $k$-point is connected, for reasons of simplicity. More importantly, as mentioned earlier, these evaluation maps are defined whenever there are no ``critical exponents'', i.e., no logarithmic divergence. 

The definition is given in Section \ref{sec:evaluation} with the help of the linearized version of the quotient stack, i.e., with the help of Luna's \'etale slice theorem. In the world of the trace formula, this linearized version is known as the ``Lie algebra version''. Luna's \'etale slice Theorem \ref{lunathm} allows us to replace $\mathcal X$, in the neighborhood of the given semisimple $k$-point $x$ corresponding to a closed $G$-orbit $C$ on $X$, by the quotient $N_C X/G$, the normal bundle of $C$ divided by the action of $G$. Thus, we arrive at stacks of the form $V/H$, where $V$ is a representation of a reductive group $H$ (the automorphism group of $x$), and we have isomorphisms between the stalk of the global Schwartz space $\mathcal S(\mathcal X(\adele))$ around $x$ and the stalk of $\mathcal S(\mathcal V(\adele))$ around the point corresponding to $0\in V$. These isomorphisms depend on choices, though, and a lot of care is taken to show that the definitions given are independent of those choices. Luna's slice theorem and its corollaries are discussed in Section \ref{sec:stalks}.

In Section \ref{sec:orbitalintegrals}, the main goal is to define the functional ``evaluation at zero'' for $V/H$; this is the core of the argument, and has nothing to do with stacks and their Schwartz spaces, hence can be read independently.  Having done that, we check in \S \ref{ssinvariance} that the definition does not depend on the choice of \'etale neighborhood $V/H$, therefore we have a well-defined map $\ev_x$ on the global Schwartz space of $\mathcal X$ (Definition \ref{evX}). 

The regularization technique used here rests upon a description of the asymptotic behavior of the function: 
$$\Sigma_V f(h) := \sum_{\gamma\in V(k)} f(\gamma h) $$
on $[H]=H(k)\backslash H(\adele)$, instead of a description of its truncated orbital integrals. The main result, Theorem \ref{isfinite} says that this function is ``asymptotically finite'', where ``finite'' means that the function is, up to a rapidly decaying one, a generalized eigenfunction with respect to some multiplicative group actions. To describe such ``multiplicative group actions'' one needs a theory of ``equivariant toroidal compactifications''. These are blow-ups of the ``reductive Borel--Serre compactification'' which are described by some toric-type invariants, more precisely: a fan in the Weyl chamber of anti-dominant coweight for $H$. I think that this description of the function $\Sigma_V f$ is very natural, and it gives rise to a natural notion of regularized orbital integral (essentially as the analytic continuation of a Mellin transform). I note that it is \emph{easier to describe} the asymptotic behavior of this function, than it is to describe the asymptotic behavior of its truncated orbital integrals: the behavior of the function is partitioned in cones of the anti-dominant chamber that are dual to the weights of $H$ acting on $V$, while for the behavior of truncated orbital integrals one needs several projections of those weights onto the faces corresponding to proper parabolics.
The argument, which has also been used with truncation methods, essentially rests on the fact that the ``nilpotent cone'' of $V$ consists of those elements $v$ for which there is a cocharacter $\lambda$ into $H$ with $\lim_{t\to 0} v\cdot \lambda(t) =0$; and that along each cocharacter, the function becomes $\Sigma_V f$ becomes multiplicative up to a rapidly decaying function, as can be seen from a simple application of the Poisson summation formula on the vector space $V$.

\subsection{Acknowledgements}

This article is dedicated to Joseph Bernstein, in deep appreciation of his mathematical genius and gratitude for everything that he has taught me. It is impossible to overstate the impact that his ideas, and his generosity in sharing them, have had on the world of mathematics. It has been the greatest privilege of my mathematical life to have had him as my post-doctoral mentor, besides frequent mathematical conversations that I still enjoy with him.

I would also like to thank Pierre-Henri Chaudouard for inviting me to Paris in January 2015, during which some of the ideas in this paper were worked out, and for many conversations on the trace formula which proved very fruitful for the present paper, including references to the work of Jason Levy and explanations on the thesis of Michal Zydor. Finally, I thank Brian Conrad, Fran\c{c}ois Loeser and Akshay Venkatesh for helpful feedback on some questions, Dmitry Gourevitch for several corrections on an earlier draft, and the referee for a careful reading and many insightful comments.

This work was supported by NSF grants DMS-1101471 and DMS-1502270.

\subsection{General notation}

Here is a general guide to the notation most frequently used. 

\begin{itemize}
\item $k$ is used for a number field, $\adele$ for its ring of adeles, $F$ for a local, locally compact field.
\item $\mathcal S$ is used for ``Schwartz'' spaces of complex-valued measures, $\mathcal F$ is used for ``Schwartz'' spaces of functions. These are always sections of certain cosheaves, and they do not have to be ``Schwartz'' in the usual sense of rapid decay: in Section \ref{sec:orbitalintegrals}, I also define extensions of these cosheaves, where the measures/functions have ``asymptotically finite'' behavior. That means that, instead of being of rapid decay at infinity, they extend over a compactification to sections of vector bundles determined by characters of tori.
\item The quotient algebraic stack of an algebraic variety $X$ by a group $G$, or the quotient Nash stack of a Nash manifold $X$ by a Nash group $G$, will simply be denoted by $X/G$; it is more usual in the literature to denote it by $[X/G]$. If $X$ is an affine variety (over, say, the field $k$), the invariant-theoretic quotient $\spec k[X]^G$ will be denoted by $X\sslash G$.
\item I typically use calligraphic letters ($\mathcal X, \mathcal Y$...) to denote algebraic stacks, and gothic letters ($\mathfrak X, \mathfrak Y$...) for Nash stacks (or $\mathcal X(F), \mathcal Y(F)$, when the Nash stack is obtained by taking $F$-points of an algebraic stack). For a stack $\mathcal X$ over some category (such as schemes or Nash manifolds), and an object $U$ in that category, the groupoid of $\mathcal X$ over $U$ is denoted by $\mathcal X_U$, and its set of isomorphism classes by $\mathcal X(U)$. 
\item If $H$ is an algebraic group defined over $k$, the automorphic quotient space $H(k)\backslash H(\adele)$ is denoted by $[H]$.
\end{itemize}

\section{Nash stacks} \label{sec:stacks}

In this section we work over a local field $F$.

\subsection{General definition} \label{ssdefstacks}

When $F$ has characteristic zero, we let $\mathfrak N$ be the category of Nash (smooth semi-alge\-braic) $F$-manifolds. I point the reader to \cite{DuCloux, AGSchwartz} for a discussion and references on Nash manifolds over $\mathbb R$. For definitions that include the complex case, cf.\ \cite{TT, KB}. In the non-Archimedean case, I do not know any reference for Nash manifolds, but their construction can clearly be performed in exactly the same manner by glueing a \emph{finite} number of smooth semi-algebraic submanifolds of $\mathbb A^n(F)=F^n$ \footnote{I will often be using $\mathbb A^n$, the symbol for affine space, to emphasize the semi-algebraic structure that is deduced from the algebraic one.} as in \cite[\S 3.3]{AGSchwartz}. I will recall the definitions below, and point the reader to \cite{Denef} for a discussion of semi-algebraic sets and functions in the non-Archimedean case. 

In positive characteristic there are many issues with semi-algebraic geometry. For the purposes of this paper, in order to proceed, one would need to know that Macintyre's theorem \cite{Denef} holds for \emph{submersive} (``smooth'') semi-algebraic maps: the image of such a map is also semi-algebraic. Since this statement is, to the best of my knowledge, not known, one needs to translate all statements to the category of $F$-analytic manifolds; hence, $\mathfrak N$ will stand for this category, and every reference of ``Nash'' manifolds should be replaced by ``$F$-analytic'' manifolds. The restricted topology of open semi-algebraic subsets, to be used for Nash manifolds, should be replaced by the usual topology on $F$-analytic manifolds.

Back to characteristic zero, I summarize the definitions of Nash manifolds and related notions:

\begin{definition}
A \emph{semi-algebraic subset} of $\mathbb A^n(F)$ is any subset obtained by boolean combinations (finite intersections, finite unions, complements) of subsets of the form
$$\{ x\in F^n| \exists y \,\, f(x) =y^m\}$$
where $f\in F[X_1,\dots, X_n]$, $m\in \mathbb N$. The above sets generate the closed sets for a restricted topology on $\mathbb A^n(F)$, the \emph{semi-algebraic topology}.
\end{definition}

I remind the reader that ``restricted'' refers to the fact that infinite unions of open sets do not need to be open. Any references to ``closed'', ``open'' sets and to ``sheaves'' will refer to this topology, unless otherwise specified.

Notice that we get all polynomial equalities by setting $m=0$, and in the real case we get all polynomial inequalities from taking complements of the condition: $\exists y\,\, f(x)=y^2$. In the complex case, on the other hand, the allowed inequalities can all be written in the form $f(x)\ne 0$, and hence in that case semi-algebraic subsets of $\mathbb A^n(F)$ are precisely the constructible subsets. For inequalities derived from the defining condition in the non-Archimedean case, s.\ \cite[Lemma 2.1]{Denef}.

\begin{definition}
A \emph{smooth semi-algebraic} or \emph{Nash submanifold} $M$ of $\mathbb A^n(F)$ is a closed semi-algebraic subset which is also an $F$-analytic submanifold. It is equipped with the induced semi-algebraic restricted topology, and the sheaf $\mathfrak O_M$ of \emph{smooth semi-algebraic} or \emph{Nash functions}, i.e., those $F$-valued functions which are $F$-analytic and whose graph is a semi-algebraic subset of $M\times \mathbb A^1(F)$.
\end{definition}

\begin{definition}
A \emph{Nash $F$-manifold} is an $F$-analytic manifold $M$ equipped with a restricted topology and a sheaf $\mathfrak O_M$ of $F$-valued functions, which admits a finite open cover $M=\bigcup_i M_i$ such that each $(M_i,\mathfrak O_M)$ is isomorphic to a Nash submanifold of $\mathbb A^n(F)$ together with its sheaf of Nash functions. A morphism: $X\to Y$ of Nash manifolds is a Nash map between them, i.e., an $F$-analytic map whose graph is a semi-algebraic subset of $X\times Y$.
\end{definition}

In the complex case, Nash manifolds are simply smooth algebraic varieties (of finite type over $\CC$), and Nash maps/functions are algebraic.

We now proceed to define Nash stacks. The approach is analogous to the definition of differentiable stacks in \cite{BX}.

We consider the category $\mathfrak N$ of Nash $F$-manifolds, where the morphisms are Nash maps, as a site, equipped with the Grothendieck topology generated by \emph{covering families} $\{U_i\to X\}_{i\in I}$ such that:
\begin{enumerate}[(i)]
\item the set of indices $I$ is finite;
\item the (Nash) maps $U_i\to X$ are \'etale, that is: local diffeomorphisms;
\item the total map $\bigsqcup U_i\to X$ is surjective. 
\end{enumerate}

In any case other than the complex case, the second condition can be replaced by the condition that the maps are open embeddings, since every \'etale morphism becomes an open embedding over elements of a finite open cover. In fact, it is true, more generally, over $\RR$ or $p$-adic fields (see \cite[Appendix A]{AGdeRham} for $\RR$, and it can similarly be proven for $p$-adic fields) that every \emph{smooth} --- that is: \emph{submersive} --- semi-algebraic map between Nash manifolds admits local Nash sections over a finite open cover of the image. In the complex (algebraic) case, any smooth map is known to admit sections locally in the \'etale topology. Hence, in every case, this Grothendieck topology is equivalent to the Grothendieck topology generated by smooth covers, i.e., the one obtained by replacing the second condition above by the requirement that the maps $U_i\to X$ are smooth.

For a covering family as above, and an ordered set of indices $(i_1, i_2, \dots, i_r)$ we will be denoting $U_{i_1}\times_U \times \cdots \times_U U_{i_r}$ by $U_{i_1 i_2 \cdots i_r}$.

Let $\mathfrak X\to \mathfrak N$ be a category fibered in groupoids. Recall that this means that:
\begin{enumerate}[(i)] 
\item for every arrow $V\to U$ in $\mathfrak N$, and every object $x$ of $\mathfrak X$ lying over $U$, there exists an arrow $y\to x$ in $\mathfrak X$ lying over $V\to U$;
\item for every diagram of the form $W\to V \to U$ in $\mathfrak N$ and arrows in $\mathfrak X$: $z\to x$ lying over $W \to U$ and $y\to x$ lying over $V \to U$, there exists a unique arrow $z\to y$ lying over $W\to  V$, such that the composition $z\to y\to x$ equals the given $z\to x$.
\end{enumerate}

The full subcategory of objects in $\mathfrak X$ over a given object $U$ in $\mathfrak N$ (the \emph{fiber} of $\mathfrak X$ over $U$) is denoted by $\mathfrak X_U$, and its set of isomorphism classes by $\mathfrak X(U)$. An object $u\in \mathfrak X_U$ will sometimes be denoted by $u_U$ or $(U,u)$, and thought of as a $1$-morphism in the $2$-category of categories fibered in groupoids over $\mathfrak N$: 
$$u:U \to \mathfrak X.$$

The axioms imply that for any morphism $V\to U$ in $\mathfrak N$ and any $x_U \in \ob(\mathfrak X_U)$, the base change $x_V\to x_U$, $x_V \in \ob(\mathfrak X_V)$ is defined up to isomorphism. Categories fibered in groupoids over $\mathfrak N$ form a $2$-category, with $1$-morphisms being the functors which are compatible with the ``structure morphism'' to $\mathfrak N$, and $2$-morphisms being isomorphisms between those functors. Fiber products are defined in a standard way \cite[\S (2.2.2)]{LMB} for any pair of $1$-morphisms $\mathfrak X\to \mathfrak Y\from \mathfrak X'$, with objects of $\mathfrak X\times_{\mathfrak Y}\mathfrak X'$ over $U\in \Ob(\mathfrak N)$ being the triples: $(x, x', g)$, with $x\in \Ob(\mathfrak X_U)$, $x'\in\Ob(\mathfrak X'_U)$ and $g$ an arrow (isomorphism) between the image of $x$ and the image of $x'$ in $\mathfrak Y_U$.

Recall that $\mathfrak X$ is called a \emph{stack} if ``both the isomorphisms and the objects are glued from local data'' (with respect to the given Grothendieck topology), that is:
\begin{enumerate}[(i)]
\item For every object $U$ in $\mathfrak N$ and two objects $x,y$ in $\mathfrak X_U$, the pre-sheaf\footnote{The definition of this pre-sheaf requires choices of base changes, cf.\ \cite[Tags \href{http://stacks.math.columbia.edu/tag/02Z9}{02Z9} and \href{http://stacks.math.columbia.edu/tag/02XJ}{02XJ}]{stacks-project}. The conditions of it being a sheaf are independent of choices, however.} of isomorphisms between $x$ and $y$:
$$ \Isom(x,y): \mathfrak N_U \to (Set)$$
$$ (V\to U)\mapsto \Hom_{\mathfrak X_V} (x_V, y_V)$$
is a sheaf. That is, any isomorphism $\phi:x\to y$ is uniquely determined by its restrictions $\phi_i: x_{U_i}\to y_{U_i}$ to any covering family $\{U_i\to U\}_i$, and vice versa any family of such isomorphisms $\phi_i$ with $\phi_i| U_{ij} = \phi_j| U_{ij}: x_{U_{ij}}\to y_{U_{ij}}$ determines an isomorphism $\phi: x\to y$. (This condition defines a \emph{pre-stack}.)
\item Given $\{U_i\to U\}_i$ a covering family, $x_i \in \ob(\mathfrak X_{U_i})$ and morphisms: $\phi_{ij}: x_{i, U_{ij}} \to x_{j, U_{ij}}$ whose restrictions to $U_{ijk}$ satisfy the cocycle condition: $\phi_{jk}\circ \phi_{ij} = \phi_{ik}$, there exists an object $x$ in $\mathfrak X_U$ and a family of isomorphisms $\phi_i: x_{U_i} \to x_i$, such that $\phi_j|U_{ij}= \phi_{ji} \circ (\phi_i|U_{ij})$. The object $x$ and the isomorphisms $\phi_i$ are necessarily unique up to unique isomorphism, by (i).
\end{enumerate}

We say that a stack as above is \emph{representable} if there is a Nash manifold $X$, viewed as the stack $\Hom(\bullet, X)$ over $\mathfrak N$, together with an equivalence: $X\xrightarrow{\sim}\mathfrak X$. We say that a $1$-morphism of stacks: $\mathfrak X\to \mathfrak Z$ is \emph{smooth}, or a \emph{representable submersion}, if for every morphism: $U\to \mathfrak Z$, where 
$U\in\ob(\mathfrak N)$ viewed as a stack, the fiber product $\mathfrak X \times_{\mathfrak Z} U$ is representable by a Nash manifold, and the induced morphism: $\mathfrak X\times_{\mathfrak Z} U \to U$ is a submersion. 

\begin{remark}
As in \cite{BX}, we first define representable submersions, before defining representable morphisms in general; the reason is that the above definition cannot be used for a general morphism: $X\to Z$ of Nash manifolds. Indeed, for another morphism $U\to Z$, the fiber product $X\times_{Z} U$ is not necessarily defined as a Nash manifold. It is, however, if $X\to Z$ is smooth, i.e., a submersion. (This is a difficulty that does not exist, for example, in the category of schemes, but does exist in other geometric categories, such as differentiable manifolds.)
\end{remark}

A $1$-morphism $f:\mathfrak X\to \mathfrak Y$ of stacks over $\mathfrak N$ is called an \emph{epimorphism} if every $y\in \mathfrak Y_U$, where $U$ is a Nash stack, lifts \'etale-locally to $\mathfrak X$, that is: there is a covering family (consisting, without loss of generality, of a single element) $(U'\to U)$ and an object $x\in \mathfrak X_{U'}$ such that for all $V\to U'$, $f(x_V) \simeq y_V$ in $\mathfrak Y_V$.

\begin{definition} A \emph{Nash}, or \emph{smooth semi-algebraic} stack is a stack $\mathfrak X$ as above which admits a representable, epimorphic submersion from a Nash manifold: 
\begin{equation}\label{presentation}X\to \mathfrak X\end{equation}
(where $X$ is a Nash manifold). 
\end{definition}
We will call such an epimorphism a \emph{presentation} of $\mathfrak X$.  

We call a $1$-morphism: $\mathfrak X\to \mathfrak Z$ of Nash stacks \emph{representable} if for one, equivalently any, presentation $V\to \mathfrak Z$ the stack $\mathfrak X \times_{\mathfrak Z} V$ is representable. A smooth, \'etale etc.\ morphism of stacks is a representable morphism such that for every $U\to \mathfrak Z$ the morphism $\mathfrak X\times_{\mathfrak Z} U\to U$ has the said property.

\begin{lemma} \label{fiberprod}
Let $\mathfrak X\to \mathfrak Z$ be a smooth (submersive) $1$-morphism, and $\mathfrak Y\to \mathfrak Z$ any $1$-morphism of Nash stacks.

Then the fiber product $\mathfrak X\times_{\mathfrak Z} \mathfrak Y$ is a Nash stack.
\end{lemma}

\begin{proof}
Indeed, if $Y\to \mathfrak Y$ is a presentation of $\mathfrak Y$, then it is easy to see that $\mathfrak X\times_{\mathfrak Z} Y$, which by assumption is a Nash manifold, gives a presentation of $\mathfrak X\times_{\mathfrak Z} \mathfrak Y$.
\end{proof}

\subsection{Nash groupoids}

Assume that $[R_X\rightrightarrows X]$ is a groupoid object in $\mathfrak N$. This consists of a pair of smooth (submersive) morphisms between Nash manifolds (denoted $s$ and $t$ for ``source'' and ``target''), with a ``multiplication'': $R_X \times_{t, X, s} R_X \to R_X$, an identity section and an inverse map (all maps assumed to be Nash), satisfying axioms analogous to the usual group axioms.

There is a Nash stack $\mathfrak X$ associated to such a groupoid object, defined as follows: the fiber of $\mathfrak X$ over $U\in \ob(\mathfrak N)$ has as objects all \emph{$R_X$-torsors over $U$} in the sense of \cite[\S 12]{Noohi}, that is: smooth epimorphisms of Nash manifolds: $T\to U$, together with a Cartesian map of groupoids
$$[T\times_U T \rightrightarrows T] \to [R_X \rightrightarrows X].$$ 
I will usually denote such an object by the pair of maps $(T\to U, T\to X)$, and sometimes just by $T\to U$, or by $T$.

\begin{proposition}
Let $\mathfrak X$ be a Nash stack, and $X\to \mathfrak X$ a presentation. Set $R_X = X\times_{\mathfrak X} X$. Then $\mathfrak X$ is canonically (up to $2$-isomorphism) isomorphic to the Nash stack associated to the groupoid object $[R_X\rightrightarrows X]$.
\end{proposition}

\begin{proof}
The proof is identical to the one in the case of algebraic or differentiable stacks. I recall the construction of the $1$-morphisms:

On the one hand, we have an $1$-morphism from $\mathfrak X$ to $[R_X \rightrightarrows X]$ by assigning to every $u\in \mathfrak X_U$ the fiber product $T = U\times_{u,\mathfrak X} X$. Since $X\to \mathfrak X$ is an epimorphism, it follows that $T$ is an $R_X$-torsor over $U$.

Vice versa, even without the epimorphism assumption we have an $1$-morphism in the opposite direction: for $U\in \ob(\mathfrak N)$ and an $R_X$-torsor $T\to U$ choose, over an \'etale cover $U'\to U$, a section $U' \to T$. The composition with the map to $X$ and then to $\mathfrak X$ gives an object in $\mathfrak X_{U'}$, which by the stack axiom can be seen do descend to an object in $\mathfrak X_U$.
\end{proof}

\begin{remark}
In the complex case, $F=\CC$, there is no difference between Nash and smooth algebraic stacks of finite type over $\CC$. This follows, for example, from the description of a Nash stack in terms of groupoid objects and the fact that all Nash manifolds are algebraic, and all Nash maps are algebraic.
\end{remark}

The following closely related lemma will be useful later:

\begin{lemma}\label{epi}
Let $X'\to X$ be a smooth epimorphism of Nash manifolds, and let $[R_X\rightrightarrows X]$ be a groupoid object. Let $[R_{X'}\rightrightarrows X']$ be obtained by base change, that is: $R_{X'} = X'\times_{X,s}R_X\times_{t,X} X'$. Then the Nash stacks defined by the two groupoids are canonically equivalent.
\end{lemma}

\begin{proof}
Again, I will just describe the $1$-morphisms in terms of objects.

Given $U\in\ob(\mathfrak N)$ and an $R_X$-torsor $T\to U$, we obtain an $R_{X'}$-torsor $T'\to U$ by base change: $T' = T \times_X X'$.

Vice versa, given an $R_{X'}$-torsor $T'\to U$ assume, at first, that it admits a section
$$U \to T'.$$
Let $\alpha: U\to X'$ be the map obtained by composing with $T'\to X'$, and $\beta$ the map obtained by further composing with $X'\to X$. We let 
$$T = U\times_{\beta, X, s} R_X.$$
It is an $R_X$-torsor over $U$. I claim that, canonically,
\begin{equation}\label{Ttoprime} T' \simeq T\times_{X,s} R_{X'}. \end{equation}
Indeed, the $R_{X'}$-action and the section define an isomorphism
$$ U\times_{\alpha, X', s} R_{X'}\xrightarrow{\sim} T'.$$

Moreover, $R_{X'} = X' \times_{X,s} R_X \times_{t,X} X'$ because the same holds as schemes, and the claim now follows easily.

For the general case, the map $T'\to U$ always admits a section locally in the \'etale (or semi-algebraic, if $F\ne \CC$) topology, and the canonical isomorphisms \eqref{Ttoprime} (or even just the map $T'\to T$) suffice to glue them, because every local trivialization for $T'$ induces one for $T$.
\end{proof}

\subsection{From smooth algebraic to Nash stacks} \label{algtoNash}

Let $\mathcal X$ be a smooth algebraic stack over $F$. By this I will mean, throughout, that it is an algebraic stack over $F$ which admits a presentation as a smooth groupoid $[R_X \rightrightarrows X]$, where $X$ and $R_X= X\times_{\mathcal X} X$ are smooth schemes of finite type over $F$. I would like to attach to it a Nash stack $\mathfrak X$, also denoted by $\mathcal X(F)$. (The notation $\mathcal X(F)$ really stands for the isomorphism classes in the groupoid $\mathcal X_F$ of points $\spec F\to \mathfrak X$, but we may think of this set as having the richer structure of a Nash stack, just as we think of the set $X(F)$ as a Nash manifold, when $X$ is a smooth algebraic variety over $F$; thus, this notation should not cause confusion.) I do not quite achieve that in full generality, because I do not know if there is always a presentation of an algebraic stack over $F$ to which all $F$-points lift. This is true in many cases of interest, such as quotients of smooth varieties by linear algebraic groups, and in any case we can always give to $\mathcal X(F)$ the structure of a stack over $\mathfrak N$, whether there is such a presentation or not.

To do that, we need to define what is a (``semi-algebraic'') morphism from a Nash manifold $U$ to $\mathcal X$. 
I follow the standard approach of \cite[\S 20]{Noohi} that uses presentations (``charts'') of $\mathcal X$ as a quotient of a smooth groupoid. However, in contrast to \cite{Noohi}, we are not working over an algebraically closed field; as a consequence, once we have properly defined the stack $\mathfrak X$, for an epimorphism $X\to \mathcal X$ of algebraic stacks, where $X$ is a smooth variety over $F$, it does not necessarily induce an epimorphism $X(F)\to \mathfrak X$ of stacks over $\mathfrak N$. Thus, we cannot work with an arbitrary presentation $X\to \mathcal X$ and hope that every morphism $U\to \mathfrak X$ (where $U$ is a Nash manifold) will produce a surjective map: $(U\times_{\mathfrak X} X(F)) \to U$. 
 
For example, for $\mathcal X = \spec F/G$ with the standard presentation: $\spec F \xrightarrow{\pi} \spec F/G$, where $G$ is an algebraic group with $H^1(F, G)\ne 0$, we should clearly have a morphism of Nash stacks: $\pt \xrightarrow{\alpha} \mathfrak X$ induced from the algebraic morphism: $\spec F \xrightarrow{\alpha} \spec F/G$ corresponding to a non-trivial $G$-torsor; however, the fiber product
$$ \pt \times_{\alpha,\mathfrak X,\pi} \pt$$
is empty, because this is the case for the $F$-points of the algebraic fiber product
$$ \spec F \times_{\alpha,\mathfrak X,\pi} \spec F.$$

Thus, to define the Nash stack $\mathfrak X$ we will either have to choose a presentation which is ``surjective on $F$-points'', or take some kind of limit, over all possible presentations $X\to \mathcal X$ of the algebraic stack $\mathcal X$, of the Nash stacks associated to the groupoids $[(X\times_{\mathcal X} X)(F)\rightrightarrows X(F)]$. I follow the first approach, which is simpler although not very nice, because it relies on choosing a presentation (if a presentation which is surjective on $F$-points exists). Of course, it turns out that the resulting stack does not depend on the $F$-surjective presentation chosen. In Appendix \ref{app:Nash}, I will outline a more canonical approach which produces a stack over $\mathfrak N$ for \emph{every} smooth algebraic stack of finite type over $\mathfrak N$; this poses the interesting problem of determining, in general, when the resulting stack is Nash.

Hence, let $\mathcal X=$ a smooth algebraic stack over $F$, and assume  that it admits a presentation $X\to \mathcal X$ \emph{with the property that} $X(F) \to \mathcal X(F)$ (= the set of isomorphism classes of objects in $\mathcal X_F$) \emph{is surjective.} We will call such a presentation ``\emph{$F$-surjective}''.

\begin{lemma}\label{equivpres}
Let $X\to \mathcal X$, $X'\to \mathcal X$ be two algebraic presentations of $\mathcal X$ such that the induced maps 
\begin{equation}\label{pointmaps}X(F) \to \mathcal X(F),\,\,\,\, X'(F)\to \mathcal X(F)\end{equation}
 are surjective.

Let $R_X = X \times_{\mathcal X} X$, $R_{X'} = X' \times_{\mathcal X} X'$ (fiber products as algebraic stacks). Then the Nash groupoids
$$ [R_X(F) \rightrightarrows X(F)],\,\,\, [R_{X'}(F)\rightrightarrows X'(F)]$$
are canonically equivalent (up to 2-isomorphism).
\end{lemma}

\begin{proof}
The equivalence implied in the lemma is induced from the obvious morphisms from the third groupoid $[R_{X''}(F)\rightrightarrows X''(F)]$, where $X'' = X\times_{\mathcal X} X'$. Notice that since both maps \eqref{pointmaps} are surjective, the same holds for the corresponding map from $X''(F)$, and also for the maps from $X''(F)$ to $X(F), X'(F)$.

This reduces us to the case when we have smooth, $F$-surjective epimorphisms
$$ X' \to X \to \mathcal X, $$
in which case the claim is Lemma \ref{epi}.
\end{proof}

We can now define:

\begin{definition}
Given a smooth algebraic stack $\mathcal X$ of finite type over $F$, which admits an $F$-surjective presentation, the associated Nash stack $\mathfrak X = \mathcal X(F)$ is ``the'' Nash groupoid
$$[R_X(F) \rightrightarrows X(F)],$$
where $X\to \mathcal X$ is any $F$-surjective presentation. 
\end{definition}

The association $\mathcal X\mapsto \mathcal X(F)$ is a functor from the 2-category of smooth stacks of finite type over $F$ which admit $F$-surjective presentations to the category of Nash stacks over $F$. Indeed, as in the proof of Lemma \ref{epi}, given a morphism $\mathcal X \to \mathcal Y$ of two algebraic stacks, and $F$-surjective presentations $X\to \mathcal X$, $Y\to\mathcal Y$, a morphism $U\to \mathcal X(F)$ lifts \'etale-locally to $X(F)$, and then in turn to $(X\times_{\mathcal Y} Y)(F)$. As in the lemma, it can be seen that this gives rise to a well-defined morphism $U\to \mathcal Y(F)$. 

The functor preserves fiber products:

\begin{proposition}
Let $\mathcal X \to \mathcal Z \from \mathcal Y$ be morphisms of smooth algebraic stacks of finite type over $F$, all of which admit $F$-surjective presentations. Then we have
$$ (\mathcal X \times_{\mathcal Z} \mathcal Y)(F) \simeq \mathcal X(F) \times_{\mathcal Z(F)} \mathcal Y(F)$$
canonically, up to $2$-isomorphism.
\end{proposition}

\begin{proof}
The $1$-morphism from $(\mathcal X \times_{\mathcal Z} \mathcal Y)(F)$ to $\mathcal X(F) \times_{\mathcal Z(F)}  \mathcal Y(F)$ is clear.

Let us construct its quasi-inverse. First, we notice that the proposition is clearly true iff all three stacks are schemes. We will reduce the general case to that. Denote by $X\to\mathcal X$, $Y\to \mathcal Y$, $Z\to\mathcal Z$ three $F$-surjective presentations.

Recall that an object in $\mathcal X(F) \times_{\mathcal Z(F)} \times Y(F)$ over $U\in \ob(\mathfrak N)$ consists of morphisms $x:U\to \mathcal X(F)$, $y:U\to Y(F)$ together with an isomorphism of their compositions with the maps to $\mathcal Z(F)$.

By definition of the functor on $1$-morphisms, the composition $U\to \mathcal X(F)\to\mathcal Z(F)$ is obtained, \'etale-locally, by a map
$$\alpha: U\to (X\times_{\mathcal Z} Z)(F).$$
(We assume that all statements that hold \'etale-locally hold already over our Nash manifold $U$, in order to simplify notation.)
Similarly for the composition $U\to \mathcal X(F)\to\mathcal Z(F)$:
$$\beta: U\to (Y\times_{\mathcal Z} Z)(F).$$

Now, these maps depend on choices of sections, and it is not necessarily the case that composing with projection to $Z(F)$ the two maps will coincide. (Denote these compositions by $\bar\alpha$, $\bar\beta$.) However, the isomorphism of their compositions with the maps to $\mathcal Z(F)$ is an isomorphism between the $R_Z(F)$-torsors
$$ T:= U\times_{\bar\alpha, Z(F), s} R_Z(F) \simeq U\times_{\bar\beta, Z(F), s} R_Z(F),$$
where these isomorphisms are over $Z(F)$ with respect to the ``target'' map $t: R_Z(F)\to Z(F)$.

Therefore, the compositions of the two sequences of maps below coincide:

$$ T \xrightarrow{\alpha \times_{Z(F),s} R_Z(F)} (X\times_{\mathcal Z} Z)(F) \times_{Z(F),s} R_{Z(F)} = (X\times_{\mathcal Z} R_Z)(F) \xrightarrow{t}$$
$$\xrightarrow{t} (X\times_{\mathcal Z} Z)(F) \to Z(F),$$
$$T \xrightarrow{\beta \times_{Z(F),s} R_Z(F)} (Y\times_{\mathcal Z} Z)(F) \times_{Z(F),s} R_{Z(F)} = (Y\times_{\mathcal Z} R_Z)(F) \xrightarrow{t}$$
$$\xrightarrow{t} (Y\times_{\mathcal Z} Z)(F) \to Z(F).$$

Hence we get a map
$$ T\to (X\times_{\mathcal Z} Z)(F) \times_{Z(F)} (Y\times_{\mathcal Z} Z)(F) = (X\times_{\mathcal Z} Y \times_{\mathcal Z} Z)(F).$$

Taking a section, \'etale-locally, of the map $T\to U$, and composing with the map $(X\times_{\mathcal Z} Y \times_{\mathcal Z} Z)(F)\to (X\times_{\mathcal Z} Y)(F)$, 
we obtain that the map $U\to \mathcal X(F) \times_{\mathcal Z(F)} \mathcal Y(F)$ factors \'etale-locally through $(X\times_{\mathcal Z} Y)(F)$, and hence factors through $(\mathcal X\times_{\mathcal Z}\mathcal Y)(F)$.
\end{proof}

\subsection{Quotients by linear groups} \label{Nash-quotient}

Consider the case of an algebraic quotient stack $\mathcal X=X/G$, where $X$ is a smooth variety over a local field $F$ and $G$ is a (smooth) linear algebraic group acting on it.

The Nash stack associated to the Nash groupoid
$$[X(F)\times G(F)\rightrightarrows X(F)]$$
will be denoted by $X(F)/G(F)$.

Let $T$ be a $G$-torsor, and set\footnote{Throughout the paper we will be using the same notation for a right $G$-torsor and the corresponding left $G$-torsor obtained by inverting the $G$-action; for example, $T\times^G T$ is canonically the trivial $G$-torsor.} $X^T=X\times^G T$, $G^T = \Aut^G(T)$ (an inner form of $G$). Then $X^T$ carries a $G^T$-action, and we have a canonical equivalence of algebraic stacks:
\begin{equation}\label{isom-stacks} X^T/G^T \simeq X/G.\end{equation}

\begin{proposition}\label{groupquotient}
The algebraic stack $\mathcal X$ admits an $F$-surjective presentation and hence $\mathfrak X=\mathcal X(F)$ is a Nash stack. It is canonically equivalent to the Nash stack $X_\rho(F)/\GL_N(F)$, for any faithful representation $\rho:G\hookrightarrow \GL_N$ (where $N\in \mathbb Z_{>0}$) and $X_\rho= X\times^G \GL_N$.

Equivalently, via \eqref{isom-stacks} it is canonically equivalent with the disjoint union
\begin{equation}
\bigsqcup_T X^T(F)/G^T(F),
\end{equation}
where $T$ runs over a set of representatives for the isomorphism classes of $G$-torsors over $F$. 
\end{proposition}

Notice that $X^T(F)$ could be empty.

\begin{remark}
In characteristic zero, the set $H^1(F,G)$ classifying $G$-torsors is finite, thus the above disjoint union is finite and, hence, a Nash manifold. In positive characteristic, any linear group can be embedded in a group with finite (or even trivial) cohomology, but since we are working with $F$-analytic manifolds in this case, we do not need the above disjoint union to be finite.
\end{remark}

\begin{proof}
If $\rho:G\hookrightarrow \GL_N$ is a faithful representation and $X_\rho$ is as in the statement of the proposition, we have a canonical equivalence of algebraic stacks:
\begin{equation}\label{isomorphism}
\mathcal X = X/G = X_\rho/\GL_N.
\end{equation}
Because $\GL_N$ has trivial Galois cohomology (Hilbert 90), the set $X_\rho(F)$ surjects onto the set of isomorphism classes of $F$-points of $\mathcal X$.  
Therefore, $\mathfrak X$ is a well-defined Nash stack.

Alternatively, we observe that the presentation $\tilde X:=\bigsqcup_T X^T \to \mathcal X$ is surjective on $F$-points, and hence we get that $\mathfrak X$ is equivalent to the stack defined by the Nash groupoid $[R_{\tilde X}(F)\rightrightarrows \tilde X(F)]$, where $R_{\tilde X} = \tilde X \times_{\mathcal X}\tilde X$. For non-isomorphic torsors $T_1, T_2$, the images of $X^{T_1}(F)$ and $X^{T_2}(F)$ in $\mathcal X(F)$ are disjoint (consist of different isomorphism classes), and therefore we have
$$R_{\tilde X}(F) = \bigsqcup_T R_{X^T}(F),$$
where $R_{X^T} = X^T \times_{\mathcal X} X^T$.
\end{proof}

\section{Local Schwartz spaces} \label{sec:localSchwartz}

\subsection{Schwartz measures on Nash manifolds}

In what follows we will need to do some homological algebra in the quasi-abelian category of nuclear Fr\'echet spaces; I point the reader to \cite{Schneiders, Buehler, FS} for presentations of the topic. The basic notions have been summarized in Appendix \ref{app:homology}. We recall that a complex of nuclear Fr\'echet spaces is \emph{strictly exact} if the \emph{set-theoretic} images and kernels coincide; in particular, such a complex is made up of \emph{strict} morphisms, i.e., morphisms with closed image. Derived categories of a quasi-abelian category are well-defined, by localizing the homotopy category of complexes with respect to the null system of strictly exact complexes. Two complexes are said to be (strictly) \emph{quasi-isomorphic} if they have the same image in the derived category; equivalently, if they are dominated by a third one by maps whose mapping cones are strictly exact.
In the non-Archimedean case, of course, the corresponding spaces are simply vector spaces without topology, and any mention of Fr\'echet spaces should be ignored.\footnote{Unless one wants to work with \emph{almost smooth} functions, cf.\ \cite[Appendix A]{SaBE1}.}

The space of Schwartz functions on Nash manifolds is well-known, and has been studied as a cosheaf on the restricted topology of semi-algebraic open sets in \cite{DuCloux, AGSchwartz}. Again, in Appendix \ref{app:homology}, I recall the basic notions that we will be using about cosheaves. In the Archimedean case, the sections of the Schwartz cosheaf over an affine open semi-algebraic set $U$ are those functions $f$ on $U$ such that $Df$ is bounded for every Nash differential operator on $U$. For the notions of Nash differential operators, densities, measures etc.\ cf.\ \cite[\S 3.4, 3.5 and A.1]{AGSchwartz}. In the non-Archimedean case, Schwartz functions are simply locally constant, compactly supported functions. 
By \cite[Theorem 5.4.3]{AGSchwartz}, this cosheaf is (strictly) \emph{flabby}: extension maps are strict and injective. Hence (cf.\ Lemma \ref{flabby}), it is acyclic for the semi-algebraic topology.

\begin{remark} \label{restriction-manifolds} For a finite extension $F'/F$ of local fields, and a Nash $F'$-manifold $X$, the space of Schwartz functions on $X$ does not change if we consider $X$ as a Nash $F$-manifold, and more generally its cosheaf of Schwartz functions for the $F'$-semi-algebraic topology is simply the restriction of the corresponding cosheaf for the $F$-semi-algebraic topology. This, together with an acyclicity result that we will prove (Proposition \ref{Poincare}), will render it indifferent for Schwartz spaces whether a Nash $F'$-manifold, or stack, is considered as an $F'$-manifold/stack or an $F$-manifold/stack.
\end{remark}

Functions do not have well-defined push-forwards, and therefore we will instead be working with spaces of \emph{Schwartz measures}: these are products of Schwartz functions by Nash measures. In the non-Archimedean case, Schwartz functions are by definition locally constant and compactly supported, and hence so are Schwartz measures. We will be denoting the space of Schwartz measures on a Nash manifold $X$ by $\mathcal S(X)$.
I will now show that Schwartz measures actually form an acyclic cosheaf in the \emph{smooth} topology over a Nash manifold. 

Let $\mathfrak N_\infty$ be the category whose objects are Nash manifolds (that is, the objects of $\mathfrak N$) and whose morphisms are only those morphisms of Nash manifolds which are smooth (=submersive). Consider $\mathfrak N_\infty$ as a site, with the Grothendieck topology generated by smooth covers, that is collections of smooth morphisms $\{p_i:U_i\to X\}$ such that $\bigsqcup_i p_i: \bigsqcup_i U_i \to X$ is surjective.

\begin{proposition}\label{cosheaf-smooth}
The spaces of Schwartz measures form a cosheaf of nuclear Fr\'echet spaces on $\mathfrak N_\infty$, with strict push-forward maps. That is, every smooth morphism $\pi:X\to Y$ of Nash manifolds induces, functorially, a strict morphism: $\pi_!:\mathcal S(X)\to \mathcal S(Y)$, and if $\pi$ is surjective, it induces a coequalizer diagram:
\begin{equation}\label{sheaf-smooth} \mathcal S(X\times_Y X) \overset{s_!}{\underset{t_!}\rightrightarrows} \mathcal S(X) \xrightarrow{\pi_!} \mathcal S(Y).\end{equation}
\end{proposition}

The proof will rely on the following:

\begin{proposition}
Let $F\ne \CC$. Let $f:X\to Y$ be a smooth morphism of Nash $F$-manifolds. We can cover $X$ by a finite number of open subsets $U_i$ such that for each $i$ the restriction of $f$ to $U_i$ factors through an open embedding
$$ e_i: U_i \hookrightarrow Y\times \mathbb A^n(F)$$
and the projection to $Y$.

Moreover, we may choose the data so that the image of $e_i$ contains the zero section: $f(U_i)\times\{0\}$; in particular, there are smooth local sections of $f$: $f(U_i)\to U_i$.

Finally, choosing the data this way, there is a(n $F$-valued) Nash function $T_i$ on $f(U_i)$, for every $i$, with the property that the tubular neighborhood
\begin{equation}\label{tubularnbhd} B_{T_i}:= \{(y,v)\in f(U_i)\times F^n\,\, | \,\, |T_i(y)|\cdot \Vert v\Vert \le 1\},
\end{equation} 
where $\Vert\bullet\Vert$ is some fixed norm on $F^n$, is contained in $e_i(U_i)$.
\end{proposition}

Notice that the first statement is the analog of the \'etale factorization of smooth morphisms in algebraic geometry, though here we can replace \'etale by an open embedding, because (as a special case of this proposition) every \'etale map admits semi-algebraic local sections (when $F\ne \mathbb C$). We notice that the case $F=\CC$ is excluded, because the semi-algebraic topology is not sufficiently fine in this case; most importantly, sets of the form $|x|<c$ in $\mathbb A^1(\CC)$ are not semi-algebraic; however, we will be able to use this proposition by restriction of scalars from $\CC$ to $\RR$, when we need it.

\begin{proof}
Starting from the last statement, the existence of such tubular neighborhoods is \cite[Theorem 3.6.2]{AGSchwartz} over $\mathbb R$. 

Using this, Aizenbud and Gourevitch prove the existence of local sections in \cite[2.4.3]{AGdeRham}. In the course of the proof of this result, namely in the proof of \cite[Proposition A.0.4]{AGdeRham}, they show the existence of local open embeddings $U_i \hookrightarrow Y\times \mathbb A^n(F)$.

The proofs in the non-Archimedean case are similar, but easier.
\end{proof}

\begin{proof}[Proof of Proposition \ref{cosheaf-smooth}]
We may, and will, assume that $F\ne \CC$, by considering any complex Nash manifolds $X$ and $Y$ as in the statement of the proposition as Nash $\RR$-manifolds.

Since it is known that Schwartz spaces form a co\-sheaf (clearly with strict push-forward maps) for the restricted topology of semi-algebraic open covers,  we may replace $X$ by a cover $U_i$ such as in the previous proposition, and $Y$ by the image of $U_i$. Thus, we may assume that
$$Y\times\{0\} \subset X\subset Y\times \mathbb A^n(F),$$
with the map $f$ being the standard projection to $Y$. We also fix a Nash function as in the proposition (here denoted $T$), such that $X\supset B_T$. We will denote the set $B_T$ by $X_T$, in order to remember that it is a subset of $X$. 
We may assume that $T$ is bounded away from zero, so that its absolute value is a smooth function on $Y$.

For a direct product of Nash manifolds $Z\times U$ it is known that $\mathcal S(Z\times U) = \mathcal S(Z)\hat\otimes\mathcal S(U)$. Thus, the push-forward of measures in $\mathcal S(Z\times U)$ to $Z$ clearly lands in $\mathcal S(Z)$, and the map $\mathcal S(Z\times U)\to \mathcal S(Z)$ is a strict epimorphism. In our setting, since $\mathcal S(X)\subset \mathcal S(Y\times\mathbb A^n(F))$, we get that push-forward of Schwartz measures on $X$ is a continuous map into Schwartz measures on $Y$. On the other hand, the Nash function $T$ essentially allows us to treat a subset of $X$ as a product space. Namely, we can fix a Schwartz measure $\mu$ on the open unit ball in $F^n$ with total mass $1$, and then notice that for a given $f\in \mathcal S(Y)$ the measure
$$ (y,a) \mapsto f(y) \mu(T(y) a) \in \mathcal S(Y\times \mathbb A^n(F))$$
actually belongs to $\mathcal S(X_T)$ and its push-forward to $Y$ is $f$. Thus, the map $\mathcal S(X)\to \mathcal S(Y)$ is a strict epimorphism.

Now consider the diagram \eqref{sheaf-smooth}. We have $\pi_!\circ s_!= \pi_! \circ t_!$ since the push-forward of measures corresponding to a commutative diagram of spaces is commutative. We will show that any measure $f$ in the kernel of $\mathcal S(X)\to \mathcal S(Y)$ is of the form $(s_!-t_!) F$, for $F\in \mathcal S(X\times_Y X)$. 

It is known \cite[Theorem A.1.1]{AGKl} that a Schwartz function can be written as a product of two Schwartz functions.  
\red{In fact, there is a positive Schwartz function $\phi$ on $Y$ and an $f_0\in \mathcal S(X)$ such that $f = \pi^*\phi \cdot f_0$.} Notice that $\pi_! f_0$ is also zero.

Choose a nowhere vanishing, positive smooth measure of polynomial growth $\mu$ on $Y$, and let $\mathcal S'(X)$ be ``the quotient of elements of $\mathcal S(X)$ by $\mu$'', by which I just mean that push-forwards of elements of $\mathcal S'(X)$ to $Y$ will be thought of as Schwartz functions by dividing by $\mu$. (There is a canonical way to think of $\mathcal S'(X)$ as some space of densities on $X$ which are ``measures in the vertical direction and functions in the horizontal''.)

Choose $h\in \mathcal S'(X)$ with $\pi_! h = \phi$ as Schwartz functions on $Y$. 
The product $F(x_1,x_2)=f_0(x_1) h(x_2)$ is a well-defined Schwartz measure on $X\times_Y X$, and we have
$$(s_!-t_!) F (x) = f_0(x) \pi_! h(\pi(x)) - \pi_!f_0(\pi(x)) h(\pi(x)).$$

We have $\pi_!f_0 = 0$ and $\pi_! h = \phi$, therefore we get:
$$(s_!-t_!) F (x) = f_0(x) \phi(\pi(x)) = f(x).$$
This proves the cosheaf condition.
\end{proof}

Consider, for every $X\in\ob(\mathfrak N)$, the smooth site $X_\infty$ over $X$, whose objects are all smooth morphisms $Y\to X$ from other Nash manifolds, a morphism from $(Y\to X)$ to $(Z\to X)$ is a smooth map $Y\to Z$ which commutes with the maps to $X$, and covering families are \emph{finite} families $\{U_i\to X\}_i$ which are jointly surjective. The site $\mathfrak N_\infty$ is therefore the smooth site of $X=$ a point. 

The cosheaf of Schwartz functions on $\mathfrak N_\infty$ induces, by restriction, a cosheaf on $X_\infty$. We will be denoting this cosheaf by $\mathfrak S_X$, to distinguish it from the space $\mathcal S(X)$ of its global sections. Since $\mathfrak S_X(U)=\mathcal S(U)$ for every smooth $U\to X$, i.e., this space depends only on $U$ and not on $X$, we may also drop the index $X$ and denote this cosheaf by $\mathfrak S$, without creating any confusion.

 The following is the ``Poincar\'e lemma'' stating that this cosheaf is (strictly) acyclic over any Nash manifold.

\begin{proposition}\label{Poincare}
For every smooth morphism $\pi:X\to Y$ of Nash manifolds, denote by $[X]_Y^n$ the $n$-fold fiber product $X\times_Y X\times\cdots \times_Y X$ and let $\partial_n: \mathcal S([X]^{n+1}_Y) \to \mathcal S([X]^n_Y)$ be the alternating sum of push-forwards:
$$ f\mapsto \sum_{i=0}^n (-1)^i\pi_{i,!} f,$$
where $\pi_i: [X]^{n+1}_Y\to [X]^n_Y$ is the omission of the $i$-th variable (counting from zero).  

For any smooth morphism $X\to Y$ the sequence
$$ \cdots \xrightarrow{\partial_n}\mathcal S([X]^n_Y) \xrightarrow{\partial_{n-1}} \cdots \mathcal S(X\times_Y X) \xrightarrow{\partial_1} \mathcal S(X)\to 0$$
is strictly exact.
\end{proposition}

\begin{proof}
The proof is very similar to that of Proposition \ref{cosheaf-smooth}: 
Let $f\in \mathcal S([X]^n_Y)$ be in the kernel of $\partial_{n-1}$ and write it as
$$ f(\xx) = \phi (\pi(\xx)) f_0 (\xx)$$
($\xx \in [X]^n$)
where $\phi\in\mathcal S(Y)$ and $f_0$ is also in the kernel of $\partial_{n-1}$.  (We denote by $\pi$ the map to $Y$, for any fiber power of $X$ over $Y$.)

Choose $h\in \mathcal S(X)$ with $\pi_! h = \phi$.

Let $F\in \mathcal S([X]^{n+1}_Y)$ be the product:
$$ F(x_0, \xx) = h(x_0) \otimes f_0(\xx), \,\, (x_0,\xx)\in [X]^{n+1}_Y.$$

Then $\partial_n F (\xx) = \pi_! h(\pi(\xx)) f_0(\xx) - h\otimes \partial_{n-1}f_0(\xx) = \phi(\pi(\xx)) f_0(\xx) - 0 = f(\xx)$.
\end{proof}

As we will see, acyclicity fails when we replace Nash manifolds by general Nash stacks.

\subsection{The Schwartz cosheaf on a Nash stack}

Let $\mathfrak X$ be a Nash stack. The \emph{smooth site} $\mathfrak X_\infty$ on $\mathfrak X$ has as objects those objects $(U,u)\in \ob(\mathfrak X)$ which represent \emph{smooth} $1$-morphisms of Nash stacks $u: U\to \mathfrak X$. A morphism from $(U,u)$ to $(V,v)$ will be a morphism: $u\to v$ in $\mathfrak X$ which is smooth; this is equivalent to requiring that the induced morphism: $U\to V$ be smooth (since $u\to v$ is automatically a base change of the latter).

This category is endowed with the coverage whose elements over $(U,u)$ are \emph{finite} collections of smooth morphisms $\{\phi_i: u_i \to u\}_i$ such that the induced morphism of Nash manifolds
$$ \bigsqcup_i \phi_i : \bigsqcup_i U_i \to U$$
is surjective. 

The assignment $(U,u)\mapsto \mathfrak S_{\mathfrak X}((U,u)):= \mathcal S(U)$ is a pre-cosheaf on this site, i.e., a functor that associates to every morphism $\pi: (U,u)\to (V,v)$ the push-forward map $\pi_!:\mathcal S(U)\to \mathcal S(V)$. 

\begin{proposition}
The pre-cosheaf $\mathfrak S_{\mathfrak X}$ is a cosheaf on $\mathfrak X_\infty$. Moreover, for every (representable) smooth $1$-morphism of Nash stacks $\phi:\mathfrak X\to \mathfrak Z$ there is a canonical identification of cosheaves:
\begin{equation}\label{coshident}
\mathfrak S_{\mathfrak X} \simeq \phi^{-1} \mathfrak S_{\mathfrak Z}.
\end{equation} 
\end{proposition}

The last statement allows us to denote the cosheaf $\mathfrak S_{\mathfrak X}$ simply by $\mathfrak S$.

\begin{proof}
The cosheaf property follows from the cosheaf property that has already been established for Nash manifolds, Proposition \ref{cosheaf-smooth}.

A smooth $1$-morphism of Nash stacks is, in particular, a functor between categories fibered over $\mathfrak N$, which preserves the classes of smooth objects. Hence, any $(U,u)\in \ob(\mathfrak X_\infty)$ induces $(U,u')\in\ob(\mathfrak Z_\infty)$.

The identification \eqref{coshident} is simply the identification: $\mathfrak S_{\mathfrak X} ((U,u) ) = \mathcal S(U) = \mathfrak S_{\mathfrak Z}((U,u'))$.
\end{proof}

The fact that $\mathfrak S((U,u))$ is equal to the Schwartz space of the Nash manifold $U$ does not mean that there is no new information contained in the stack. Indeed, we have a ``copy'' of $\mathcal S(U)$ for every $u:U\to\mathfrak X$, and the copies of $u, u'$ are identified with each other only when there is a morphism $u\to u'$; this includes the automorphisms of $u$, which induce the identity on $\mathcal S(U)$. 

\subsection{Resolutions}

Consider now a presentation of a Nash stack
$$ X \twoheadrightarrow \mathfrak X.$$

It induces a simplicial object in Nash manifolds:
$$[X]_{\mathfrak X}^\bullet: \cdots \overset{\vdots}{\underset{\vdots}\rightrightarrows} [X]^n_{\mathfrak X} \overset{\vdots}{\underset{\vdots}\rightrightarrows}  \cdots X\times_{\mathfrak X} X \rightrightarrows X,$$
where $[X]^n_{\mathfrak X}$ is the fiber product of $n$ copies of $X$ over $\mathfrak X$.

By taking Schwartz spaces, this induces a simplicial object in the category of nuclear Fr\'echet spaces (with strict morphisms) and the corresponding long sequence (as in Proposition \ref{Poincare}):
\begin{equation}\label{resolution}  \cdots
 \xrightarrow{\partial_n}\mathcal S([X]^n_{\mathfrak X}) \xrightarrow{\partial_{n-1}}  \cdots \mathcal S(X\times_{\mathfrak X} X) \xrightarrow{\partial_1} \mathcal S(X) \xrightarrow{\partial_0} 0.
\end{equation}

\begin{theorem} \label{independent}
For any two presentations of $\mathfrak X$, the complexes \eqref{resolution} are canonically strictly quasi-isomorphic. In particular, we get a well-defined object $\mathcal S_\bullet(\mathfrak X)$ (up to unique isomorphism) of the derived category $\mathcal D^-(\mathscr F)$ of nuclear Fr\'echet spaces (vector spaces without topology, in the non-Archimedean case) concentrated in non-positive degrees. The association $\mathfrak X\mapsto \mathcal S_\bullet(\mathfrak X)$ is functorial with respect to smooth $1$-morphisms of Nash stacks. 
\end{theorem}

The quotient $\mathcal S_0(\mathfrak X)/\overline{\partial_1 \mathcal S_1(\mathfrak X)}$ will simply be denoted by $\mathcal S(\mathfrak X)$ and called ``the Schwartz space of $\mathfrak X$'' or the space of ``global sections'' of its Schwartz cosheaf. Notice that this is simply a coarse version of the ``zeroth homology'' of the complex $\mathcal S_\bullet(\mathfrak X)$ because, at least a priori, the image of $\mathcal S_1(\mathfrak X)$ by $\partial_1$ may not be closed.

\begin{proof}
The quasi-isomorphism for any two presentations follows formally from the acyclicity of Schwartz spaces for Nash manifolds, Proposition \ref{Poincare}, by the usual argument for \v{C}ech cohomology: For any two presentations $X\to \mathfrak X$, $X'\to \mathfrak X$, we form the double complex $K^{ij}=\mathcal S([X]^i_{\mathfrak X} \times_{\mathfrak X} [X']^j_{\mathfrak X})$. Notice that $[X]^i_{\mathfrak X} \times_{\mathfrak X} [X']^j_{\mathfrak X}$ is the $i$-fold fiber product of $X\times_{\mathfrak X} [X']^j_{\mathfrak X}$ over $[X']^j_{\mathfrak X}$. Thus, the rows and columns of the double complex (except for the zeroth row and zeroth column) are strictly exact by Proposition \ref{Poincare}, and by the standard argument we get a quasi-isomorphism between the zeroth row and the zeroth column.

Smooth $1$-morphisms $\mathfrak X\to\mathfrak Z$ preserve the corresponding smooth sites, and hence functoriality is clear.
\end{proof}

\begin{remark}\label{restriction-stacks}
It immediately follows from the definitions and Remark \ref{restriction-manifolds} that for a finite extension $F'/F$ of local fields and a Nash stack $\mathfrak X$ over $F'$, the cosheaf $\mathfrak S_{\mathfrak X}$ is simply the restriction to all smooth morphisms $X\to \mathfrak X$ of Nash $F'$-manifolds of the corresponding cosheaf when $\mathfrak X$ is considered as a Nash stack defined over $F$. 

Moreover, by applying Theorem \ref{independent} to a presentation $X\to \mathfrak X$ of Nash $F'$-manifolds, we see that the element $\mathcal S_\bullet(\mathfrak X)$ in the derived category and the space $\mathcal S(\mathfrak X)$ are independent of whether we consider $\mathfrak X$ as a Nash $F'$-manifold or a Nash $F$-manifold.
\end{remark}

\subsection{The Schwartz space of a quotient stack} \label{Schwartz-quotient}

Now let $\mathfrak X$ be the quotient of a Nash manifold $M$ by a Nash group $H$, i.e., the Nash stack associated to the groupoid object $[M\times H\rightrightarrows M]$ in $\mathfrak N$.  

In the Archimedean case, we consider the category $\mathcal M_H$ of smooth $F$-representations of $H$, in the language of Bernstein and Kr\"otz cf.\ \cite{BK}. (They are called smooth representations of moderate growth elsewhere.) This means that the representation is a countable inverse limit of Banach space representations, and (topologically) equal to the space of its smooth vectors. Again, in the non-Archimedean case we just consider smooth representations, without topology; since this case is easier, we discuss the Archimedean case here, and the adaptation to the non-Archimedean case is obvious. 

The category of smooth $F$-representations of $H$ is equivalent, by \cite[Proposition 2.20]{BK}, to the category of non-degenerate continuous representations of the (nuclear Fr\'echet) algebra $\mathcal S(H)$ of Schwartz measures on $H$ on Fr\'echet spaces. (\emph{Non-degenerate} means that $\mathcal S(H)V = V$, where $V$ is the space of the representation.) For any such module $V$, the action of $\mathcal S(H)$ gives rise to a topological quotient map: $\mathcal S(H)\hat\otimes V\to V$, where $\hat\otimes$ denotes the projective tensor product.

The functor of \emph{coinvariants} $V\mapsto V_H$ is the functor from $\mathcal M_H$ to $\mathscr F=$ Fr\'echet spaces that takes $V$ to its quotient by the \emph{closure} of the span of vectors of the form $v-g\cdot v$, $v\in V, g\in H$. Because of the closure operation, and the lack of exactness, this is a very rough version of the operation of ``modding out by the $H$-action'' (like the ``Schwartz space'' $\mathcal S(\mathfrak X)$ that we assigned to a Nash stack was a very rough version of the notion of ``global sections''). A finer version is obtained by the homological algebra of ``total derived functors'' of Deligne, explained in \cite[\S 10.6]{Buehler}. To apply this theory to $\mathcal S(H)$-modules, we need to adapt some results from \cite{Taylor}, which is written for unital algebras, to the category of non-degenerate modules of a non-unital algebra. I do this in Appendix \ref{app:homology}, and present here only the result:

An object $V\in \Ob(\mathcal M_H)$ is quasi-isomorphic (for the appropriate exact structure, explained in the appendix) to the complex
\begin{equation}\label{resolution1} \dots \to \mathcal S(H^n)\hat\otimes V \to \dots \to \mathcal S(H)\hat\otimes V \to 0.\end{equation}
Notice that $\mathcal S(H^{n+1}) = \mathcal S(H)\hat\otimes\mathcal S(H^n)$, and the action of $\mathcal S(H)$ is on the copy of $\mathcal S(H)$ on the left, by definition. The boundary maps are 
$$f_0\otimes \dots \otimes f_n\otimes v \mapsto  \sum_{i=0}^{n-1} (-1)^i f_0\otimes \dots\otimes f_i f_{i+1} \otimes \dots \otimes v + (-1)^n f_0 \otimes \dots \otimes f_{n-1} \otimes f_n v,$$
where we denote convolution of measures simply as multiplication.

The corresponding ``total derived functor'':
$$L_H:\mathcal D^-(\mathcal M_H)\to \mathcal D^-(\mathscr F),$$
where $\mathcal D^-$ denotes the derived categories of complexes which are strictly exact in sufficiently large (positive) degrees, is obtained by applying the coinvariant functor to \eqref{resolution1}. The functor of coinvariants is recovered as the ``rough'' version of the zeroth homology of the resulting complex, i.e., as the quotient of $(\mathcal S(H)\hat\otimes V)_H$ by the \emph{closure} of the image of $(\mathcal S(H^2)\hat\otimes V)_H$.

Now we return to the Nash stack $\mathfrak X$ associated to the groupoid object $[M\times H\rightrightarrows M]$.

\begin{proposition}\label{groupquotient-Schwartz}
The ``complex'' $\mathcal S_\bullet(\mathfrak X)$ is canonically strictly quasi-isomorphic to the derived coinvariant complex of the $H$-module $\mathcal S(M)$:
$$\mathcal S_\bullet(\mathfrak X) \simeq L_H\left(\mathcal S(M)\right).$$
In particular, the Schwartz space $\mathcal S(\mathfrak X)$ is canonically isomorphic to the coinvariant space $\mathcal S(M)_H$.
\end{proposition}

\begin{proof}
We have a presentation: $M\to \mathfrak X$. It is easy to see that $[M]^{n+1}_{\mathfrak X} = M\times H^n$ with the $i$-th map ($i=0,\dots,n$) to $[M]_{\mathfrak X}^n$ being
$$ (x, g_n, \dots, g_1)\mapsto (x, g_n \dots, g_2) $$
when $i=0$, 
$$ (x, g_n, \dots, g_{i+1}, g_i, \dots, g_1)\mapsto (x, \dots, g_{i+1} g_i, \dots, g_1) $$
when $0< i <n$, and 
$$ (x, g_n, \dots, g_1)\mapsto (xg_n, g_{n-1}, \dots, g_1) $$
when $i=n$.

On the other hand, setting $R = M\times H$ we can identify $[R]^{n+1}_M$ with $M\times H^{n+1}$, the copies of $H$ numbered in a decreasing order from $n$ to $0$, in such a way that the $i$-th map to $\mathcal S([R]^n_M)$ is forgetting the $i$-th copy of $H$. By Proposition \ref{Poincare}, the sequence resulting from taking alternating sums of these maps:
\begin{equation}\label{resolution2} \cdots\to \mathcal S(M\times H^{n+1}) \to \mathcal S(M \times H^n) \to \cdots \to \mathcal S(M\times H) \to 0\end{equation}
is quasi-isomorphic to $0\to\mathcal S(M)\to 0$. Moreover, it is $H$-equivariantly so, when $H$ acts diagonally on the right on all coordinates of $M \times H^n$. 

By a change of variables, we can easily see that \eqref{resolution2} coincides with \eqref{resolution1} for $V = \mathcal S(M)$. Applying the functor of $H$-coinvariants we get the a complex in the strict quasi-isomorphism class of $L_H\left(\mathcal S(M)\right)$.

On the other hand, the push-forward under the map 
$$[R]^{n+1}_M = M \times H^{n+1}  \to  M \times H^n = [M]^n_{\mathfrak X}$$ given by $(x, g_n,\dots,g_0) \mapsto (x g_n^{-1}, g_n g_{n-1}^{-1}, \dots, g_1 g_0^{-1})$ identifies $\mathcal S( [M]^n_{\mathfrak X})$ as the coinvariant space $\mathcal S([R]^{n+1}_M)_{H}$, and the boundary maps for the former descend to boundary maps for the latter. This proves the proposition. 

\end{proof}

\begin{corollary}\label{directsum}
Let $\mathcal X=X/G$, where $X$ is a smooth variety under the action of a linear algebraic group $G$. Let $\mathfrak X = \mathcal X(F)$ be the Nash stack associated to it. Then $\mathcal S_\bullet(\mathfrak X)$ is canonically strictly quasi-isomorphic to the direct sum
\begin{equation}\label{directsumeq} \bigoplus_T L_{G^T(F)} \mathcal S(X^T(F)),\end{equation}
where $T$ runs over a set of representatives for all isomorphism classes of $G$-torsors over $F$, $G^T = \Aut^G(T)$ and $X^T = X\times^G T$.
\end{corollary}

\begin{proof}
Immediate corollary of Proposition \ref{groupquotient} and Proposition \ref{groupquotient-Schwartz}.
\end{proof}

This implies, of course, a corresponding decomposition for the Schwartz space $\mathcal S(\mathfrak X)$.

\begin{example}
Let $\mathcal X$ be the quotient of the space $X=G$ by the adjoint $G$-action, where $G$ is an algebraic group. We have
$$ \mathcal S(\mathcal X(F)) = \bigoplus_{T\in H^1(F,G)} \mathcal S(G^T(F))_{G^T(F)},$$
by remembering that the first cohomology set $H^1(F,G)$ parametrizes isomorphism classes of $G$-torsors over $F$. The coinvariants here are taken with respect to the adjoint action.

If $G$ is quasi-split, the local Langlands conjectures in the form described by Vogan \cite{Vogan} describe local $L$-packets as unions of packets over all pure inner forms of $G$, i.e., over all elements of $H^1(F,G)$.

\end{example}

\begin{example}
Consider the algebraic quotient stack $\mathcal X= G\backslash H/G$, where $G=\SO((V, q))$ the special orthogonal group of a non-degenerate quadratic space $(V,q)$, and $H = \SO((V,q)) \times \SO((V\oplus \Ga, q\oplus 1))$, where by `$1$' we denote some non-degenerate quadratic form on the one-dimensional vector space $\Ga$, and $G$ is embedded ``diagonally'' in $H$. Equivalently, $\mathcal X$ is the quotient of $\SO((V\oplus \Ga, q\oplus 1))$ by $\SO((V, q))$-conjugacy. This is the quotient stack associated to the Gross--Prasad conjectures \cite{GP}.

In this case, isomorphism classes of $G$-torsors $T$ are parametrized by isomorphism classes of quadratic spaces $(V',q')$ of the same dimension and determinant as $V$, and viewing $\mathcal X$ as the quotient of $X=G\backslash H$ by $G$, for such a $G$-torsor we have $G^T = \SO((V', q'))$ and $X^T = G^T\backslash H^T$, where $H^T = \SO((V',q')) \times \SO((V'\oplus \Ga, q'\oplus 1))$. Thus, \eqref{directsumeq} gives a sum of coinvariant complexes parametrized by such classes of torsors. 

Notice that we can also present $\mathcal X$ in terms of the group $G\times G$ acting on the space $H$. In this case, we would apparently get more summands in \eqref{directsumeq} parametrized by pairs $T_1\times T_2$ of $G$-torsors; nonetheless, only when $T_1=T_2$ would the set of $F$-points of the space be non-empty.
\end{example}

\subsection{Stalks} \label{ssstalks}

For restricted topological spaces, I use the notion of \emph{stalks} for cosheaves as in \cite[\S B.4]{SaBE1}, that is: the stalk of a cosheaf $\mathcal F$ of vector spaces or Fr\'echet spaces on a restricted topological space $X$ over a closed subset $Z$ is the cosheaf over $X$ which to an open $U\subset X$ assigns the quotient of $\mathcal F(U)$ by the closure of the image of $\mathcal F(U\smallsetminus Z)$. (In \cite{SaBE1} it was assumed that the image of $\mathcal F(U\smallsetminus Z)$ is closed, and this will be the case for the sheaves considered here, even over the smooth site of Nash manifolds.)
Clearly, this cosheaf is \emph{supported} on $Z$, in the sense that for $U\subset X\smallsetminus Z$, $\mathcal F_Z(U)=0$. The image of an element of $\mathcal F(X)$ in $\mathcal F_Z(X)$ will be called its \emph{germ} at $Z$.

Let $\mathfrak X$ be a Nash stack. An \emph{open Nash substack} $\mathfrak U$ of $\mathfrak X$ is a strictly full subcategory $\mathfrak U \subset \mathfrak X$ such that $\mathfrak U$ is a Nash stack and $\mathfrak U\to \mathfrak X$ is an open immersion; that is, the morphism is representable, and for one, equivalently any, presentation $X\to \mathfrak X$, the morphism $\mathfrak U\times_{\mathfrak X} X\to X$ is an open immersion (embedding) of Nash manifolds.

Let $\mathfrak Z$ be the ``complement of $\mathfrak X$'', a ``closed substack'' of $\mathfrak X$. There is an obvious way to define this as a stack over $\mathfrak N$, but it will not necessarily be a Nash stack, since it is not necessarily a smooth image of a smooth semi-algebraic manifold. In any case, for us the notion of ``closed substack'' will be just symbolic, the rigorous notion being its open complement.

We define the \emph{stalk of $\mathfrak S$ at $\mathfrak Z$} to be the cosheaf $\mathfrak S(\bullet)_{\mathfrak Z}$ over $\mathfrak X_\infty$ defined by
$$\mathfrak S(U,u)_{\mathfrak Z}:=  \mathcal S(U)/ \mathcal S(U \times_{u,\mathfrak X} \mathfrak U),$$
for every $(U,u)\in\ob\mathfrak X$.

Notice that $U \times_{u,\mathfrak X} \mathfrak U$ is an open Nash submanifold of $U$, and therefore its space of Schwartz measures is a closed subspace of $\mathcal S(U)$, so the quotient makes sense.

\begin{proposition}
The association $(U,u)\mapsto \mathfrak S((U,u))_{\mathfrak Z}$ defined above is indeed a cosheaf, with strict push-forward maps, on the smooth site of $\mathfrak X$.
\end{proposition}

\begin{proof}
Let $\pi:(V,v)\to (U,u)$ be a smooth map over $\mathfrak X$. If $U':=U \times_{u,\mathfrak X} \mathfrak U$, then $V':=V \times_{v,\mathfrak X} \mathfrak U = \pi^{-1}(U')$. Therefore, 
$$ \pi_!(\mathcal S(V)) = \mathcal S(\pi(V)),$$
$$\pi_!(\mathcal S(V')) = \mathcal S(\pi(V) \cap U') \subset \mathcal S(U'),$$
and therefore the push-forward map
$$ \mathfrak S((V,v))_{\mathfrak Z} \to \mathfrak S((U,u))_{\mathfrak Z}$$
is well-defined and strict, since $\mathcal S(\pi(V))+\mathcal S(U') = \mathcal S(\pi(V)\cup U')$ is closed in $\mathcal S(U)$.

Now assume that $V\to U$ is surjective. We verify that we have a coequalizer diagram (omitting, for simplicity, the maps to $\mathfrak X$ from the notation):
$$\mathfrak S(V\times_U V)_{\mathfrak Z} \rightrightarrows \mathfrak S(V)_{\mathfrak Z} \to \mathfrak S(U)_{\mathfrak Z}.$$

This follows immediately from the corresponding coequalizer diagrams:
$$ \mathcal S(V\times_U V)\rightrightarrows\mathcal S(V)\to \mathcal S(U),$$
$$ \mathcal S(V'\times_U V')\rightrightarrows\mathcal S(V')\to \mathcal S(U'),$$
and the fact that $(V\times_U V) \times_{\mathfrak X} \mathfrak U = V'\times_U V'$.

\end{proof}

The stalk, as a cosheaf, has a lot of representation-theoretic interest. For the purposes of this paper, however, I will restrict myself to an ad hoc definition of its \emph{global sections} as
\begin{equation}\label{glst1} \mathcal S(\mathfrak X)_{\mathfrak Z} := \mathcal S(\mathfrak X) / \overline{\iota(\mathcal S(\mathfrak U))},\end{equation}
where $\iota: \mathcal S(\mathfrak U)\to \mathcal S(\mathfrak X)$ is the natural morphism.

Notice the following:
\begin{lemma}\label{globalstalk}
Let $X\to\mathfrak X$ be any presentation, $R_X = X\times_{\mathfrak X} X$, with the ``source'' and ``target'' maps $s, t : R_X \to X$. Then $\mathcal S_{\mathfrak Z}(\mathfrak X)$ can alternatively be described as
\begin{equation}\label{glst2}
\mathcal S(\mathfrak X)_{\mathfrak Z} = \mathcal S(X)_{\mathfrak Z} / \overline{(s_!-t_!) \mathcal S(R_X)_{\mathfrak Z}}.
\end{equation}
\end{lemma}

\begin{proof}
This simply follows from the definitions: 

If $U:=X\times_{\mathfrak X} \mathfrak U\to \mathfrak U$ is the corresponding presentation of $\mathfrak U$, then both \eqref{glst1} and \eqref{glst2} amount to the formula:
$$\mathcal S(\mathfrak X)_{\mathfrak Z} = \mathcal S(X) /\overline{\mathcal S(U) + (s_!-t_!)\mathcal S(R_X)}.$$
\end{proof}

\section{Local and global stalks for affine reductive group quotients} \label{sec:stalks}

\subsection{Luna's \'etale slice theorem}

We now focus on the case $\mathcal X=$ an algebraic quotient stack of the form $X/G$, where $X$ is a smooth affine variety and $G$ is a reductive group acting on it.

The variety and the group may be defined over a number field $k$, in which case the Nash stack obtained from a completion $F=k_v$ of $k$ (denoted by $\mathfrak X$ in \S \ref{Nash-quotient}) will be denoted by $\mathcal X_v$. (This notation is compatible with the notation $X_v = X(k_v)$ for a smooth variety, if we consider the Nash manifold $X(k_v)$ as a Nash stack.) Let $\pi:\mathcal X \to \cc$ be the invariant-theoretic quotient attached to $\mathcal X$, i.e., $\cc$ is the affine variety $X\sslash G = \spec k[X]^G$. (It does not depend on the presentation $\mathcal X=X/G$ chosen.)

It is well-known that there is a bijection between geometric points of $\cc$ and closed geometric orbits of $G$ on $X$. 
The fiber of $\mathcal X$ over a point $\xi \in \cc(k)$ contains, as a closed substack, a \emph{gerbe} $\mathcal X_\xi$ corresponding to the closed geometric orbit of $G$ on the preimage of $X$ in $\xi$. By definition, a gerbe (over $k$) is special type of stack characterized by the fact that the morphisms $\mathcal X_\xi\to\spec k$ and $\mathcal X_\xi \to \mathcal X_\xi\times_{\spec k} \mathcal X_\xi$ are epimorphisms. In simple terms, over the algebraic closure it becomes a stack of the form $\pt/H$, where $H$ is an algebraic group (reductive, in our case), but the automorphism group $H$ may not, in general, be defined over $k$. Note that we don't denote by $\mathcal X_\xi$ the whole fiber over $\xi$, but just its closed sub-gerbe.

 The gerbe is called \emph{neutral} if it admits a section $\tilde\xi:\spec k \to \mathcal X_\xi$, that is: $\tilde\xi\in\Ob(\mathcal X_{\xi, k})$, in which case it is isomorphic to $\spec k/H$, where $H=\Aut(\tilde\xi)$. The group $H$ depends, of course, on the chosen $k$-point $\tilde\xi$, and the isomorphism class of this $k$-point determines $H$ up to $H(k)$-conjugacy. All the isomorphism classes in $\mathcal X_{\xi,k}$ are obtained from the chosen one by choosing the isomorphism class of an $H$-torsor, in which case $H$ gets replaced by the $H$-automorphism group of the torsor. We will call a point $\xi\in\cc(k)$ \emph{neutral} if $\mathcal X_\xi$ is such. In the general case, the gerbe $\mathcal X_\xi$ becomes neutral after passing to a finite extension. For a discussion of these gerbes in the case of the adjoint quotient of a reductive group, cf.\ \cite{Ko-rational}.

Assume that $\xi\in\cc(k)$ is a neutral point, and let $\tilde\xi:\spec k\to \mathcal X_\xi$ be a section over $\xi$. Without loss of generality (as we may replace the presentation $X/G$ by the presentation $X^T/G^T$, where $T$ denotes a $G$-torsor, or even replace $G$ by $\GL_N$, see \S \ref{Nash-quotient}), the section corresponds to a $G(k)$-orbit $C$ in $X(k)$. Let $\mathcal V$ denote the quotient stack of the normal bundle of $C$ by $G$. It can be identified with the quotient $V/H$, where $V$ is the fiber of the normal bundle over a point $x\in C$, and $H=G_x$ is the stabilizer of this point.  The isomorphism class of $\tilde \xi$ determines the pair $(H, V)$ up to the simultaneous action of an element of $H(k)$ on $H$ (by conjugation) and on $V$. On the other hand, choosing a different isomorphism class for $\tilde\xi$ will have the following effect: $H$ gets replaced by an inner twist $H'$ corresponding to a Galois $1$-cocycle into $H(\bar k)$, and $V$ is replaced by the representation $V'$ of $H'$ obtained by twisting by this cocycle. 

In particular, the quotient $\mathcal V$ is canonically determined up to a canonical $1$-isomorphism by $\xi$, and does not depend on the choice of $\tilde\xi$. On the other hand, the pair $(H,V)$ is determined up to simultaneous $H(k)$-action by the isomorphism class of $\tilde\xi$. Whenever we have chosen a $k$-point $\tilde\xi$, we will be using the presentations $\mathcal X=X/G$, $\mathcal V=V/H$ corresponding to this $k$-point as above (i.e., with $\tilde\xi$ corresponding to a $G(k)$-orbit $C$ on $X(k)$ and $V$ being the fiber of the normal bundle over a $x\in C$), without further mention.

We now recall Luna's slice theorem \cite{Luna}, in its generalization by Alper \cite{Alper}, in order to use it in our local and global analysis of Schwartz spaces. It holds over an arbitrary field (or more general, locally noetherian base), which here we denote by $k$:

\begin{theorem}\label{lunathm}
Let $\xi\in \cc(k)$ be a neutral point, and choose a section $\tilde \xi:\spec k\to \mathcal X$ over $\xi$, corresponding to a $G(k)$-orbit $C$ in $X(k)$. Let $\mathcal V=V/H$ be as above. Identifying the stabilizer of a point $x_0 \in C$ with $H$, there is an $H$-stable subscheme $W\hookrightarrow X$, smooth and affine over $k$, containing $x_0$, and an $H$-equivariant morphism of pointed schemes: $(W,x_0)\to (V,0)$, such that the induced diagram of pointed stacks
\begin{equation}\label{etale} \xymatrix{
& \ar[dl] W/H \ar[d]^{\pi_W}\ar[dr] &\\
\mathcal V=V/H \ar[d]_{\pi_V} & \ar[dl] \cc_W=W\sslash H \ar[dr] & \mathcal X=X/G \ar[d]^\pi\\
\cc_V=V\sslash H && \cc = X\sslash G
}\end{equation}
is Cartesian with \'etale diagonals.
\end{theorem}

\begin{remark}\label{remarkslice}
Although the stack $\mathcal V$ is completely determined by $\xi$, as mentioned above, the identification of \'etale neighborhoods of $\mathcal V$ and $\mathcal X$ of the above theorem is not canonical.

We can make it a little more canonical by recalling that $V$ is, up to the $H$-action, identified with the fiber of the normal bundle over $x_0$. We can also identify the tangent space of $W$ at $x_0$ with $V$ (through the canonical quotient map from the tangent space to $x_0$ of $X$ to the normal space of $x_0 G$), and then require that the \'etale map $W\to V$ induce the identity on tangent spaces. This is implicit in the construction of \cite[III.1.Lemme]{Luna}, but will not be used here.
\end{remark}

The way that \'etale neighborhoods of $\tilde\xi\in \Ob(\mathcal X_k)$ and $0\in \Ob((V/H)_k)$ are identified is important (we freely use $0$ for the zero point in $V(k)$ but also for its images in $(V/H)_k$ and $(V\sslash H)(k)$), and we will insist on clarifying the dependence of various constructions on choices. Let $W, W'$ be two $H$-stable subvarieties of $X$ as in Theorem \ref{lunathm}, together with \'etale, $H$-equivariant morphisms: $W\to V$, $W'\to V$. This gives rise to a diagram with \'etale diagonal maps:
\begin{equation}\label{etaleV} \xymatrix{
& \mathcal Y \ar[dl]\ar[d]\ar[dr]&\\
\mathcal V\ar[d] & \cc_{\mathcal Y} \ar[dl]\ar[dr]& \mathcal V\ar[d] \\
\cc_V && \cc_V,
}\end{equation}
where $\mathcal Y = W/H\times_{\mathcal X} W'/H$, $\cc_{\mathcal Y}=W\sslash H\times_{\cc} W'\sslash H$.

Again, this is a diagram of ``pointed stacks'', in the sense that there is a distinguished isomorphism class of points $\xi'\in \Ob(\mathcal Y_k)$  mapping to the isomorphism class of $0\in \Ob(\mathcal V_k)$ and, moreover, the diagram induces the identity on the (neutral) gerbe $\mathcal V_0 \subset \mathcal V$ (as easily follows from the construction of this diagram by Theorem \ref{lunathm} --- namely, the fact that $0\in V(k)$ corresponds to a $G(k)$-orbit on $X(k)$).  

We should make sure that, whenever we make a definition for $\mathcal X$ using Luna's \'etale slice theorem, the  objects that we define are invariant with respect to correspondences of the form \eqref{etaleV} defined over $k$. However, the diagram \eqref{etaleV} does not capture all the information present; for example, it is not clear how $k$-points of $W$ and $W'$ are related, and when they ``appear'' in the fiber product $\mathcal Y = W/H\times_{\mathcal X} W'/H$ (i.e., when they have the same image in $\mathcal X$). In order to encode such information, we complete this diagram to the following diagram with Cartesian squares:

\begin{equation}\label{fiberproduct}
\xymatrix{
&& P \ar[dl]\ar[dr]&&\\
& W_1\ar[dl]\ar[dr] & & W_2\ar[dl]\ar[dr] &\\
V \ar[dr] && \mathcal Y \ar[dl]\ar[d]\ar[dr]&& V\ar[dl]\\
&\mathcal V\ar[d] & \cc_{\mathcal Y} \ar[dl]\ar[dr]& \mathcal V\ar[d] &\\
&\cc_V && \cc_V. &
}
\end{equation}

The varieties $W_1$, $W_2$ and $P$ are defined by the Cartesian property, i.e., $W_1$ and $W_2$ are the fiber products of $\mathcal Y$ with $V$ over $\mathcal V$ with respect to the left and right maps of \eqref{etaleV}, respectively, and $P$ is their fiber product over $\mathcal Y$. Alternatively, in terms of the subvarieties $W, W'$ of $X$ used to define $\mathcal Y$, we have
\begin{equation}\label{PWW}
P = W \times_{\mathcal X} W'.
\end{equation} 

The space $P$ carries an action of $H\times H$, and is an $H$-torsor over $W_1$, resp.\ $W_2$, with respect to the second, resp.\ first copy of $H$. It contains a distinguished point $x_{00}$ which corresponds to the point $x_0$ of $W$ and $W'$, whose $H\times H$-orbit is isomorphic to $H$ (acted upon by left and right multiplication). The fiber at $x_{00}$ of the normal bundle to its orbit can be identified with $V$ via either the projection to $W_1$ or the projection to $W_2$, and it is easy to see that these identifications coincide. Applying Luna's \'etale slice theorem to $P$, we deduce that there is a subvariety $W_3$, containing $x_{00}$ and stable under the diagonal copy of $H$, and an \'etale map $W_3\to V$ giving rise to a diagram analogous to \eqref{etale}. What we will need from this is the following diagram with Cartesian squares and \'etale diagonal maps:

\begin{equation}\label{etaleW} \xymatrix{
& W_3 \ar[dl]\ar[d]\ar[dr]&\\
W_1\ar[d] & \cc_{W_3} \ar[dl]\ar[dr]& W_2\ar[d] \\
\cc_{\mathcal Y}=\cc_{W_1} && \cc_{\mathcal Y}=\cc_{W_2},
}\end{equation}
where $\cc_{W_i} = W_i\sslash H$.

We will return to these diagrams several times. As a first corollary, we get the existence of a well-defined \emph{semisimplification} for $k$-point of $\mathcal X$.

Let $\xi \in \cc(k)$ be a neutral point, and choose a closed $k$-point $\tilde\xi:\spec k\to \mathcal X$ over $\xi$, with stabilizer $H$.
Let $\mathcal N_\xi$ be the preimage of $\xi$ under the map: $\mathcal X\to \cc$. Let $W$ be an $H$-stable subvariety of $X$ as in Luna's theorem, and $N_W$ the preimage of the distinguished point of $\cc_W$ in $W$. From \eqref{etale} we get an equivalence
 $$  N_W/H \xrightarrow\sim \mathcal N_\xi.$$
Composing with the natural map $N_W/H\to \spec k/H \xrightarrow\sim \mathcal X_\xi$, we get a diagram
\begin{equation}\label{ss}
\mathcal N_\xi \overset\sim\from N_W/H \to \spec k/H \xrightarrow{\sim} \mathcal X_\xi.
\end{equation}

\begin{proposition}\label{Jordan}
The map from isomorphism classes of $\mathcal N_\xi$ to isomorphism classes of $ X_\xi$ induced from \eqref{ss} does not depend on the chosen $k$-point $\tilde \xi$, or on the choice of $W$ and the \'etale map $W\to V$. 
\end{proposition}

\begin{proof}
Two choices $W\to V$ and $W'\to V$ as before give rise to the above diagrams, and in particular to \eqref{etaleW}. This, together with the maps $W_1\to W$, $W_2\to W'$ (which are also Cartesian over $\cc_{\mathcal Y} \to \cc_W$, $\cc_{\mathcal Y} \to \cc_{W'}$) give rise to a diagram of equivalences
$$\xymatrix{
& N_{W_3}/H \ar[dl]\ar[dr]&\\
N_W/H \ar[dr] && N_{W'}/H \ar[dl]\\
& \mathcal N_\xi &
},$$
and both ``maps'' (at the level of isomorphism classes) $\mathcal N_\xi \to \spec k/H$ are obtained by inverting those and composing with $N_{W_3}/H\to \spec k/H$. Therefore, they coincide.

It is easy to see that a choice of different $k$-point modifies these maps and the isomorphism $\spec k/H\xrightarrow\sim \mathcal X_\xi$ compatibly; thus we get a well-defined map from isomorphism classes of objects in $\mathcal N_\xi$ to isomorphism classes of objects in $\mathcal X_\xi$.
\end{proof}

\subsection{Local Schwartz stalks}

Let $F$ be a local field, and assume that $X,G$ are defined over $F$. Denote by $\mathfrak X$ the stack $\mathcal X(F)$, as before. 
We use the notation of the previous subsection, with $\xi$ now being a neutral point of $\cc(F)$. 
Applying the slice theorem and the fact that \'etale maps become local isomorphisms in the semi-algebraic topology on $F$-points (when $F\ne \CC$), we get from Theorem \ref{lunathm}:

\begin{corollary}\label{Nashnbhds}
Assume $F\ne \CC$. There is a semi-algebraic open neighborhood $U$ of $\xi$ in $\cc(F)$ and a semi-algebraic open neighborhood $U'$ of the image of $0\in c_V(F)$, such that the corresponding open Nash substack $\mathfrak X_U = \mathfrak X \times_{\cc(F)} U$ of $\mathfrak X$ is equivalent to the open Nash substack of $\mathcal V(F)$ lying over $U'$.

In particular, setting $\mathfrak Z = \pi^{-1}(\xi)$, $\mathfrak Z'=\pi_V^{-1}(0)$, where $\pi: \mathfrak X\to \cc(F)$, $\pi_V:\mathcal V(F)\to \cc_V(F)$, the stalk  (see \S \ref{ssstalks}) of the Schwartz cosheaf $\mathfrak S_{\mathfrak X}$ over $\mathfrak Z$ is isomorphic to the stalk of $\mathfrak S_{\mathcal V(F)}$ over $\mathfrak Z'$.
\end{corollary}

From now on we will be denoting Schwartz stalks over $\mathfrak Z$ simply by the index $~_\xi$, and stalks over $\mathfrak Z'$ by the index $~_0$. Note, however, that $\mathfrak Z$ (and this notation for the stalk) does \emph{not} refer to the gerbe denoted by $\mathcal X_\xi$ before, but contains it.

Notice that the isomorphism between the stalks is a tautology, given the first statement of the corollary and the fact that the stalks are cosheaves supported away from the complement of $\mathfrak Z$, resp.\ of $\mathfrak Z'$. Applying the isomorphism to the global sections \eqref{glst1}, we get an isomorphism:
\begin{equation}\label{glsections} \mathcal S(\mathcal X(F))_\xi \simeq \mathcal S(\mathcal V(F))_0.   
\end{equation}
This isomorphism, though, depends on the identification of \'etale neighborhoods as in \eqref{etale}, and it is important to keep this in mind in constructions that follow.

\begin{remark}
From Remark \ref{restriction-stacks} it follows that the isomorphism \eqref{glsections} is induced from a diagram of the form \eqref{etale} even when $F=\CC$, simply by considering $\mathfrak X$ as a Nash $\RR$-manifold. Therefore, in what follows we put no restriction on the local field $F$.
\end{remark}

Now I describe a way to pass from measures to functions, which will be useful for the regularization of orbital integrals. We will only work with the coinvariant space $\mathcal S(V(F))_{H(F)}$ which by Corollary \ref{directsum} is canonically a direct summand of $\mathcal S(\mathcal V(F))$ (and hence also provides a direct summand $\left(\mathcal S(V(F))_{H(F)}\right)_0$ of the stalk $ \mathcal S(\mathcal V(F))_0$).

We would like to divide all Schwartz measures on $V(F)$ by a Haar measure on $V(F)$, in order to obtain Schwartz functions on $V(F)$. However, Haar measure on $V(F)$ will not be preserved by the correspondences of the form \eqref{etaleV}; therefore we need to consider more general classes of measures.

Let $\delta_V$ be the character by which the group $H(F)$ acts on Haar measure on $V(F)$; it is the absolute value of an algebraic character $\mathfrak d_V$. Consider the space $\mathcal L$ of volume forms on $V$ which are $H$-eigenforms with eigencharacter $\mathfrak d_V$; clearly, they are all multiples of a ``Haar'' volume form by a polynomial function on $\cc_V=V\sslash H$; thus, we can think of them as a line bundle over $\cc_V$ (also to be denoted by $\mathcal L$). The absolute value of such volume a volume form, multiplied by any $H(F)$-invariant Nash function on $V(F)$, is a smooth, $(H(F),\delta_V)$-equivariant measure of polynomial growth on $V(F)$, and all those measures are multiples of the Haar measure by such a function. We will think of them as ``sections'' of a ``Nash line bundle'' $\mathscr L$ over $\cc_V(F)$; although $\cc_V(F)$ may not be smooth, or, even if it is,  $H(F)$-invariant Nash functions on $V(F)$ do not necessarily descend to Nash functions on $\cc_V(F)$, we will by abuse of language talk about the ``fiber'' of $\mathscr L$ over $0\in \cc_V(F)$ to refer to the quotient of the sections of $\mathscr L$ by those sections of the form $f(v) dv$, where $dv$ is a Haar measure and $f(v)$ in an $H(F)$-invariant Nash function on $V$ which vanishes on the preimage of $0$. (This fiber is clearly one-dimensional.)

Let $\mathcal F$ denote the cosheaf of Schwartz functions on $V(F)$. Let $N\subset V$ be the preimage of $0\in\cc_V$, the ``nilpotent cone'' of all elements of $V$ which contain $0$ in their $H$-orbit. Let $\mathcal F(N(F))$ denote the space of functions on $N(F)$ which are restrictions of elements of $\mathcal F(V(F))$. In the non-Archimedean case, this is isomorphic to $\mathcal F(V(F))_{N(F)}$, (global sections of) the stalk of $\mathcal F$ over $N(F)$. In the Archimedean case, though, it is just a quotient of it, since it does not remember derivatives in the transverse direction. In any case, we have a map
$$ \mathcal F(V(F))_{N(F)} \to \mathcal F(N(F)).$$
 The following is obvious:

\begin{lemma}\label{meastofunction}
Multiplication by a $\delta_V$-eigenmeasure $\mu\in \mathscr L$, which is non-vani\-shing at $0\in\cc_V(F)$ (i.e., has non-zero image in the fiber over $0$) gives rise to an isomorphism of the coinvariant spaces of stalks
\begin{equation}
\left(\mathcal F(V(F)_{N(F)}\right)_{(H(F),\delta_V^{-1})} \xrightarrow\sim \left(\mathcal S(V(F))_{N(F)}\right)_{H(F)}.
\end{equation}
 Moreover, the resulting map
\begin{equation}\label{maptofns1}
\left(\mathcal S(V(F))_{N(F)}\right)_{H(F)} \to \mathcal F(N(F))_{(H(F),\delta_V^{-1})} 
\end{equation}
depends only on the image of $\mu$ in the fiber of $\mathscr L$ over $0\in\cc_V(F)$.
\end{lemma}

Notice that the index $~_{(H(F),\delta_V^{-1})}$ denotes $(H(F),\delta_V^{-1})$-coinvariants of the space, i.e., its quotient by the closure of the span of vectors of the form $(v-\delta_V(h) h\cdot v)$ (so that $H(F)$ acts on the quotient by the character $\delta_V^{-1}$). 
I explain the statement in words, because the notation has become a bit heavy: 
The $H(F)$-coinvariants of the stalk of $\mathcal S(V(F))$ over $N(F)$ are identified, via this choice of measure, with the $(H(F),\delta_V^{-1})$-coinvariants of the corresponding stalk of Schwartz functions; moreover, the resulting map that one obtains by restricting the functions to $N(F)$ depends only on the (non-zero) image of the chosen measure in the ``fiber'' of $\mathscr L$ at $0\in\cc_V(F)$.

Notice that taking $H(F)$-coinvariants commutes with taking stalks at $N(F)$, by Lemma \ref{globalstalk}. In other words, $\left(\mathcal S(V(F))_{N(F)}\right)_{H(F)}$ is the same as the direct summand $\left(\mathcal S(V(F))_{H(F)}\right)_0$ of the stalk $\mathcal S(\mathcal V(F))_0$ that we saw before.
Moreover, by  Corollary \ref{directsum} the coinvariant space $\mathcal S(V(F))_{H(F)}$ is not just a subspace, but also canonically a direct summand of $\mathcal S(\mathcal V(F))$:
\begin{equation}\label{directsummand-Vlocal}
\mathcal S(V(F))_{H(F)} \underset{\from}{\hookrightarrow} \mathcal S(\mathcal V(F)).
\end{equation}
Thus, from \eqref{maptofns1} we get a map
\begin{equation}\label{maptofns}\mathcal S(\mathcal V(F))_0 \to \mathcal F(N(F))_{(H(F),\delta_V^{-1})},
\end{equation}
depending only on the evaluation of $\mu$ over $0$.

In the next subsection we will see that the direct summand \eqref{directsummand-Vlocal} (though not, necessarily, the map \eqref{maptofns}) is preserved at the level of stalks by a diagram of the form \eqref{etaleV}.

\subsection{Global Schwartz stalks}
Now let $X, G$ be defined over a global field $k$, and let $\xi\in\cc(k)$ be a neutral point. \emph{From now on we will assume that the reductive automorphism group associated to $\xi$ is connected.} The reason is that the presentation would be more complicated if it were not connected, and I do not presently have the experience to know which approach would be more useful in applications. 

To demonstrate the issues that arise in the non-connected case, recall that the semisimple (closed) geometric point over the neutral point $\xi$  corresponds to a closed substack of $\mathcal X$ of the form $\mathcal X_\xi=\spec k/H$, for some reductive group $H$. Its $k$-points: $\spec k\to \spec k/H$ can be partitioned into equivalence classes for a relation stronger than geometric equivalence, which we can call ``stable equivalence'', whereby the $k$-point corresponding to the trivial $k$-torsor is equivalent to all points in the image of
$$ \spec k/H^0(k)\to \spec k/H(k)$$
(where $H^0$ is the connected component of $H$). Since every $k$-point is isomorphic to the trivial one for some presentation $\spec k/H\simeq \spec k/H'$, this uniquely defines a relation, which is easily checked to be an equivalence relation on $\mathcal X_\xi(k)$.

Note that the definition of stable equivalence can be extended to arbitrary geometric classes of $k$-points of $\mathcal X$, simply by observing that to any isomorphism class of $k$-points: $x:\spec k \to \mathcal X$ lying over $\xi\in \cc(k)$ one can attach by Proposition \ref{Jordan} a unique isomorphism class of \emph{closed} $k$-points: $x_s:\spec k \to \mathcal X$ lying over $\xi$ (the ``semisimple part of $x$''). We can then call two points in the same geometric equivalence class \emph{stably equivalent} if this is the case for their semisimple parts. In the adjoint quotient of the group, for example, $x_s$ is the $k$-conjugacy class of the semisimple part of the Jordan decomposition of $x$, and this notion of stable conjugacy is the one of \cite{Ko-rational}. 

Now, for two stably equivalent closed $k$-points: $\spec k \to \spec k/H$, the $H$-torsors that they define (with respect to this presentation), and hence also the stabilizer groups, are isomorphic at almost every place, and the map $G\to H\backslash G$ gives an onto map of integral points, for almost every place. This is not the case, in general, with non-stably equivalent $k$-points, which creates the need for some extra bookkeeping, that I will not do here. Thus, from now on we assume that the stabilizer groups of the closed points we are considering are connected.

Let $\xi\in \cc(k)$. We define the \emph{global Schwartz stalk} at $\xi$ as a restricted tensor product:
\begin{equation}\label{Stp} \mathcal S(\mathcal X(\adele))_\xi := \bigotimes' \mathcal S(\mathcal X_v)_\xi,\end{equation}
with respect to ``basic vectors'' that will be described below.
Recall that $\mathcal S(\mathcal X_v)$ denotes the nuclear Fr\'echet space --- vector space without topology in the non-Archimedean case --- of ``global sections'' of the Schwartz cosheaf over $\mathcal X_v$. Thus, from this point on we are talking about actual (topological) vector spaces, not cosheaves. I repeat that the index $~_\xi$ that we are using does not refer to the stalk at the closed gerbe $\mathcal X_\xi$, but to the stalk over the \emph{whole} preimage in $\mathcal X$ of $\xi\in\cc$, which includes all geometric points whose closure contains the closed point corresponding to $\xi$. 

The restricted tensor product is taken with respect to a ``basic vector'' described as follows: First, fix a section $\spec k\to \mathcal X$ over $k$, in order to have a pair $(H,V)$ as before. Fix a (non-zero) invariant volume form $\omega_H$ on $H$ (over $k$); all local volumes will be computed with respect to its absolute power. Moreover, fix a section $\omega_V$ of the line bundle $\mathcal L$ over $\cc_V$ described before Lemma \ref{meastofunction}; here, both the line bundle and the section $\omega_V$ are defined over $k$. Finally, fix compatible integral models for $X, G, V, H$ over the $S$-integers $\mathfrak o_S$, where $S$ is some finite set of places, containing the Archimedean ones.

Now consider the following vector in $\mathcal S(V(k_v))$, where $v\notin S$:
\begin{equation}\label{basic-V} f_v^0:= \frac{1}{|\omega_H|_v(H(\mathfrak o_v))} \cdot 1_{V(\mathfrak o_v)} \cdot |\omega_V|_v,
\end{equation}
where $1_{V(\mathfrak o_v)}$ denotes the characteristic function of $1_{V(\mathfrak o_v)}$. 

\begin{lemma}
For any two sets of choices as above, the images of the resulting vectors $f_v^0$ in the stalk $\mathcal S(\mathcal V(k_v))_0$ coincide for almost every $v$.
\end{lemma}

\begin{proof}
First, for any two sections $\spec k\to \mathcal X$ and $\mathfrak o_S$-models of the corresponding groups $H$, the resulting $\mathfrak o_v$-groups $H_{\mathfrak o_v}$ will be canonically isomorphic up to $H(\mathfrak o_v)$-conjugacy for almost every $v$; we are using here our assumption that $H$ is connected.

Thus, for almost every place the sets $H(\mathfrak o_v)$ resulting from two different choices are identified, and so are the sets $V(\mathfrak o_v)$. The Haar measures $|\omega_H|_v$ are equal almost everywhere. The measures $|\omega_V|_v$ do not need to coincide, but their images in the one-dimensional fiber of $\mathscr L$ over $0\in\cc_V$ do coincide, and this is enough by Lemma \ref{meastofunction}, in the non-Archimedean case, in order to identify the resulting elements in the stalk. 
\end{proof}

The images (``germs'') of the vectors $f_v^0$ in the stalk \emph{will be our basic vectors for the definition of the global Schwartz stalk $\mathcal S(\mathcal V(\adele))_0$}, i.e.,
$$\mathcal S(\mathcal V(\adele))_0 := \bigotimes' \mathcal S(\mathcal V_v)_0,$$
where the restricted tensor product is taken with respect to the basic vectors $f_v^0$ above. 

Before I proceed to use this for defining the global Schwartz stalk of $\mathcal X$ as in \eqref{Stp}, I introduce, for later use, a map to a related restricted tensor product of spaces of functions. Namely, choose again a global section $\omega_V$ of the line bundle $\mathcal L$ over $\cc_V$ (for example, the Haar volume form on $V$). Recall that $N$ denotes the nilpotent cone in $V$. Using the absolute value of $\omega_V$, at every place $v$, we get out of \eqref{maptofns}  maps
$$ \mathcal S(\mathcal V(k_v))_0 \to \mathcal F(N(k_v))_{(H(k_v),\delta_V^{-1})}.$$

Putting them all together, we obtain a linear map:
\begin{equation}\label{tofns-global}
\mathscr E: \mathcal S(\mathcal V(\adele))_0 \to \bigotimes_v' \mathcal F(N(k_v))_{(H(k_v),\delta_V^{-1})},
\end{equation}
where the restricted tensor product on the right-hand side is taken with respect to the images of the functions:
$$\frac{1}{|\omega_H|_v(H(\mathfrak o_v))} \cdot 1_{V(\mathfrak o_v)}.$$

\begin{remark}\label{notfunctions} Notice that this is \emph{not} the space of restrictions to $N(\adele)$ of global Schwartz functions on $V(\adele)$; rather, it is a \emph{formal multiple} of that. The reason is that the partial Euler product of the factors $\frac{1}{|\omega_H|_v(H(\mathfrak o_v))} $ may \emph{not} make sense, even if we try to interpret it as a special value of a partial $L$-function. This happens precisely when (the connected reductive group) $H$ has a non-trivial $k$-character group; but these factors will be formally cancelled out when we attempt to integrate against (non-regularized) Ta\-magawa measure.
\end{remark}

We have the following easy lemma:

\begin{lemma}\label{doesnotdepend}
The map \eqref{tofns-global} does not depend on the volume form $\omega_V$ chosen.
\end{lemma}

\begin{proof}
Indeed, we know that the local factors of this map depend only on the image of $|\omega_V|_v$ in the fiber over $0\in\cc_V(k_v)$. For any two choices, those will coincide at almost every place, and their quotients at the remaining places will multiply to $1$, by the product formula.
\end{proof}

We return to the definition of the stalk $\mathcal S(\mathcal X(\adele))_\xi$. In order to use the basic functions for $\mathcal V$, via  \eqref{glsections}, to define the basic vectors for the stalk of $\mathcal X$, we need to make sure that the basic vectors are invariant, at almost every place, under the isomorphisms of stalks induced by a diagram of the form \eqref{etaleV}. The whole point of introducing the line bundles $\mathcal L, \mathscr L$, instead of just choosing Haar volume forms on $V$, was precisely to make these definitions invariant.

\begin{proposition}\label{preservesbasic}
Consider a correspondence as in \eqref{etaleV} defined over $k$, and the automorphism: $\mathcal S(\mathcal V_v)_0 \xrightarrow\sim \mathcal S(\mathcal V_v)_0$ that it induces on stalks. For almost every place $v$, this automorphism preserves the basic vector, and for every place $v$ it preserves the distinguished direct summand of $\mathcal S(\mathcal V_v)_0$ corresponding to $H(k_v)$-coinvariants of $\mathcal S(V(k_v))$.
\end{proposition}

The proof will use the following lemma that will recur in other proofs, as well:

\begin{lemma}\label{Nlifts}
Consider diagram \eqref{fiberproduct}, induced from two different choices of data for Luna's theorem. Let $N\subset V$ denote the preimage of $0\in \cc_V$, and $N_P\subset P$ the preimage of the distinguished point of $\cc_{\mathcal Y}$. Consider the map: $P\to V$ induced from either the left or the right sequence of maps from $P$ to $V$. Then this map admits a section $N\to N_P$ over $N$. In particular, the map $N_P\to N$ is surjective on $R$-points, for every ring $R$ over which this diagram is defined.
\end{lemma}

\begin{proof} 
Indeed, let us denote by $N_{W_i}$ the corresponding preimages of the distinguished point of $\cc_{\mathcal Y}$ in $W_i$. Recall diagram \eqref{etaleW}, obtained by applying Luna's slice theorem to $P$. The maps $\cc_{W_3}\to \cc_{\mathcal Y}$ and $\cc_{\mathcal Y} \to \cc_V$ being \'etale, and the diagrams \eqref{fiberproduct} and \eqref{etaleW} Cartesian, we get isomorphisms for the preimages of $0 \in \cc_\bullet$ (where $\bullet = W_3$, $\mathcal Y$ or $V$): 
$$ N_{W_3}\xrightarrow{\sim} N_{W_i} \xrightarrow{\sim} N$$
for $i=1,2$.
\end{proof} 

\begin{proof}[Proof of Proposition \ref{preservesbasic}]

Consider again diagram \eqref{fiberproduct}. The fact that, by Lemma \ref{Nlifts}, the (smooth, $H$-torsor) map $N_P\to N$ is surjective on $k_v$-points, implies that this will be the same for semi-algebraic neighborhoods of $N_P$ and $N$, as well (when $k_v=\CC$ we base change to $\RR$ for this statement to be true), and therefore the map on coinvariant stalks
$$ \left(\mathcal S(P(k_v))_{(H\times H)(k_v)}\right)_0 \to \left(\mathcal S(V(k_v))_{H(k_v)}\right)_0$$
is surjective (both for the ``left'' and ``right'' sequence of arrows). This shows that the distinguished summand is preserved by the isomorphism of stalks induced from \eqref{etaleV}.

The maps $W_1\to V$, $W_2\to V$ are \'etale and $H$-equivariant; that means that the pull-back of any $(H,\mathfrak d_V)$-equivariant volume form on $V$ is again an $(H,\mathfrak d_V)$-equivariant volume form on $W_1$, resp.\ $W_2$. Volume forms of this type on $W_i$ form again a line bundle over $\cc_{\mathcal Y}=W_i\sslash H$, whose global sections are generated by the pull-back of Haar measure. It will not create any confusion to denote this line bundle again by $\mathcal L$.

The space $P$ carries an action of $H\times H$, and is an $H$-torsor over $W_1$, resp.\ $W_2$, with respect to the second, resp.\ first copy of $H$. 
Since the $W_i$ carry a nowhere vanishing $(H,\mathfrak d_V)$-equivariant volume form, $P$ carries a nowhere vanishing, $(H\times H, \mathfrak d_V\times \mathfrak d_V)$-equivariant volume form. Let $\omega_P$ be such a volume form. I claim that the basic vector of the stalk $\mathcal S(\mathcal V(\adele))_0$ is (at almost every place) the image of the measure:
\begin{equation}\label{basic-P}\frac{1}{|\omega_H|_v(H(\mathfrak o_v))^2} \cdot 1_{P(\mathfrak o_v)} \cdot |\omega_P|_v \in \mathcal S(P(k_v))\end{equation}
under either the left or the right maps to $\mathcal V$. 

By Lemma \ref{Nlifts}, we get that for any integral models and almost every place $v$, the set $N_P(\mathfrak o_v)$ surjects onto $N(\mathfrak o_v)$ through either of the left or the right sequence of maps. It is thus clear that the image of \eqref{basic-P} under either the left or right maps to $V$ coincides as a measure, in a neighborhood of $N(k_v)$, with \eqref{basic-V} at almost every place.

This shows that the basic vector is preserved at almost every place under the correspondence of stalks defined by \eqref{etaleV}.
\end{proof}

\begin{corollary}
Let $x:\spec k\to \mathcal X$ be a semisimple point and let $(H,V)$ be the linearization of $\mathcal X$ at $x$, as per Luna's \'etale slice theorem. Let $\xi$ denote the image of $x$ in $\cc(k)$.

The global Schwartz stalk
$$ \mathcal S(\mathcal X(\adele))_\xi := \bigotimes' \mathcal S(\mathcal X_v)_\xi$$
is well-defined independently of $x$, and in such a way that, for any diagram of the form \eqref{etale} defined over $k$, the isomorphism of stalks of Corollary \ref{Nashnbhds} gives rise to an isomorphism of global stalks
\begin{equation}\label{globalisom}
\mathcal S(\mathcal V(\adele))_0 \xrightarrow\sim \mathcal S(\mathcal X(\adele))_\xi.
\end{equation}

Moreover, the direct summand 
\begin{equation}\label{directsummand-V}
\mathcal S(V(\adele))_{H(\adele),0} \underset{\from}{\hookrightarrow} \mathcal S(\mathcal V(\adele))_0 
\end{equation}
corresponds, under this isomorphism, to a distinguished direct summand:
\begin{equation}\label{directsummand}
\mathcal S(\mathcal X(\adele))_\xi^x \underset{\from}{\hookrightarrow} \mathcal S(\mathcal X(\adele))_\xi,
\end{equation} 
which depends only on the isomorphism class of $x$, not on the choice of diagram \eqref{etale}.
\end{corollary}

The summand $\mathcal S(\mathcal X(\adele))_\xi^x $ corresponds to considering, locally, small semi-algebraic neighborhoods of the $G(\adele)$-orbit containing the $G(k)$-orbit defined by $x$. Locally equivalent $k$-points give rise to the same summand \eqref{directsummand}. Notice that the identification \eqref{globalisom} \emph{does} depend on the choice of diagram \eqref{etale}.

\begin{remark} \label{remarkgroup}
In many applications, the variety $X$ will be homogeneous for a larger group $\tilde G\supset G$, with a chosen $\tilde G$-eigen-volume form $\omega_X$ inducing measures $|\omega_X|_v$ on the points over each completion. What matters, actually, is just that $\omega_X$ is a nowhere vanishing $G$-eigen-volume form. 

For example, in the relative trace formula one deals with a pair of $G$-homogeneous spaces $X_1, X_2$, and one considers the quotient stack $(X_1\times X_2)/G^\diag$. However, the $\tilde G = G\times G$-action on $X_1\times X_2$ typically admits an eigen-volume form.

In this setting, one can define a \emph{global Schwartz space} (not just the stalk), as the restricted tensor product of the local ones with respect to the measures: 
$$\mu_{X_v}= \frac{1}{|\omega_G|_v(G(\mathfrak o_v))} 1_{X(\mathfrak o_v)} |\omega_X|_v.$$

What would these global Schwartz spaces have to do with our global Schwartz stalk at a neutral point $\xi\in \cc(k)$? I claim that, locally at almost every place, the germ of $\mu_{X_v}$ in the stalk over $\xi$ coincides with the basic vector of $\mathcal S(\mathcal X_v)_\xi$. 

Indeed, choosing $x_0$ and $W$ (over $k$) as in Luna's Theorem \ref{lunathm}, and integral models outside of a finite set of places, we have an \'etale map: $W\times^H G\to X$ whose image is the preimage $X'$ of a Zariski open subset of $\cc$. The point $x_0$ will belong to $X(\mathfrak o_v)$ almost everywhere, and because $H$ is assumed connected, the above map will almost everywhere give a surjection:
$$ W(\mathfrak o_v) \times G(\mathfrak o_v) \twoheadrightarrow X'(\mathfrak o_v).$$
The same is true for the \'etale map $W\to V$ with image $V'$. 

This means that the germ of the measure $1_{X(\mathfrak o_v)} |\omega_X|_v$ can be lifted to a measure supported on $W(\mathfrak o_v)$ (for almost every $v$), and similarly for the germ of $1_{V(\mathfrak o_v)}|\omega_V|_v$. Now let us check volume forms.

A nowhere vanishing volume form on $V$ pulls back to a nowhere vanishing volume form $\omega_W^V$ on $W$; if the former is $(H,\mathfrak d_V)$-equivariant, then the same will be the case for the latter.

Similarly, the pull-back to $W\times^H G$ of a $G$-eigen-volume form on $X$ with eigencharacter $\mathfrak d_X$ can be factored into an $H$-eigen-volume form $\omega_W^X$ on $W$ and a $G$-eigen-volume form $\omega_{H\backslash G}$ on $H\backslash G$ valued in some line bundle. The line bundle is defined by the $H$-eigencharacter, but since $H$ is reductive this eigencharacter has to be equal to the restriction of $\mathfrak d_X$ to $H$. Correspondingly, since $(H\backslash G)(\mathfrak o_v) = H(\mathfrak o_v)\backslash G(\mathfrak o_v)$ almost everywhere (by the connectedness of $H$), the germ of the measure:
$$\frac{1}{|\omega_G|_v(G(\mathfrak o_v))} 1_{X(\mathfrak o_v)} |\omega_X|_v$$
in the stalk of $\mathcal S(\mathcal X_v)$ over $\xi$ has to be equal to the germ of the measure
\begin{equation}\label{Wmeasure}\frac{1}{|\omega_H|_v(H(\mathfrak o_v))} 1_{W(\mathfrak o_v)} |\omega_W^X|_v.\end{equation}
Comparing the forms $\omega_W^V$ and $\omega_W^X$, since they are both non-vanishing in a neighborhood of the $H$-fixed point $x_0$, and are both $H$-eigenforms, it follows that their quotient is \emph{polynomial} and \emph{non-vanishing} in an $H$-stable neighborhood of $x_0$. The absolute value of such a function will be equal to $1$ in an $H(k_v)$-stable Nash neighborhood of $x_0$, for almost every $v$, and hence the measure \eqref{Wmeasure} will not change if we replace $\omega_W^X$ by $\omega_W^V$.
\end{remark}

\section{Equivariant toroidal compactifications, and orbital integrals on linear spaces} \label{sec:orbitalintegrals}

\subsection{Overview}

Let $H$ be a connected reductive group over a global field $k$, and let $[H] = H(k)\backslash H(\adele)$. In this section I will discuss a certain class of $H(\adele)$-equivariant compactifications of $[H]$, which I term \emph{equivariant toroidal}. They are not directly related to the toroidal compactifications of \cite{toroidal}, as far as I can tell; rather, they are refinements of the reductive Borel--Serre compactification \cite{BJ}. I then define a notion of ``asymptotically finite'' functions, an extension of the Schwartz space of rapidly decaying smooth functions on $[H]$, where we allow multiplicative behavior ``at infinity'', forming a cosheaf over an equivariant toroidal compactification. Then I proceed to prove the main theorem of this section:

\begin{theorem}
Let $V$ be a finite-dimensional representation of $H$ over $k$, and let $N$ be its ``nilpotent'' set, i.e., the closed subvariety of all points whose $H$-orbit closure contains $0$. For any Schwartz function $f$ on $V(\adele)$, the function 
\begin{equation}\Sigma_N f: h\mapsto \sum_{\gamma \in N(k)} f(\gamma h)\end{equation} on $[H]$ is asymptotically finite.

More precisely, this map represents a continuous map from the space of Schwartz functions $\mathcal F(V(\adele))$ to a space $\Ff_E([H]^\F)$ of asymptotically finite functions on $[H]$ with asymptotic behavior explicitly determined by the representation $V$. 
\end{theorem}

Details on the asymptotic behavior mentioned in the theorem will be given after the discussion of compactifications. This theorem gives a way of defining a regularized integral
$$ \int^*_{[H]} \Sigma_N f (h) dh,$$
which, for representatives $\gamma_i$ of the $H(k)$-orbits on $N(k)$, is formally equal to the sum
\begin{equation}\label{orbital} \sum_i \Vol([H_{\gamma_i}]) \int_{H_{\gamma_i}(\adele)\backslash H(\adele)} f(\gamma h)\end{equation}
of orbital integrals of all rational nilpotent orbits of $H$ on $V$. The regularization of orbital integrals defined here will be used in the next section, in combination with Luna's \'etale slice theorem, to define ``evaluation maps'' (regularized orbital integrals) for more general reductive quotient stacks.

This regularization will be possible
\emph{if and only if} the ``exponents'' of the asymptotic behavior of $\Sigma_N f$ are not ``critical'', in some sense. The prototype for this is the following:

Let $s\in \CC$. Consider the space $\mathcal S(\Gamma\backslash \mathcal H)^s$ of functions on $\Gamma\backslash \mathcal H$, where $\mathcal H$ is the complex upper half-plane and $\Gamma=\SL_2(\Z)$, which are smooth and have the property that 
$$ f(x+iy) \sim y^s$$
for $y\gg 0$, where $\sim$ means that the difference is a function which, together with all its polynomial derivatives, is of rapid decay. Then, the regularized integral
$$ \int^*_{\Gamma\backslash \mathcal H} f(x+iy) \frac{dx d^\times y}{|y|}$$
(where $d^\times y$ denotes multiplicative measure $\frac{dy}{|y|}$)
is well-defined \emph{unless} $s=1$, i.e., unless the growth of the function is inverse to that of volume. The exponent (multiplicative character) $y\mapsto y^1$ is what I call a \emph{critical exponent}. The definition of the regularized integral is as follows: fix any large $T>0$, and define
$$f_t(x+iy) = \begin{cases} f(x+iy), &\mbox{ if }y\le T  \\ f(x+iy) |y|^{-t}, &\mbox{ if } y>T.\end{cases}$$
Then $\int_{\Gamma\backslash \mathcal H} f_t(x+iy) \frac{dx d^\times y}{|y|}$ is convergent for $\Re(t)\gg 0$, and admits meromorphic continuation with only a simple pole at $t+1=s$. Thus, if $s\ne 1$, we can define the regularized integral as the analytic continuation of the above integral to $t=0$.

\subsection*{Conventions and notation for this section}

In this section we will, for simplicity, assume that $k=\QQ$, which we may, if $k$ is a number field, by restriction of scalars. For function fields in positive characteristic, the analogous constructions can be performed by picking a place that one calls ``infinity''. In this case, as in the rest of the paper, one should ignore any mention of semi-algebraic topologies, and work with the usual, honest topologies on the spaces under consideration; rapidly decaying functions become functions which are eventually zero, and asymptotic equalities up to rapidly decaying functions become exact asymptotic equalities in some neighborhood of infinity. Hence, in positive characteristic the theory simplifies considerably, and I leave the details to the reader.

For any torus, say $T$, we denote by the corresponding gothic lowercase letter the vector space $\mathfrak t:= \Hom(\Gm, T) \otimes \RR$. If $F$ is a valued field (or ring), we have a well-defined logarithmic map
\begin{equation}\label{logmap}
\log: T(F)\to \mathfrak t
\end{equation}
given by $\left<\log(t), \chi\right> = \log|\chi(t)|$ for any $\chi\in \Hom(T,\Gm)$.

In this section I define various cosheaves on toric varieties over $\RR$ and on ``equivariant toroidal embeddings'' of $[H]$; these are stratified spaces with a unique open stratum, and carry a ``semi-algebraic'' restricted topology, on which the cosheaf is defined. Their definition requires some familiarity with the theory of toric varieties: I remind the reader here that a normal affine embedding $Y$ of a torus $T$ over a field $k$ is given by a \emph{strictly convex, rational polyhedral cone} $C\subset \mathfrak t$. The faces of this cone are in bijection with $T$-orbits on $Y$, in such a way that cocharacters $\lambda$ in the relative interior of a face are those for which $\lim_{t\to 0} \lambda(t)$ belongs to the corresponding orbit. More general normal embeddings of $T$ are described by \emph{fans} in $\mathfrak t$, i.e., collections of such cones closed under the operation of passing to a face of a cone and with disjoint relative interiors. The notation $[H]^\F$ in the statement of the above theorem refers to a ``equivariant toroidal'' embedding of $[H]^\F$ described by a fan $\F$, as will be explained.

These cosheaves are denoted by the letter $\mathcal F$ (for \emph{functions} --- at some point we multiply by measures and denote them by $\mathcal S$), and by an index $E$ for an ``exponent arrangement''. When we want to clarify the space on which they are defined, we will also put its open orbit (under a given group action) as an index. Thus, on a space $Y$ with open orbit $W$ we will denote by $\mathcal F_E$ or $\mathcal F_{W,E}$ the corresponding cosheaf.

Note that sections of the cosheaf are not, in general, functions on  the space. They can be functions on its open stratum $W$, with ``asymptotically multiplicative'' behavior in the neighborhood of the other strata, or, more generally, they can be sections of some vector bundle over the open stratum with this asymptotic behavior (which can be pulled back to functions on a principal torus bundle $\tilde W\to W$ that are eigenfunctions for the torus).

\subsection{Asymptotically finite functions on toric varieties} \label{finite-toric}

\subsubsection{Cosheaves of multiplicatively finite functions}
Let $T$ be a split torus defined over $\RR$ and let $Y$ be a toric variety for $T$. We assume that $Y$ is normal, but we \emph{do not} assume that $T$ acts faithfully on $Y$: the open $T$-orbit of $Y$ is a principal homogeneous space for a torus quotient $T'$ of $T$; for convenience, we will assume that $T'$ is the quotient of $T$ by a subtorus $T_0$, so that the map on $\RR$-points: $T(\RR)\to T'(\RR)$ is onto. I would advise the reader to consider the case $T=T'$ at first reading.

If $Y$ is smooth, then we have well-defined notions of Schwartz functions on $Y(\RR)$, and Schwartz functions on $T'(\RR)$ are those Schwartz functions on $Y(\RR)$ which vanish, together with all their derivatives, on the complement of $T'(\RR)$.

We can generalize this description to sections of more general vector bundles over $Y(\RR)$ which, however, over $T'(\RR)$ will coincide with the trivial line bundle (i.e., the sections will be functions on $T'(\RR)$) or, more generally, with the vector bundle whose sections are generalized $T_0(\RR)$-eigen\-functions on $T(\RR)$ with specified eigencharacters (``exponents''). We do not need to assume that $Y$ is smooth. 

First, let $E$ be a finite, non-empty multiset of characters of $T(\RR)$. It is sometimes helpful to think of positive, real-valued characters, which, via the logarithmic map \eqref{logmap} and exponentiation: $\RR\to\RR^\times_+$, can be identified with elements of $\mathfrak t^*=\Hom(\mathfrak t, \RR)$.

We will say that a vector in a representation of $T(\RR)$ is a \emph{generalized eigenvector with (multi)set of exponents $E$} if it is annihilated by the operator
$$ \prod_{\chi\in E} (a-\chi(a))$$
for all $a\in T(\RR)$, that is:
\begin{itemize}
\item in the case where all elements of $E$ are equal to the same character $\chi$, it is a generalized eigenvector with eigencharacter $\chi$ and degree less or equal to the multiplicity of $\chi$ in $E$;
\item in the general case, it is a sum of such generalized eigenvectors, corresponding to the distinct characters in $E$.
\end{itemize}

Now fix an orbit $Z\subset Y$. The pointwise stabilizer of $Z$ is a subgroup $T_Z\subset T$, containing $T_0$. A given multiset $E$ of characters of $T_Z(\RR)$ defines a complex vector bundle over $Z(\RR)$, as follows: 
\begin{quote}
Sections of this vector bundle are functions on $T(\RR)$ which are generalized eigenfunctions for $T_Z(\RR)$ with exponents $E$. 
\end{quote}
Sections of this vector bundle will be called \emph{multiplicatively finite functions} on $T(\RR)$ \emph{with multiset of exponents $E$}. There are obvious notions of smooth sections and sections of polynomial growth for this vector bundle, which coincide with the analogous notions for functions on $T(\RR)$. Moreover, there is a notion of \emph{Schwartz sections} of this vector bundle: a section $\sigma$, represented by the function $f_\sigma$ on $T(\RR)$, is Schwartz if for one, equivalently any, smooth semi-algebraic lift (s.\ the following remark) $e:Z(\RR)\to T(\RR)$  the function $e^*f_\sigma$ is Schwartz on $Z(\RR)$. 

Schwartz sections of this vector bundle over $Z(\RR)$ form a strictly flabby cosheaf of nuclear Fr\'echet spaces on $Z(\RR)$ which will be denoted by $\mathcal F_{Z,E}$.

\begin{remark}\label{contraction}
By ``lift'' $Z(\RR)\to T(\RR)$ or $Z(\RR)\to T'(\RR)$, here and later, we mean a section of the contraction map $T(\RR)\to T'(\RR)\to Z(\RR)$ that takes an element of $t\in T'(\RR)$ to the unique element of $Z(\RR)$ that is contained in the $T_Z(\RR)$-orbit closure of $t$. Here we are identifying the open orbit in $Y$ with the torus quotient $T'$ of $T$, but this is only to avoid adding another piece of notation; we might as well distinguish between $T(\RR)$ as the space on which functions live and $T(\RR)$ as the group acting on them. Those sections will always be chosen to be semi-algebraic.

For a general toric variety and an orbit $Z$ contained in it, the contraction map is defined on the affine open neighborhood $Y_Z$ of all orbits which contain $Z$ in their closure. 
In what follows, when we choose a semi-algebraic neighborhood of a point in $Z(\RR)$ it will be assumed to be contained in $Y_Z(\RR)$, so that the contraction map is well-defined, and any mention of a ``lift'' $Z(\RR)\to Y(\RR)$ will implicitly have image in $Y_Z(\RR)$ and be a section of this contraction map.
\end{remark}

\subsubsection{Definition of the cosheaf of asymptotically finite functions} 
\label{ssdefcosheaf}

Now assume that we have an assignment $Z\mapsto E(Z)$ of multisets of exponents for every geometric orbit $Z\subset Y$, with the following property: 

\begin{quote} If an orbit $W$ is contained in the closure of an orbit $Z$, then the restrictions of elements of $E(W)$ to $T_Z(\RR)\subset T_W(\RR)$ are contained, with at least the same multiplicity,\footnote{When different characters in $E(W)$ restrict to the same character of $T_Z(\RR)$, it is enough to assume that its multiplicity in $E(Z)$ is at least the \emph{maximum} of their multiplicities, not their sum.} in $E(Z)$.
\begin{equation} \end{equation} 
\end{quote}

It is instructive to think of what this condition means combinatorially, when (for simplicity) $T'=T$ and all characters are positive real-valued: The embedding $Y$ is described by a fan $\mathfrak F$ of strictly convex, rational polyhedral cones in $\mathfrak t$, corresponding bijectively to (geometric) orbits: $Z\leftrightarrow C_Z$. The linear span of $C_Z$ is precisely the linear span of the cocharacter group of the torus $T_Z$. For each orbit $Z$ we have a multiset $E(Z)$ of characters of $T_Z(\RR)$ which, assumed positive real-valued, can be thought of as functionals on the span of $C_Z$. The condition, then, is that if $W$ is contained in the closure of $Z$, i.e., if $C_Z$ is a face of $C_W$, then all restrictions of elements of $E(W)$ to the linear span of $C_Z$ are contained (with at least the same multiplicity) in $E(Z)$.

We now define the strictly flabby cosheaf $\mathcal F_E$ (in the restricted semi-algebraic topology over $Y(\RR)$) of functions $f$ on $T(\RR)$ which are \emph{asymptotically finite with exponent arrangement $E$}, inductively, as follows:
\begin{itemize}
\item over the open orbit $T'(\RR)$, the cosheaf $\mathcal F_E$ coincides with the cosheaf $\mathcal F_{T',E(T')}$ defined above, i.e., Schwartz sections of the vector bundle of $T_0(\RR)$-generalized eigenfunctions on $T(\RR)$ with exponents $E(T')$;
\item for an arbitrary open set $U$, assuming that $W$ is an orbit of maximal codimension among those meeting $U$, we demand that there is a Schwartz section 
$$\sigma \in \mathcal F_{W, E(W)}(U\cap W(\RR)),$$
and a semi-algebraic neighborhood $V$ of $W(\RR)\cap U$ in $U$, proper (with respect to the contraction map, s.\ Remark \ref{contraction}) over $W(\RR)\cap U$, with the property that $f-f_\sigma$ coincides, in $V\cap T(\RR)$, with the restriction (as a function) of an $f'\in \mathcal F_E(U\smallsetminus W(\RR))$.
\end{itemize}

The two conditions, together with the strict flabbiness of the cosheaf (so that open embeddings of sets give closed embeddings of the spaces of sections) completely define the cosheaf: By strict flabbiness we can always replace an open set by a cover of open sets which meet a unique orbit of minimal dimension, and by the second criterion we can remove this orbit, eventually arriving to an open subset of the open orbit.

\begin{remark}\label{iscompact}
Without loss of generality, we may assume that $Y$ is complete ($Y(\RR)$ is compact); if it is not, we can compactify it, and attach the empty multiset of exponents to the new orbits. For example, Schwartz functions on $T'(\RR)$ are sections over some compactification, when we attach the empty set of exponents to every non-open orbit, and the trivial exponent (character of $T_0(\RR)$) to the open orbit $T'$.
\end{remark}

There is a natural nuclear Fr\'echet space structure on the sections of this cosheaf, with respect to which extension maps are closed embeddings. We use the above inductive definition to define the Fr\'echet space structure on $\mathcal F_E(U)$: First, we have a map to the Fr\'echet space $\mathcal F_{W,E(W)}(U\cap W(\RR))$, which provides some continuous seminorms. Secondly, we may fix $V$ as above and apply the seminorms of the \emph{stalk} $\mathcal F_E(U\smallsetminus W(\RR))/\mathcal F_E(U\smallsetminus V)$ to the function $f-f_\sigma$. Finally, we may apply the seminorms of the stalk $\mathcal F_E(U\smallsetminus W(\RR))/\mathcal F_E(\mathring V\smallsetminus W(\RR))$ to the function $f$. This forms a complete set of seminorms for the space $\mathcal F_E(U)$.

\subsection{Regularized integral on toric varieties}\label{regularized-toric}

Let $T, Y$ be as above, and $E$ an exponent arrangement for $Y$. Assume now that $T=T'$, i.e., $T$ acts faithfully on $Y$. We will say that an exponent $\chi\in E(Z)$ (for some orbit $Z$, other than the open one) is \emph{critical} if $\chi$ is trivial. Assume that $E$ does not contain any critical exponents. Clearly, by the compatibility of exponents for different orbits it is enough to check this condition for orbits $Z$ of codimension one, where $T_Z$ is of dimension one.

We denote by $\mathcal S_E$ the cosheaf of measures on $T_E$ which are equal to elements of $\mathcal F_E$ times a Haar measure. We can define a \emph{regularized integral}, that is, a $T(\RR)$-invariant functional:
$$ \mathcal S_E(Y(\RR))\ni \mu \mapsto \int^*_{T(\RR)} \mu \in \CC,$$
essentially as the Mellin transform of the distribution $\mu$ evaluated at the trivial character.

Since this distribution is not tempered in the analytic sense (i.e., with respect to the Harish-Chandra Schwartz space), I explicate the definition: 

First, any toric variety has a finite open cover by affine toric varieties, and any section $\mu$ of $\mathcal S_E$ can be written as a sum of sections over these open subvarieties. Hence we may, without loss of generality, assume that $Y$ is affine. 

We may then choose an algebraic character $\omega$ of $T$ which vanishes on the complement of $T$ in $Y$ and define:
\begin{equation}\label{reg-toric} \int^*_{T(\RR)^0} \mu= \mbox{the analytic continuation of } \int^*_{T(\RR)^0} |\omega|^s(t) \mu(t) \mbox{ to }s=0.\end{equation}
This integral converges for $\Re(s) \gg 0$ (how large depends on the specific exponents).  

\begin{proposition}\label{invariant-toric}
If the exponent arrangement $E$ does not contain critical exponents, the regularized integral $\mu\mapsto \int^*_{T(\RR)} \mu $ is a well-defined, $T(\RR)$-invariant continuous functional on $\mathcal S_E(Y(\RR))$.
\end{proposition}

\begin{proof}
We can always choose a smooth blowup $\tilde Y\to Y$ of the original variety $Y$. Combinatorially, this corresponds to partitioning the cones of the fan $\F$ into (simplicial) sub-cones $C$ with the property that $C\cap \Hom(\Gm,T')$ is a free monoid, for every $C$. This way, the relative interior of each new cone $C$ belongs to the relative interior of an old cone $D$, hence $T_C\subset T_D$, and we assign to the new fan $\tilde\F$ the exponent arrangement $\tilde E$ obtained from restricting characters from $T_D$ to $T_C$. Since the exponents of the original arrangement were non-critical, we can always choose such a partition where the new arrangement $\tilde E$ remains non-critical. 

Sections of $\mathcal F_E$ over $Y(\RR)$ become, by pull-back, a subspace of the sections of $\mathcal F_{\tilde E}$ over $\tilde Y(\RR)$. Therefore, we may assume that $Y$ is smooth. Hence, it is covered by affine open subsets of the form $\Gm^a \times \Ga^b$ (under the action of $T\simeq \Gm^{a+b}$).

We now proceed to the proof in accordance with the definition of the cosheaf $\mathcal F_E$ in \S \ref{ssdefcosheaf}. On $\mathcal S_E(T(\RR))$ the regularized integral coincides with the proper, absolutely convergent integral, and therefore there is nothing to prove in this case. Now fix a $T$-orbit $Z$, let $Y_Z$ be the affine open of all orbits containing it in their closure, as before, and assume that the proposition has been proven for $\mathcal S_E((Y_Z\smallsetminus Z)(\RR))$. Let us extend it to $\mathcal S_E(Y_Z(\RR))$. 

A measure $\mu \in \mathcal S_E(Y_Z(\RR))$ is of the form $f dt$, where $dt$ is a Haar measure on $T(\RR)$, and $f\in \mathcal F_E(Y_Z(\RR))$.  By the definition of $\mathcal F_E(Y_Z(\RR))$ in \S \ref{ssdefcosheaf}, there is a semi-algebraic open neighborhood $V$ of $Z(\RR)$, proper over $Z(\RR)$, and a function $f_\sigma$ on $T(\RR)$ corresponding to a Schwartz section $\sigma$ of $\mathcal F_{Z,E(Z)}(Z(\RR))$, such that $f-f_\sigma$ coincides on $V\cap T(\RR)$ with the restriction of an element of $\mathcal S_E((Y_Z\smallsetminus Z)(\RR))$. Without loss of generality, if $\omega_1, \dots, \omega_n$ is a set of characters generating the ideal of $Z$ in (the affine variety) $Y_Z$, we may assume that $V$ is given by inequalities: $|\omega_i(y)|< \epsilon$. 

In particular, by our induction assumption, the regularized integral
$$ \int^*_{V\cap T(\RR)} (f-f_\sigma)(t) dt$$
is well-defined, and continuous with respect to the seminorms of $\mathcal S_E(Y_Z(\RR))$.

Having assumed smoothness, we have $Y_Z \simeq \Gm^a \times \Ga^b$, with $Z = \Gm^a \times\{0\}^b$, and we may assume that the $\omega_i$'s are the coordinate functions for $\Ga^b$ ($n=b$). From the definitions, we clearly have
$$ \mathcal F_{Z,E(Z)}(Z(\RR)) \simeq \mathcal F((\RR^\times)^a) \otimes W_{E(Z)},$$
where $\mathcal F((\RR^\times)^a)$ denotes, as before, the usual space of Schwartz functions, and $W_{E(Z)}$ is the finite-dimensional space of generalized eigenfunctions on $T_Z(\RR)$ with the given multiset of exponents. Thus, the regularized integral of $f_\sigma$ over $V$ is the product of a usual integral of a Schwartz function on $(\RR^\times)^a$ with the regularized integral of a generalized eigenfunction on $T_Z(\RR)\cap V$, which is easily seen to be well-defined in the absence of critical exponents.

Let us now prove invariance of the regularized integral assuming, as in the definition, that $Y$ is affine. Using the notation of the definition and writing $\operatorname{cont}_{s=0}$ for the analytic continuation to $s=0$ we have, for any $a\in T(\RR)$: 
$$ \int^* (a\cdot f - f) dt = \operatorname{cont}_{s=0} \int_{T(\RR)} (f(at) - f(t) ) |\omega(t)|^s dt =$$
$$ = \operatorname{cont}_{s=0} (|\omega(a)|^{-s} -1) \int_{T(\RR)} f(t) |\omega(t)|^s dt = $$
$$ = \left(\operatorname{cont}_{s=0} (|\omega(a)|^{-s} -1)\right) \int_{T(\RR)}^* f(t) dt = 0.$$

\end{proof}

\begin{remark}
Depending on the exponent arrangement, asymptotically finite functions may include functions which are eigenfunctions with respect to a subgroup of $T(\RR)$. But in this case the regularized integral will be zero, unless the corresponding eigencharacter is trivial --- hence critical --- on a maximal such subgroup, in which case the regularized integral is not defined.
\end{remark}

\subsection{Equivariant toroidal compactifications of the automorphic quotient}

Now let $H$ be a connected reductive group over $k=\QQ$. In this subsection we will describe compactifications of the automorphic quotient $[H]=H(k)\backslash H(\adele)$ which are refinements of the reductive Borel--Serre compactifications of \cite{BJ}.   
I do not know of any reference in the literature for these toroidal compactifications, and will therefore present them from scratch, though without many details.

We denote $\mathfrak a = \Hom(\Gm,A) \otimes \RR$, where $A$ is the universal maximal split torus of $H$, and we denote by $\mathfrak a^+$ the anti-dominant cone. Faces of $\mathfrak a^+$ correspond to conjugacy classes of parabolic subgroups defined over $k$, and the span of the face associated to $P$ will be denoted by $\mathfrak a_P$, the corresponding subtorus of $A$ by $A_P$. (It is the maximal split central subtorus of the Levi quotient of $P$.)

For any class of parabolics $P$, we have the ``boundary degeneration''
$$[H]_P = M(k)N(\adele)\backslash H(\adele) = [M]\times^{P(\adele)} H(\adele),$$
 where $N$ is the unipotent radical of $P$ and $M$ is its reductive quotient. Notice that any two parabolics defined over $k$ are $H(k)$-conjugate, and self-normalizing. Therefore, the space $[H]_P$ is defined up to unique isomorphism by the class of $P$, not a choice of parabolic in this class, but we will remember that it does not have a distinguished point represented by $1\in H$ (which would depend on the choice of representative for $P$). The same is true for the space $P(k)\backslash H(\adele)$ that will be used later.

Let $\F$ be a fan of strictly convex, rational polyhedral cones, whose support is contained in $\mathfrak a^+$. The relative interior of any cone $C$ of $\F$ is contained in the relative interior of a unique face of $\mathfrak a^+$, and hence corresponds to a unique class of parabolics $P_C$; we will say that the cone $C$ \emph{belongs} to the class $P_C$. The corresponding subtorus $A_C$ (the subtorus whose cocharacter group spans the linear span of $C$) belongs to the center of the Levi quotient of $P_C$.

We will use the fan to define a topological, $H(\adele)$-equivariant compactification $[H]^\F$ of $[H]=H(k)\backslash H(\adele)$, obeying similar closure relations as for toric varieties: every cone $C\in \F$ corresponds to a locally closed stratum ($H(\adele)$-orbit) $Z_C$ of $[H]^\F$, we have $Z_{C_1}\subset \overline{Z_{C_2}} \iff C_1 \supset C_2$, etc. The stratum $Z_C$ corresponding to a cone $C$ in the fan is isomorphic to the quotient of $[H]_{P_C}$ by the action of $A_C(\RR)$; for example, if $C$ is of full dimension, then $A_C=A$, $P_C$ is the minimal parabolic $P_0=M_0 N_0$ and $Z_C$ is isomorphic to the (compact) quotient $M_0(k) A(\RR) N_0(\adele)\backslash H(\adele)$.

\begin{remark} The compactification is suitable for describing functions that vary asymptotically according to $A_C(\RR)$-characters (in the vicinity of the stratum $Z_C$). This is not ideal, but notationally convenient, and sufficient for our purposes. In our application, actually, these characters will be real-valued and trivial on the maximal compact (finite) subgroup of $A_C(\RR)$ and hence will factor through the log map \eqref{logmap}
$$ A_C(\RR)\ni a \mapsto \log(a)\in \mathfrak a_{C,\RR} = \Hom(\Gm,A_C)\otimes\RR.$$

There are two alternatives which are philosophically preferable: One is to compactify dividing only with respect to the connected component of the identity in $A_C(\RR)$; this is suitable to encode the minimal information needed for the definition of regularized integrals, but would be notationally a bit more cumbersome. Another extreme is to compactify dividing by $A_C(\adele)$, or even by the analogous torus when we take $A$ to be the universal Cartan of $H$, not just its maximal split subtorus. This is suitable for encoding the maximal amount of information on asymptotically finite functions, should they be asymptotically equal to $A_C(\adele)$-eigenfunctions; however, it would not add anything to the applications of the present paper, where characters are trivial on the maximal compact subgroup of $A_C(\adele)$.
\end{remark}

We actually use $\F$ to define embeddings not only for $[H]$ itself but also for its ``boundary degenerations'' $[H]_P$. They also have their orbits parame\-trized by cones in $\F$, so here we need to make a notational distinction: The notation $Z_C$ for the stratum associated to $C$ will be reserved for the corresponding stratum in $[H]^\F_P$ \emph{for all parabolics $P$ containing $P_C$}; in fact, the strata $Z_C$ for all those embeddings will be identified. On the other hand, if the parabolic $P$ does not contain (a parabolic in the class of) $P_C$, then the stratum of $[H]_P^\F$ corresponding to $C$ will be \emph{different} from $Z_C$, and we will not reserve a symbol for it.

Assume first that $\F$ is a fan supported entirely on the face of the anti-dominant cone corresponding to a parabolic $P$ with Levi quotient $M$, whose split center $A_P$ is canonically a subtorus of $A$. Then $\F$ defines a toric variety $Y_P$ for $A_P$, and we set
\begin{equation}\label{MF} [M]^{\F}:= [M]\times^{A_P(\RR)} Y_P(\RR),\end{equation}
\begin{equation}\label{HF} [H]_P^{\F}:= [M]^{\F} \times^{P(\adele)} H(\adele).\end{equation}

For a general fan $\F$, the $H(\adele)$-space $[H]_P^{\F}$ will be defined by the formula \eqref{HF}, once $[M]^{\F}$ is defined. To define $[M]^\F$, we may assume that $M=H$, and that the spaces $[M]^{\F}$, $[H]_P^{\F}$ have been defined for all proper parabolics $P$.

We first consider the restriction $\F_H$ of $\F$ to $\mathfrak a_H$ (i.e., the sub-fan consisting of all cones which are contained in $\mathfrak a_H$, the span of central cocharacters into $H$). By \eqref{MF}, it defines an embedding $[H]^{\F_H}$ of $[H]$. Now, all the strata $Z_C$ have been defined: If $C$ belongs to $H$, then $Z_C\subset [H]^{\F_H}$. If not, then $Z_C$ has been defined as a stratum of $[H]_P^\F$, for all $P\supset P_C$. It remains to explain how to glue those onto $[H]^{\F_H}$. 

Equivalently, we should define when a sequence $(z_n)_n$ in $\bigsqcup_D Z_D$ (this includes the trivial cone $D=\{0\}$ for which $Z_C = [H]$) converges to a point $z$ of the stratum $Z_C$. (The topology will be separable, of course.) We can restrict to cones $D$ which are faces of $C$, since otherwise the orbit $Z_D$ will not contain $Z_C$ in its closure. If the sequence belongs to a stratum $Z_D$, with $D\subset C$, $D$ belonging to a proper parabolic $Q=P_D$, then both $Z_D$ and $Z_C$ can be considered as sub-strata of $[H]_Q^{\F}$, and convergence is defined there. There remains to consider the case that $(z_n)_n \subset [H]^{\F_H}$.

For this, consider the pair of quotient maps:
\begin{equation}\label{leftright} \xymatrix{
& P(k)\backslash H(\adele) \ar[dl]_{\pi_H}\ar[dr]^{\pi_P}& \\
[H] && [H]_P.
}\end{equation}

By ``the cusp'' in $[H]_P$ we will mean the limit of all one-parameter orbits of the form $M(k)N(\adele)\lambda(t) g$, as $t\to 0$, where $\lambda$ is a strictly $P$-anti-dominant cocharacter into the center of the Levi quotient $M$ of $P$ (and $N$ is the unipotent radical of $P$). Equivalently, the cusp is the closed orbit in $[H]_P^\F$, when $\F$ is the fan of all faces of $\mathfrak a^+\cap \mathfrak a_P$.
 We will call ``neighborhood of the cusp'' in $P(k)\backslash H(\adele)$ the preimage of any neighborhood of the cusp in $[H]_P$.

There is a neighborhood of the cusp in $P(k)\backslash H(\adele)$ where the left arrow of \eqref{leftright} is an isomorphism onto its image --- its image will be called a ``neighborhood of the $P$-cusp'' in $[H]$. If we apply the operation
$$\times^{A_H(\RR)} Y_H(\RR)$$ 
(recall that this operation on $[H]$ produces $[H]^{\F_H}$), then we get corresponding neighborhoods of the $P$-cusp in $[H]^{\F_H}$: 
\begin{equation}\label{leftrighttimes} \xymatrix{
& P(k)\backslash H(\adele) \times^{A_H(\RR)} Y_H(\RR) \ar[dl]\ar[dr]& \\
[H]^{\F_H} && [H]_P \times^{A_H(\RR)} Y_H(\RR).
}\end{equation}

Now, $[H]_P \times^{A_H(\RR)} Y_H(\RR)$ is an open subspace of $$[H]_P^{\F_P}= [H]_P \times^{A_P(\RR)} Y_P(\RR),$$
and the latter contains the stratum $Z_C$ that we are interested in. Notice that since $C$ belongs to $P$, this stratum is ``in the cusp'', i.e.: every open neighborhood of the cusp contains an open neighborhood of $Z_C$. (This fails to be true for orbits in $[H]_P^{\F_P}$ attached to cones which do not meet the relative interior of the $P$-face of $\mathfrak a^+$.)

We declare that the sequence $(z_n)_n\subset [H]^{\F_H}$ converges to $z\in Z_C$ if it can be eventually lifted (under the left arrow of \eqref{leftrighttimes}) to a sequence $(\tilde z_n)_n$ in the neighborhood of the cusp in $P(k)\backslash H(\adele) \times^{A_H(\RR)} Y_H(\RR)$, whose image in $[H]_P \times^{A_H(\RR)} Y_H(\RR)$ converges to $z$.

This completes the description of the compactification $[H]^{\F}$.

\begin{remark}
Returning to the discussion of strata associated to a cone $C$ in the fan $\F$, we notice once again that, by definition, the stratum $Z_C$ is identified as a stratum of $[H]_Q^\F$, for every $Q\supset P_C$. On the other hand, if $Q$ does not contain $P$, there is a stratum attached to $C$ in $[H]_Q^\F$, but it is not isomorphic to $Z_C$. Rather, it is isomorphic to the quotient of $[H]_{P\cap Q}$ by $A_C(\RR)$. 
\end{remark}

\begin{remark}
Assume that $G$ is semisimple. Then there is a canonical fan $\F$ on $\mathfrak a^+$, consisting of all faces of $\mathfrak a^+$. One can check that in that $[H]^\F$ one gets the \emph{reductive Borel--Serre compactification} $[H]^{\rm RBS}$ of \cite{BJ}, with the only difference\footnote{Of course, since \cite{BJ} was not written in the adelic language, literally speaking to obtain their compactification from ours one needs to divide by an open compact subgroup of the finite adeles of $H$, which amounts to choosing an arithmetic subgroup of $H(\RR)$.} that here (for notational simplicity) we have defined the compactifications \eqref{MF} by using $A_P(\RR)$, while in \cite{BJ} they just use the identity component $A_P(\RR)^0$. In other words, to obtain the compactification of \cite{BJ} one would need to replace in \eqref{MF} the torus $A_P(\RR)$ by $A_P(\RR)^0$, and the toric variety $Y_P(\RR)$ by the closure of $A_P(\RR)^0$ in it. With this minor modification, it is easy to see that the reductive Borel--Serre compactification is final among equivariant toroidal compactifications, that is: for every cone $\F$ the identity map on $[H]$ extends to an $H(\adele)$-equivariant map $[H]^\F \to [H]^{\rm RBS}$. Verifying these claims is just a matter of going through the definitions, and using the following standard fact in toric geometry: there is a map between two normal embeddings of the same torus, with fans $\F_1$ and $\F_2$ if and only if each cone of $\F_1$ is contained in a cone of $\F_2$.
\end{remark}

Although we have described the space $[H]^\F$ as a Hausdorff (it can easily be checked) topological space, it is a ``semi-algebraic'' restricted topology that will be more useful for us. First, notice that quotients of $[H]$ by any open subgroup $J$ of the finite adeles of $H$ are Nash manifolds, in a way compatible with each other, covered, e.g., by open Siegel domains in $H(\RR)$ (which are clearly Nash manifolds); thus, we have ``open semi-algebraic subsets'' as pull-backs of those of $[H]/J$, for some $J$.

The $\RR$-points of the toric variety $Y_H$ have a semi-algebraic restricted topology, and hence the space
$$[H]^{\F_H} = [H]\times^{A_H(\RR)} Y_H(\RR)$$
also inherits a restricted semi-algebraic topology.

Similarly, there is a semi-algebraic restricted topology on $[H]_P$ and $[H]_P^{\F_P}$. In a neighborhood of the $P$-cusp, the left arrow of \eqref{leftright} is an isomorphism and the right arrow is semi-algebraically continuous and open (again, think of Siegel domains). Thus, the same is true for \eqref{leftrighttimes}, and this is enough to define what a semi-algebraic open neighborhood of a point $z\in Z_C$ is: Start with an open neighborhood in $[H]_P^{\F_P}$, which can be assumed to be sufficiently close to the cusp (since $C$ belongs to $P$), pull it back and push it forward to $[H]^{\F_H}$ via  \eqref{leftrighttimes}. (The intersections of this neighborhood with a stratum $Z_D$ with $D\subset C$ belonging to a proper parabolic $Q$ can be defined in $[H]_Q^\F$.) 

\subsection{Asymptotically finite functions on the automorphic quotient}\label{ssaf}

We remain in the setting of $[H]$ as above.

An \emph{exponent arrangement} will consist of a fan $\F$ as above, i.e., with support in the anti-dominant cone $\mathfrak a^+$, and an assignment, to every cone $C$ of $\F$,  of a multiset $E(C)$ (i.e., repetitions allowed) of linear functionals (the ``exponents'') $\chi\in \mathfrak a_C^* = \Hom(\mathfrak a_C,\RR)$, where $\mathfrak a_C$ denotes the linear span of the cone $C$ in $\mathfrak a$. 
As in \S \ref{finite-toric}, these multisets need to have the following property: If a cone $C_1$ is in the closure of the cone $C_2$, then the restrictions of elements of $E(C_2)$ to $\mathfrak a_{C_1}$ are contained in $E(C_1)$ (with at least the same multiplicity). 

More generally, we may, as we did in the toric case, require our functions to have central (generalized) characters. That is, instead of taking a fan $\F$ on $\mathfrak a$, we may take the fan to be on $\mathfrak a/\mathfrak a_0$, where $\mathfrak a_0\subset \mathfrak a_H$ corresponds to a split central torus $A_0$ of $H$. In that case, our functions will be sections of a cosheaf over a toroidal compactification of $A_0(\RR)\backslash [H]$, but the exponents are still allowed to be arbitrary functionals on $\mathfrak a$, i.e., non-trivial on $\mathfrak a_0$. To keep notation simple, I will not present this case; the generalizations are immediate.

These linear functionals will be considered as positive characters of the real, or adelic, points of the corresponding subtorus of $A$ --- the subtorus $A_C$ whose cocharacter group spans $\mathfrak a_C$ --- as follows: every such functional $\chi$ can be written as a linear combination $\sum c_i\chi_i$, where the $\chi_i$ are algebraic characters of $A_C$ and $c_i\in\mathbb R$, and then we define (and denote by the same letter):
\begin{equation}\label{ascharacters} \chi: A_C(\adele)\ni a\mapsto \prod_i |\chi_i(t_i)|^{c_i} \in \RR_+^\times.\end{equation}
Equivalently, $\chi(a) = e^{\log a}$, where $\log$ is the map \eqref{logmap}. 
Obviously, the discussion here can be extended to arbitrary complex characters, but positive ones are sufficient for our purposes, and notationally simpler since they can be identified with elements of $\mathfrak a_C^*$.

We will define a cosheaf $\Ff_E$ of ``asymptotically finite functions'' for the restricted semi-algebraic topology over each of the spaces $[H]_P^{\F}$, as we did for toric varieties.

First, fix a cone $C\subset \F$; recall that it corresponds to a stratum $Z_C\subset [H]^{\F}$. Let $P=P_C$ be the parabolic associated to $C$; we have inclusions of tori: $A_C\subset A_P \subset A$. The multiset $E(C)$ of characters of $A_C(\RR)$ defines a complex vector bundle over $Z_C$, whose sections are functions on $[H]_P$  which are generalized eigenfunctions for $A_C(\RR)$ with exponents $E(C)$.
As in the toric case, there are obvious notions of smooth sections, sections of polynomial growth, and Schwartz sections of this bundle on $Z_C$. The cosheaf over $Z_C$ of Schwartz sections of this vector bundle will be denoted by $\mathcal V_{Z_C,E(C)}$.

Now we define the strictly flabby cosheaf $\Ff_E$ (in the semi-algebraic restricted topology) over $[H]_P^{\F}$ 
of functions $f$ which are \emph{asymptotically finite with exponent arrangement $E$} inductively:
\begin{itemize}
\item We apply any of the definitions that follow to any $M$, instead of $H$, where $M$ is  the reductive quotient of any parabolic $P$, and $\F_H$ is replaced by $\F_E$; if $\Ff_E$ is defined\footnote{To avoid heavy notation, we use the same notation $\Ff_E$ for $[H]$, $[M]$ and $[H]_P$; this should not cause any confusion.} for $[M]^\F$, we define it for $[H]_P^\F$ by smooth induction from $P(\adele)$ to $H(\adele)$.
\item Over an open set $U$ which is contained in $[H]$, the space $\Ff_E(U)$ coincides with the space $\mathcal F(U)$ of Schwartz (rapidly decaying together with their polynomial derivatives) functions on $U$; more generally, over $[H]^{\F_H}$, sections of the cosheaf are obtained as push-forwards via the quotient map
$$[H]\times^{A_H(\RR)} Y_H(\RR) \to [H]^{\F_H}= [H]\times^{A_H(\RR)} Y_H(\RR)$$
of elements of
$$ \mathcal F([H]) \hat\otimes \mathcal S_E(Y_H(\RR)),$$
where $\mathcal S_E$ is the corresponding cosheaf of asymptotically finite measures on $Y_H(\RR)$, defined in \S \ref{finite-toric}, \ref{regularized-toric} (with the exponent arrangement $E$ restricted to the sub-fan $\F_H$).
\item Given a cone $C$ belonging to a parabolic $P$, we may now by induction assume that $\Ff_E(U)$ has been defined when $U$ is any open subset: 
\begin{enumerate}[(i)]
\item of $[H]_P^\F$;
\item of $[H]^\F$, but not meeting any orbit of codimension $\ge \codim(Z_C)$.
\end{enumerate}
Now, let $U$ be an open neighborhood $Z_C$ in $[H]^\F$, which we may assume not to meet any orbits of codimension $\ge \codim(Z_C)$ other than $Z_C$, and to be sufficiently close to the $P$-cusp, so that its intersection with $[H]$ can be identified, via \eqref{leftright}, with a neighborhood $U'$ of the cusp in $P(k)\backslash H(\adele)$. The image of $U'$ in $[H]_P$ is the intersection with $[H]_P$ of an open neighborhood $U''$ of $Z_C$ in $[H]_P^\F$. For $f$ to be a section of $\Ff_E(U)$, we demand that there is are sections $f'\in\Ff_E(U\smallsetminus Z_C)$, $f''\in \Ff_E(U'')$ such that the difference $\pi_H^* f - \pi_H^* f' - \pi_P^* f''$ on $U'$ is equal to a rapidly decaying function on $U'$ (i.e., an element of $\mathcal F(U')$).
\end{itemize}

There is a natural nuclear Fr\'echet space structure on the sections of this cosheaf, which is defined as in the toric case.

\subsection{Regularized integral of asymptotically finite automorphic functions}\label{regularized-automorphic}

Let $(\F, E)$ be an exponent arrangement for $[H]$, and let $f$ be a section of $\Ff_E$ over $[H]^\F$. We will say that a character (``exponent'') $\chi\in E(C)$, for some cone $C\in \F$, is \emph{critical} if it coincides with the modular character  of $P_C$, restricted to $A_C(\RR)$, and $A_C$ is non-trivial. (The modular character is the character by which $P_C$ acts on the quotient of its right by its left Haar measure; since $A_C$ is in the center of the corresponding Levi, we can equivalently take the modular character of the minimal parabolic.)

Assume that there are no critical exponents in $E$. Then we can define on $\Ff_E([H]_P^\F)$ a regularized integral, i.e., a continuous, $H(\adele)$-invariant extension of the integral, against some fixed invariant measure on $[H]_P$\footnote{Of course, we will be using Tamagawa measure, but see the discussion following \ref{ev-regular} for the cases in which Tamagawa measure needs to be regularized; we assume throughout compatible choices on all the spaces $[H]_P$, with Tamagawa measure on unipotent radicals.}, on the space of Schwartz functions on $[H]$:
$$ \Ff_E([H]_P^\F) \ni f \mapsto \int_{[H]_P}^* f(h) dh \in \CC,$$
as in the toric case.
 Namely, the definition is parallel to the definition of sections of $\Ff_E$, and goes as follows:
\begin{itemize}
\item For a section $f\in \Ff_E([H]_P^\F)$, we fix a factorization of the measure:
$$ \int_{[H]_P} dh = \int_{P\backslash H(\adele)} \int_{[M]} \delta_P(m) dm dh,$$
and define accordingly:
$$ \int_{[H]_P}^* dh = \int_{P\backslash H(\adele)} \int^*_{[M]} \delta_P(m) dm dh$$
once the regularized integral $\int^*_{[M]}$ has been defined; we can apply all the definitions that follow to any $M$, instead of $H$.
\item On $\Ff_E ([H]^{\F_H})$ the regularized integral is induced from the regularized integral on $\mathcal S_E(Y_H(\RR))$: Recall from \eqref{MF} that $[H]^{\F_H}$ is a quotient of $[H]\times Y_H(\RR)$; we thus have a surjective map
$$\mathcal F([H]) \hat\otimes \mathcal S_E(Y_H(\RR)) \twoheadrightarrow \mathcal F_E([H]^{\F_H})$$
which coincides with taking coinvariants with respect to the $A_H(\RR)$-action, and when pulled back to $\mathcal F([H]) \hat\otimes \mathcal S_E(Y_H(\RR)) $ the regularized integral $\int^*_{[H]}$ will be the functional
$$\int_{[H]} dh\times \int_{A_H(\RR)}^*:  \mathcal F([H]) \hat\otimes \mathcal S_E(Y_H(\RR)) \to\CC.$$
The same definition applies to the space $[M]^{\F_P}$, and hence to $[H]_P^{F_P}$ since sections over $[H]_P^{F_P}$ are simply obtained by smooth (compact) induction from $P(\adele)$ to $H(\adele)$.
\item Given a cone $C$ belonging to a parabolic $P$, we may now by induction assume that the regularized integral has been defined  on $\Ff_E(U)$ when $U$ is any open subset: 
\begin{enumerate}[(i)]
\item of $[H]_P^\F$;
\item of $[H]^\F$, but not meeting any orbit of codimension $\ge \codim(Z_C)$.
\end{enumerate}
Now let $U$ be an open neighborhood $Z_C$ in $[H]^\F$, which we may assume not to meet any orbits of codimension $\ge \codim(Z_C)$ other than $Z_C$, and to be sufficiently close to the $P$-cusp, so that its intersection with $[H]$ can be identified, via \eqref{leftright}, with a neighborhood $U'$ of the cusp in $P(k)\backslash H(\adele)$. With $U', U'', f', f''$ as in the definition of sections at the end of \S \ref{ssaf}, we define
\begin{equation} \int^*_U f dh := \int^*_U f' dh + \int^*_{U''} f'' dh + \int_{U'} (\pi_H^* f - \pi_H^* f' - \pi_P^* f'') dh.\end{equation}
The first regularized integral is defined by the induction hypothesis, since $f'$ is a section of $\Ff_E$ over $U\smallsetminus Z_C$, the second is also defined by the induction hypothesis since $f''$ lives on $[H]_P^{\mathfrak F}$, and the third integrand is of rapid decay.
\end{itemize}

The following can easily be deduced from the corresponding Proposition \ref{invariant-toric} for toric varieties:

\begin{proposition}\label{invariant}
The regularized integral $\int_{[H]}^*dh$ is a well-defined, $H(\adele)$-invariant continuous functional on $\Ff_E([H]^\F)$.
\end{proposition}

\subsection{Regularized orbital integrals for group representations} \label{ssregularized}

Having described the equivariant toroidal compactifications and asymptotically finite functions on the automorphic quotient, we are ready to formulate and prove the main result of this section, which concerns the sum of ``nilpotent'' orbital integrals for an algebraic, finite-dimensional group representation $V$ of a reductive group $H$ over $k$. We assume that $H$ acts faithfully on $V$. Let $N\subset V$ denote the \emph{nilpotent cone}, i.e., the closed subvariety of elements which contain $0$ in their $H$-orbit; equivalently, the preimage of the image of $0$ under the quotient map $V\to V\sslash H$. For a Schwartz function $f$ on $V(\adele)$, let
$$\Sigma_N f(h) = \sum_{\gamma \in N(k)} f(\gamma h),$$
a function on $[H]$. We are interested in defining the regularized integral
$$ \int_{[H]}^* \Sigma_N f(h) dh,$$
which is formally equal to the sum \eqref{orbital} of orbital integrals for all rational nilpotent orbits.

The character by which a group acts on invariant volume forms (or Haar measures) on a vector space will be called the ``modular character'' of this vector space.

As before, let $A$ be the universal maximal split torus in $H$, defined as the split part of the center of the reductive quotient of a minimal parabolic. (For any two choices of minimal parabolics, any element of $H$ conjugating one to the other induces a canonical isomorphism between these tori.)

Fix a minimal parabolic $P_0$, a Levi subgroup $M_0$, and identify the maximal split torus in the center of $M_0$ with the universal maximal split torus $A$ of $H$. The action of $A$ on $V$ splits into a direct sum of eigenspaces; let $\Phi(V)$ be the multiset of eigencharacters, with multiplicity equal to the dimension of the associated eigenspace. As a subset of the character group of $A$, it does not depend on the minimal parabolic or its Levi subgroup chosen. 

We consider the partition of the cone 
$$\mathfrak a^+\subset\mathfrak a = \Hom(\Gm,A)\otimes \RR$$ of anti-dominant coweights determined by the following hyperplanes:
\begin{itemize}
\item the walls of the cone, i.e., hyperplanes perpendicular to (simple) roots;
\item the hyperplanes perpendicular to elements of the character set $\Phi(V)$.
\end{itemize}

The intersection of these hyperplanes is trivial, since $H$ is assumed to act faithfully on $V$. Thus, the cones contained in this partition, together with their faces, form a fan $\F$ of rational, strictly convex polyhedral cones on the vector space $\mathfrak a$ whose support is $\mathfrak a^+$. Now we define an exponent arrangement $C\mapsto E(C)$ for this fan as follows: for every cone $C$ we let $E(C)$ consist of a single element, which is the restriction to the subspace $\mathfrak a_C$ spanned by $C$ of the character
\begin{equation}\label{chiC}\chi_C:=-\sum \chi_i,\end{equation}
where $\chi_i$ runs over all elements of $\Phi(V)$ which are positive on the relative interior of $C$. Notice that algebraic characters of $A$ are canonically elements of the dual $\mathfrak a^*$ of $\mathfrak a$, but we can, and will, also consider elements of $\mathfrak a^*$ as positive characters of the real or adelic points of $A$ by \eqref{ascharacters}.

To understand what the exponents above represent, 
let us consider the decomposition 
\begin{equation}\label{decomp-lambda} V=V_{\lambda,-}\oplus V_{\lambda,0}\oplus V_{\lambda,+},\end{equation}
for any given cocharacter $\lambda:\Gm\to H$, where $\Gm$ acts via $\lambda$ on each of the three summands with negative, trivial, or positive weights, respectively. 

If $\lambda$ is a cocharacter into $A$, in the relative interior of the cone $C$, and we consider it as a cocharacter into $H$ via the composition: $\Gm\to A \to M_0\to H$, then the pull-back of $\chi_C$ to $\Gm$ is equal to the inverse of the modular character for the action of $\Gm$ on $V_{\lambda,+}$ via $\lambda$.

Notice that for every $\lambda$ in the relative interior of a cone $C\in \F$, the decomposition \eqref{decomp-lambda} is the same, so we may write:
\begin{equation}\label{decompC} V=V_{C,-}\oplus V_{C,0}\oplus V_{C,+}.
\end{equation}
(We notice, though, that these subspaces vary by the $H(k)$-action when a different pair $(P_0,M_0)$ is chosen.)

\begin{theorem}\label{isfinite}
For any Schwartz function $f$ on $V(\adele)$, the function 
\begin{equation}\label{sigmaN}\Sigma_N f: h\mapsto \sum_{\gamma \in N(k)} f(\gamma h)\end{equation} on $[H]$ is asymptotically finite.

More precisely, this map represents a continuous map from the space of Schwartz functions $\mathcal F(V(\adele))$ to the space $\Ff_E([H]^\F)$ of asymptotically finite functions on $[H]$ with the exponent arrangement $(\F, E)$ as above. 
\end{theorem}

We will outline the proof of this theorem below, and complete it in Appendix \ref{app:finite}.

\begin{example}
As a fast check that the signs of the exponent arrangement that we defined are correct, let us consider the easy case of $\Gm=\GL_1$ with its standard representation $V$, whose weight we will denote by $\chi_0$. We identify $\mathfrak a=\mathfrak a^+$ with $\RR$ by evaluating (linear functionals) at $\chi_0$, and then the fan consists of the three sets $\{0\}, \RR_{\ge 0}$ and $\RR_{\le 0}$. The exponents are $-\chi_0$ on $\RR_{\ge 0}$ and $0$ on $\RR_{\le 0}$. If $\Phi$ is a Schwartz function on $V(\adele)=\adele$, then the function $\Sigma_N\Phi(a) = \sum_{\gamma\in V(k)} \Phi(\gamma a)$ on $\adele^\times$ is asymptotically equal (up to a rapidly decaying function) to $\Phi(0)$ when $|a|\ll 1$. On the other hand, by Poisson summation formula (as in Tate's thesis) it is asymptotically equal to $|a|^{-1} \check\Phi(0)$ when $|a|\gg 1$. This matches the behavior described by the above exponents.
\end{example}

Theorem \ref{isfinite} and Proposition \ref{invariant} allow us to define a regularized nilpotent orbital integral for Schwartz functions on $V(\adele)$, when all exponents of the arrangement $E$ are non-critical. I recall from \S \ref{regularized-automorphic} that a character of $A_C(\RR)$ is ``critical'' if it is equal to the restriction to $A_C(\RR)$ of the modular character of the minimal parabolic (or, equivalently, of $P_C$).

\begin{corollary}\label{corNintegral}
Assume that the exponents of the arrangement $E$ are non-critical. Then the map
\begin{equation}\label{Nintegral}
f\mapsto \int^*_{[H]} \Sigma_N f(h) dh
\end{equation}
is a well-defined, continuous, $H(\adele)$-invariant functional on the space $\mathcal F(V(\adele))$ of Schwartz functions on $V$. 
\end{corollary}

The definition depends on the choice $dh$ of invariant measure on $[H]$. While we have already specified that we will be using Tamagawa measure, there are cases (more precisely, when $H$ has non-trivial, $k$-rational characters) where the usual regularization of Tamagawa measure (using residues of zeta functions) is rather arbitrary and non-canonical. In the next section, when we define ``evaluation maps'' on measures, this ambiguity will be absorbed in the passage from measures to functions. 

\begin{proof}[Beginning of the proof of Theorem \ref{isfinite}]
In preparation for the proof of Theorem \ref{isfinite}, I present the crux of the argument, which is also at the heart of other methods of regularization of orbital integrals, cf.\ \cite{Zydor1,Zydor2}.

First, we may replace the sum over $N(k)$ in the definition of $\Sigma_N$ by a sum over $V(k)$. Indeed, consider the quotient map: $V\to \cc_V=V\sslash H$; the preimage of the image of zero is $N$. Without changing the sum \eqref{sigmaN}, we may multiply $f$ by the pull-back of a ``Schwartz function'' on $\cc_V(\adele)$ whose support meets $\cc_V(k)$ only at the image of zero. (I have put ``Schwartz function'' in quotation marks because $\cc_V$ is not necessarily smooth, but we can embed it in a larger, smooth variety and consider the restriction of a Schwartz function on that.)

Thus, it is enough to prove the theorem when the sum over $N(k)$ in \ref{sigmaN} is replaced by the corresponding sum over $V(k)$; the corresponding function will be denoted by $\Sigma_V f$.

The essence of the argument is its following simplistic version, which shows that the function has the claimed behavior along a cocharacter $\lambda:\Gm\to H$; let us fix such a cocharacter, and split $V$ as in \eqref{decomp-lambda}.

Then we have (using Tamagawa measures on all integrals):
\begin{eqnarray}\label{lambda}\nonumber \sum_{\gamma\in V(k)} f(\gamma h) = \sum_{\gamma_-\in V_{\lambda,-}(k)\smallsetminus\{0\}} \left( \sum_{\gamma_1\in (V_{\lambda,0}+V_{\lambda,+})(k)} f((\gamma_-+\gamma_1)h)\right)+ \\
\nonumber +\sum_{\gamma_0\in V_{\lambda,0}(k)} \left[\left(\sum_{\gamma_+\in V_{\lambda,+}(k)} f((\gamma_0+\gamma_+) h) - \int_{V_{\lambda,+}(\adele)} f((\gamma_0+n)h) dn \right) + \right.\\
\left. +\int_{V_{\lambda,+}(\adele)} f((\gamma_0+n)h) dn \right].
\end{eqnarray}

For $h=\lambda(t)$ and $t\to 0$, the terms on the first line are of rapid decay because $\lambda$ acts on $V_{\lambda,-}$ with negative weights.

The term on the third line is an $\adele^\times$-eigenfunction, when $\adele^\times$ acts via $\lambda$, with character $\chi_C$ (when $\lambda$ factors through $A$ and is in the interior of a cone $C$).

I claim that the term on the second line is of rapid decay. For this, the round bracket can be re-written, using Fourier transform on $V_{\lambda,+}$ (which we will write as $f\mapsto \hat f$):

$$\sum_{\gamma^*\in V_{\lambda,+}^\vee(k) \smallsetminus\{0\}} \widehat{h\cdot f_{\gamma_0}} (\gamma^*),$$
where $f_{\gamma_0}$ is the function $v\mapsto f(\gamma_0+v)$ on $V_{\lambda,+}(\adele)$, and $h\cdot$ represents its translate under $h=\lambda(t)$. Since the weights by which $\Gm$ acts on $V_{\lambda,+}^\vee$ via $\lambda$ are opposite to the ones by which it acts on $V_{\lambda,+}$, hence negative, this sum is rapidly decaying as $t\to 0$.

By making this argument ``locally uniform in $\lambda$'', in some sense, we will prove the theorem. Although not much more difficult, essentially, the final argument needs some careful quantification and the introduction of a notion of ``derivatives'' of exponent arrangements, which does not add much conceptually. I will therefore finish the proof in Appendix \ref{app:finite}.

\end{proof}

\begin{example} \label{exampleKR}
Here is an example that comes from the theta correspondence (Howe duality): Let $H= \SO(W)$, the special orthogonal group of a non-degenerate quadratic space $W$ of, say, even dimension $2m$. Let $V = W\otimes X$, where $X$ is a vector space of dimension $n$; in the theta correspondence, $X$ arises as a Lagrangian of a symplectic vector space. The problem at hand is to define a regularized integral:
\begin{equation}\label{reg-KR} \int_{[H]}^* \Sigma_V f(h) dh,\end{equation}
where $f\in \mathcal F(V(\adele))$ and $\Sigma_V$ is defined as in the proof above. This is the \emph{Kudla-Rallis} period, that appears in the theory of the Siegel-Weil formula \cite{KR}. (I thank Atsushi Ichino and Tamotsu Ikeda for drawing my attention to this example.)

For simplicity, let us assume that the quadratic space $W$ (and hence the group $H$) is split, and let $\epsilon_1,\dots,\epsilon_m$ denote the standard basis of characters of its Borel subgroup, so that the simple roots are $\epsilon_1-\epsilon_2, \dots, \epsilon_{m-1}-\epsilon_m$ and $\epsilon_{m-1}+\epsilon_m$. Thus, the anti-dominant cone $\mathfrak a^+$ is given by the inequalities
$$\epsilon_1\le \dots \le\epsilon_{m-1}\le \epsilon_m\le -\epsilon_{m-1}.$$ 

The multiset $\Phi(V)$ of weights of the representation $V$ consists of the weights of the standard representation: 
$$\pm \epsilon_1, \dots, \pm \epsilon_m,$$
each appearing with multiplicity $n$. The hyperplanes orthogonal to those weights divide the cone $\mathfrak a^+$ into two simplicial subcones, divided by the new wall with equation $\epsilon_m=0$. 

There are $m+1$ extremal rays of these cones, whose vectors we will denote, respectively, by $v_1, \dots, v_m$ and $v_m'$, where, in coordinates $(\epsilon_1,\dots, \epsilon_m)$:
$$ v_i = (x,x,\dots , x \mbox{\tiny{($i$-th position})}, 0,  \dots, 0)\mbox{ with }x\le 0$$
and  $v_m' = (x,x,\dots ,x, -x)$ with $x\le 0$.

For $C=\RR_+ v_i$ (or $v'_i$ with $i=m$), the character $\chi_C$ is equal to $v_i\mapsto ni x$. The modular character of $H$ is $2\rho=\sum_i 2(m-i) \epsilon_i$. On the vector $v_i$ (or $v_i'$ with $i=m$), we have $2\rho(v_i) = 2xi(m-\frac{i+1}{2})$.

Therefore, the exponents of $V$ are all non-critical if and only if
$$ ni x \ne 2xi(m-\frac{i+1}{2}) \mbox{ for all } 1\le i \le m$$
$$ \iff n \notin [m-1,2(m-1)],$$
in which case the regularized integral  \eqref{reg-KR} is defined. This is precisely the range (including the convergent range) in which Kudla and Rallis defined a regularized integral in \cite{KR} (by a different method), and proved a regularized Siegel-Weil formula.
\end{example}

\section{Evaluation maps} \label{sec:evaluation}

We return to the general case of $\mathcal X=X/G$, where $X$ is smooth affine and $G$ is reductive, both defined over a global field $k$. We will freely use the notation of Section \ref{sec:stalks}. Let $x:\spec k\to \mathcal X$ be a closed (semisimple) $k$-point lying over a point $\xi\in\cc(k)$. We keep assuming that the (reductive) stabilizer group of $x$ is connected. We will keep assuming, as we may without loss of generality by changing the presentation if necessary, that $x$ corresponds to a $G(k)$-orbit on $X(k)$; and, whenever we appeal to Luna's \'etale slice theorem \ref{lunathm}, we will be using the pair $(H, V)$ corresponding to the chosen point $x$, i.e., the identification of pointed \'etale neighborhoods \eqref{etale} will be such that $x$ corresponds to the image of $0 \in V$.

We will \emph{attempt} to define a canonical evaluation map
$$ \ev_x: \mathcal S(\mathcal X(\adele))_\xi  \to \CC,$$
factoring through the direct summand $\mathcal S(\mathcal X(\adele))_\xi^x$ of \eqref{directsummand}.
We define it at first for $\mathcal V=V/H$, using Theorem \ref{isfinite}, and then verify that it is invariant under the isomorphisms of stalks induced by a diagram of the form \eqref{etaleV}, thus giving rise to a well-defined functional on the stalk $\mathcal S(\mathcal X(\adele))_\xi$. 

Having defined the evaluation maps, we can define the ``relative trace formula'' for the quotient stack $\mathcal X$ as the sum of all $\ev_x$, for $x$ ranging over all isomorphism classes of semisimple $k$-points of $\mathcal X$.

Nonetheless, there is the serious restriction that the ``exponents'' obtained by application of Luna's \'etale slice theorem (see Definition \ref{defexponents}) need to be ``non-critical'' for this definition to make sense. To the best of my understanding, this definition \emph{cannot} be overcome by purely geometric methods. This includes the case of the Arthur--Selberg trace formula (i.e., the adjoint quotient of a reductive group), which is why one needs a combination of geometric and spectral considerations in this case to produce an invariant distribution. I will not get into this subject here.

\subsection{The regular case}

Consider first the case when $x:\spec k\to \mathcal X$, or equivalently its image $\xi\in \cc(k)$ is \emph{regular}, i.e., when it satisfies one of the following equivalent conditions:
\begin{itemize}
\item the  preimage of $\xi$ under the map $\mathcal X \to\cc$ consists only of the (neutral, by assumption) gerbe $\mathcal X_\xi$; 
\item the vector space $V$ of Luna's slice theorem has trivial $H$-action.
\end{itemize}

The equivalence of the two can easily be obtained from Luna's \'etale slice theorem by observing that for any non-trivial representation $V$ of $H$ there will be a (geometric) weight space with non-trivial weight, and hence a (geometric) non-zero point whose $H$-orbit contains $0$ in its closure.

Now let $(H,V)$ be Luna's \'etale neighborhood corresponding to the point $x$, and recall the map \eqref{tofns-global}:
$$\mathscr E: \mathcal S(\mathcal V(\adele))_0 \to \bigotimes_v' \mathcal F(N(k_v))_{(H(k_v),\delta_V^{-1})},$$
where the restricted tensor product on the right-hand side is taken with respect to the images of the functions:
$$\frac{1}{|\omega_H|_v(H(\mathfrak o_v))} \cdot 1_{V(\mathfrak o_v)}.$$

In the regular case, we have
$$ \mathcal S(V_v)_{H(k_v)} = \mathcal S(V_v),$$
(recall that the index $~_{H(k_v)}$ denotes $H(k_v)$-coinvariants), the nilpotent set $N\subset V$ is trivial (consists just of the point $0$) and hence the map $\mathscr E$ has image in a restricted tensor product of one-dimensional vector spaces (each of which is canonically equal to the space of functions on a point, although their restricted tensor product might not be --- see Remark \ref{notfunctions}).

We define the evaluation map for the ``zero'' point: $\spec k\to \mathcal V$ by the formula:
\begin{equation}\label{ev-regular} 
\ev_0 : \mathcal S(\mathcal V(\adele))_0 \ni \mu \mapsto |\omega_H|([H]) \cdot \mathscr E(\mu) \in \CC,
\end{equation}
where $|\omega_H|([H])$ denotes the Tamagawa volume on $[H]$ defined by any invariant volume form on $H$ over $k$. Notice that we \emph{do not} need to regularize Tamagawa measure, because the above product formally makes sense: $\mathscr E(\mu)$ is the quotient of a function by a formal Euler product of volumes, the Tamagawa measure $|\omega_H|$ is formally defined by the same formal Euler product of volumes, therefore these factors cancel out and \eqref{ev-regular} makes sense without regularization of the Euler product! Of course, $\ev_0$ is not defined (is infinite) if $[H]$ is not of finite volume; this should be seen as a special case of ``critical exponents'', to be discussed in more detail in the next subsection.

This definition is clearly invariant under the isomorphisms of stalks induced by a diagram of the form \eqref{etaleV} defined over $k$. The resulting evaluation maps on $\mathcal S(\mathcal X(\adele))_\xi$ can be intuitively described as follows: First, they factor through the direct summand \eqref{directsummand} associated to $x$. Secondly, they correspond, up to the volumes of stabilizers, to ``pushing forward a measure on $X(\adele)$ to $\cc(\adele)$, dividing by a Tamagawa measure on $\cc(\adele)$, and evaluating''. Of course, when the measures are represented by functions times differential forms, the evaluation of their push-forward is nothing else than an \emph{orbital integral} of the corresponding functions.

\subsection{The general case} \label{generalcase}

Let $N\subset V$ denote the nilpotent cone, i.e., the closed subvariety of elements which contain $0$ in their $H$-orbit; equivalently, the preimage of the image of $0$ under the quotient map $V\to V\sslash H$.

In \S \ref{ssregularized} we defined, based on the weights of the representation $V$, a pair $(\F, E)$, where $\F$ is a fan on the cone $\mathfrak a^+$  of anti-dominant cocharacters and $E$ is an exponent arrangement for this fan. If all exponents of this arrangement are non-critical, by Corollary \ref{corNintegral} we have a well-defined regularized integral:
$$\int^*_{[H]} \Sigma_N f(h) dh$$
for Schwartz functions on $V(\adele)$, where $\Sigma_N f (h) = \sum_{\gamma\in N(k)} f(\gamma h)$. We will now adapt this to measures.

A slight difference to the setup of \S \ref{ssregularized} is that here we do not necessarily have a faithful action of $H$ on $V$, so let us adapt the definitions of the exponent arrangement to the current setting:

Let $H'$ be the quotient by which $H$ acts on $V$, and add a prime to every piece of notation that was used for $H$ in order to refer to the corresponding object for $H'$. Thus, we have a well-defined exponent arrangement $(\F', E')$ on the vector space $\mathfrak a'$, with the fan $\F'$ supported on the anti-dominant chamber. I recall that $\F'$ was obtained from the hyperplanes perpendicular to the weights of $V$ (and the walls of the Weyl chamber), and $E'$ was obtained by summing all weights that are positive each given cone of the fan (and multiplying by $-1$); thus, they are both very explicit in terms of the representation.

If the kernel of $\mathfrak a\to \mathfrak a'$ contains the cocharacter space of a non-trivial central split torus of $H$, then we stop here and consider it a case of critical exponents: the orbital integral of $\Sigma_N$ over $[H]$ cannot be regularized, since a central split torus acts trivially on $V$. 

If not, then the kernel $\mathfrak a_0$ of $\mathfrak a\to\mathfrak a'$ is the (cocharacter space of the) universal maximal split torus of a sum of simple factors of $H$; it thus admits a canonical splitting:
$$ \mathfrak a = \mathfrak a_0 \oplus \mathfrak a'$$
characterized by the fact that $\mathfrak a'$ contains the (cocharacter space of the) universal maximal split torus of the remaining simple factors and the center. We then use this splitting to lift the fan $\F'$ to a fan $\F$ on $\mathfrak a$, and let $E$ be the exponent arrangement on $\F$ obtained by applying the bijection between the two fans to $E'$.

\begin{definition} \label{defexponents}
 Let $\mathfrak d_V$ be the modular character of the vector space $V$. (It is a character of $H$, hence of its universal maximal split torus $A$.) Let $\delta_V E$ denote the exponent arrangement $E$ shifted by $\mathfrak d_V$, i.e., its elements (written additively) are sums of elements of $E$ by $\mathfrak d_V$. (As multiplicative characters of $A(\adele)$, they are products of elements of $E$ by the absolute value $\delta_V$ of $\mathfrak d_V$ --- hence the notation.) The elements of $\delta_V E(C)$, as $C$ ranges over all the cones in $\F$, will be referred to as the \emph{exponents} of the point $x\in \mathcal X(k)$ from which the pair $(H, V)$ was obtained. 
\end{definition}

Recall again the map \eqref{tofns-global}:
$$\mathscr E: \mathcal S(\mathcal V(\adele))_0 \to \bigotimes_v' \mathcal F(N(k_v))_{(H(k_v),\delta_V^{-1})}.$$

Given Corollary \ref{corNintegral}, we can now define the evaluation map, first for the $k$-point corresponding to $0$ in the quotient stack $\mathcal V$, as follows: 

\begin{definition}\label{ev} 
Let $\mu$ be an element of $\mathcal S(\mathcal V(\adele))_0$, and let $f$ be a representative in $\bigotimes'_v\mathcal F(N(k_v))$ of its image under the map \eqref{tofns-global}. We define the evaluation map corresponding to the zero $k$-point $0: \spec k \to \mathcal V$ as
\begin{equation}\label{ev-general}
\ev_0(\mu) := \int^*_{[H]} \Sigma_N f(h) \delta_V(h) |\omega_H|(h),
\end{equation}
whenever the exponents of $\mathcal V$ at $0$ are not critical.
\end{definition}

Note that, since the above expression considered as a functional on the variable $f$ is $(H,\delta_V)$-equivariant, and since it depends only on the restriction of $f$ on $N(\adele)$, it does not depend on the choice of representative $f$. By Corollary \ref{corNintegral}, this is a continuous functional on the stalk $\mathcal S(\mathcal V(\adele))_0$, and it clearly factors through the direct summand $\mathcal S(V(\adele))_{H(\adele), 0}$ of \eqref{directsummand-V}.

We recall from \S \ref{regularized-automorphic} that an exponent on $A_C$ is ``critical'' if it is equal to the restriction to $A_C$ of the modular character of the minimal parabolic (or, equivalently, of $P_C$).

Here, again, the measure $|\omega_H|$ is formally the one obtained from a global volume form, and there are formal cancellations with the volume factors in the definition of the restricted tensor products, so that in the end we are taking the regularized integral of an \emph{actual} function against an \emph{actual}, non-zero measure; see the discussion following \eqref{ev-regular}.

Now we extend definition \ref{ev} to the point $x$ of the stack $\mathcal X$. It will not, at this point, be clear that the definition is independent of choices; this will be proven in the next section, Proposition \ref{indep}.

\begin{definition} \label{evX}
Provided that $\mathcal X$ has no critical exponents at $x$, we define the evaluation map:
$$ \ev_x: \mathcal S(\mathcal X(\adele))_\xi \to \CC$$
as the pull-back of the evaluation map $\ev_0$ under the identification of stalks: 
$$\mathcal S(\mathcal X(\adele))_\xi \simeq \mathcal S(\mathcal V(\adele))_0$$
of \eqref{globalisom}, induced from a Luna diagram \eqref{etale}.
\end{definition}

It is a continuous functional on the stalk $\mathcal S(\mathcal X(\adele))_\xi$, and it clearly factors through the direct summand $\mathcal S(\mathcal X(\adele))_\xi^x$ of \eqref{directsummand}.

\subsection{Invariance under isomorphism of \'etale neighborhoods} \label{ssinvariance}

Up to now we have defined a distribution $\ev_0$ on $\mathcal S(\mathcal V(\adele))_0$ which, in principle, might not be invariant under the isomorphism of stalks induced by a diagram of the form \eqref{etaleV}. Here we check that when the diagram is defined over $k$, then the distribution $\ev_0$ is invariant under the resulting isomorphisms of stalks. As a corollary, the evaluation map $\ev_x$ of Definition \ref{evX} is well-defined, in the absence of critical exponents.

\begin{proposition}\label{indep}
Consider a diagram of the form \eqref{etaleV} over $k$, arising from two different choices of data for Luna's theorem, and the automorphism
$$ \mathcal S(\mathcal V(\adele))_0 \xrightarrow\sim \mathcal S(\mathcal V(\adele))_0$$
that it induces on the stalk. The functional $\ev_0$ is invariant under this automorphism.
\end{proposition}

\begin{proof}
We will use notation as in diagram \eqref{fiberproduct} and in the  proof of Proposition \ref{preservesbasic}.

Consider the space of volume forms on $P$ which are $\mathfrak d_V \times \mathfrak d_V$-equivariant under the action of $H\times H$ --- they form sections of a line bundle $\widetilde{\mathcal L}$  over $\cc_{\mathcal Y}$. If $N_P$ denotes the preimage in $P$ of the distinguished point of $\cc_{\mathcal Y}$, then as in Lemmas \ref{meastofunction} and \ref{doesnotdepend}, any section $\omega_P$ of $\widetilde{\mathcal L}$ over $k$ gives rise to the same map
\begin{equation}\label{EP}\mathscr E_P:\mathcal S(\mathcal Y(\adele))_0 \to \bigotimes'_v \mathcal F(N_P(k_v))_{(H_v\times H_v, \delta_V^{-1}\times \delta_V^{-1})}.
\end{equation}
Here the basic vector for the restricted tensor product on the right side will be
$$\frac{1}{|\omega_H|_v(H(\mathfrak o_v))^2} \cdot 1_{P(\mathfrak o_v)}.$$

On the other hand, the evaluation map $\ev_0$ that we have defined  on $\mathcal S(\mathcal V(\adele))_0$ pulls back in two different ways to the corresponding restricted tensor product
\begin{equation}\label{productP} \bigotimes_v' \mathcal S(P(k_v)),\end{equation}
by the ``left'' and the ``right'' sequence of arrows from $P$ to $\mathcal V$.

I claim that \emph{both} pullbacks of the evaluation map  to $\mathcal S(\mathcal Y(\adele))_0$ can be expressed in terms of \eqref{EP} as follows: 
\begin{equation}\label{onP} \mu\mapsto \int^*_{[H\times H]} \sum_{\gamma\in N_P(k)}\mathscr E_P(\mu)(\gamma\cdot (h_1,h_2)) \delta_V(h_1) \delta_V(h_2) dh_1 dh_2.\end{equation}

Indeed, the map $W_1 \to V$ which is Cartesian over $\cc_{\mathcal Y} \to \cc_V$ identifies $N\subset V$ with its preimage $N_1$ in $W_1$ over the distinguished point of $\cc_{\mathcal Y}$, and the preimage of that in $P$ is $N_P$. Recall that in order to define the maps \eqref{tofns-global} and \eqref{EP} we have chosen $H$-eigen- (resp.\ $H\times H$-eigen-)volume forms $\omega_V$ and $\omega_P$ on $V$ and $P$ (although the global result did not depend on those choices). We can pull back $\omega_V$ to a volume form $\omega_1$ on $W_1$. Since $P\to W_1$ is an $H$-torsor, the conormal bundles to all $H$-orbits form a subbundle of the cotangent bundle of $P$, and since $\omega_1$ is non-vanishing, the quotient of $\omega_P$ by the pull-back of $\omega_1$ is a well-defined section $\omega_P'$ of the top exterior power of this subbundle. This induces an $H$-eigen-volume form on every fiber of the map $P\to W_1$.

Now recall from Lemma \ref{Nlifts} that the map $N_P(k)\to N_1(k)$ is surjective. 
In other words, over $k$-points of $N_1$ the $H$-torsor $P$ is trivializable. Choosing any $k$-point to trivialize it, the restriction of $\omega_P'$ on its fiber is identified with a Haar volume form on $H$. It is now immediate to unfold the integral \eqref{onP} and to see that it corresponds to the pull-back of the evaluation map via the left arrows to $V$. Exactly the same applies, of course, to the right arrows.

Again by Lemma \ref{Nlifts} the image of \eqref{productP} in $\mathcal S(\mathcal V(\adele))_0$ is precisely the direct summand $\mathcal S(V(\adele))_{H(\adele),0}$ through which $\ev_0$ factors. Hence, the fact that $\ev_0$ pulls back to the same functional under both the ``left'' and ``right'' sequence of arrows shows that $\ev_0$ is invariant under the isomorphism of stalks induced by \eqref{etaleV}.
\end{proof}

\subsection{The ``relative trace formula''} \label{ssRTF}

We keep assuming that $\mathcal X$  is an algebraic stack of the form $\mathcal X=X/G$ over $k$, where $X$ is a smooth affine variety and $G$ is a reductive group. We assume that $X$ carries a nowhere vanishing $G$-eigen-volume form $\omega_X$, so that its global Schwartz space
$$ \mathcal S(\mathcal X(\adele)) = \bigotimes'_v \mathcal S(\mathcal X_v)$$
can be defined, as in Remark \ref{remarkgroup}.

\begin{definition} \label{defRTF} Assume that the stabilizer subgroups of all semisimple points $x:\spec k \to \mathcal X$ are connected and their exponents are non-critical. The \emph{relative trace formula} for $\mathcal X$ is the following distribution on $\mathcal S(\mathcal X(\adele))$:
\begin{equation} \RTF_{\mathcal X}: f\mapsto \sum_x \ev_x,
\end{equation}
where $x$ runs over isomorphism classes of closed $k$-points into $\mathcal X$, and $\ev_x$ is the evaluation map of definition \ref{ev}.
\end{definition}

Of course, this has nothing to do with traces, in general.

\begin{example}
The simplest example of a relative trace formula where some ``stacky'' behavior is seen is the one considered by Jacquet in \cite{JW1}, where the stack is $\mathcal X = T\backslash \PGL_2/T$, where $T$ is a non-split torus in $\PGL_2$, defined over $k$. In this case, since $H^1(k,T)$ injects in $H^1(k,\PGL_2)$ (and the same is true for the completions $k_v$), isomorphism classes of $T$-torsors $R$ inject into isomorphism classes of quaternion algebras $D_R$, and the image consists of those quaternion algebras such that $D_R^\times$ contains $T$ or equivalently: the quaternion algebra splits over the quadratic extension splitting $T$. 

Thus, isomorphism classes of $k$-points of $\mathcal X$ are in bijection with
$$ \bigsqcup_R T(k)\backslash PD_R^\times (k) /T(k),$$
where $R$ runs over isomorphism classes of those quaternion algebras, and $PD_R^\times$ denotes the quotient of $D_R^\times$ by its center, and the relative trace formula can be written as a sum
\begin{equation}\label{RTF-torus} \sum_R \sum_{x\in T(k)\backslash PD_R^\times (k) /T(k)} \ev_x.\end{equation}

The local Schwartz spaces are the $T(k_v)\times T(k_v)$-coinvariants of the direct sum
$$ \sum_{R_v} \mathcal S(PD_{R_v}^\times(k_v)),$$
the sums ranging over isomorphism classes of $T$-torsors over $k_v$, and any Schwartz measure on $PD_{R_v}^\times(k_v)$ can be written as $f dg$, where $f$ is a Schwartz function and $dg$ is a Haar measure on $PD_{R_v}^\times(k_v)$. 

Thus, the evaluation maps $\ev_x$ of \eqref{RTF-torus} can be identified with orbital integrals of those functions $f$, and this holds both for regular (with trivial stabilizers, in this case) and irregular points, as can easily be checked from the definitions using the fact that $T$ is globally anisotropic and the nilpotent cones have no $k$-points. 
\end{example}

\begin{example}
In the aforementioned paper, Jacquet compares the relative trace formula for $T\backslash \PGL_2/T$ with the relative trace formula when $T$ is replaced by a split torus $A$, but with the $T(\adele)$-action twisted by a quadratic character on one side. Such characters are not part of the formalism of the present paper (and they are important, as they show up elsewhere as well --- including additive characters in the Kuznetsov formula), but it would probably not be difficult to incorporate them as line bundles over the pertinent Nash stack. 

The case of $A\backslash \PGL_2/A$ (with trivial character) was not considered in \cite{JW1} because it would lead to no new arithmetic results, but it is interesting from the analytic point of view to make sense of such a trace formula, and was considered in \cite{SaBE1,SaBE2}. In this case, at the irregular points the linearization of the stack (i.e., the space $V/H$ of Luna's theorem) is of the form
$$V=\Ga(1) \oplus \Ga(-1),$$
where $\Ga(i)$ denotes an one-dimensional vector space where $A\simeq \Gm$ acts with weight $a\mapsto a^i$. By \cite[Lemma 2.6.1]{SaBE2}, the contribution of nilpotent cone to the ``relative trace formula'' functional applied to a measure $\Phi dv$ on $V(\adele)$ (where $dv$ is Tamagawa measure) is given by
\begin{equation}\label{zeta} \lim_{t\to 0} (\zeta(\Phi|_x, t) + \zeta(\Phi|_y,-t)),\end{equation}
where $\Phi|_x$ denotes the restriction of $\Phi$ to the ``$x$-axis'' $\Ga(1)$, $\Phi|_y$ its restriction to $\Ga(-1)$ and $\zeta$ the corresponding Tate integrals. One can check that in this case the function $\Sigma_N \Phi$, in the notation of Theorem \ref{isfinite}, is asymptotically finite on $[\Gm]$ with exponents $t\mapsto |t|^{-1}$, resp.\ $t\mapsto |t|$ as $|t|\to 0$, resp.\ $|t|\to \infty$, and \eqref{zeta} is the regularized integral of $\Sigma_N \Phi$. 
\end{example}

\begin{example} \label{exampleGP}
The exponents are non-critical, and hence the relative trace formula can be defined by the methods of the current paper, for the quotient spaces $(X\times X)/G^\diag$ when $X$ is the homogeneous space of the Gross--Prasad conjectures ($\SO_n^\diag\backslash SO_n\times SO_{n+1}$ or $U_n^\diag\backslash U_n\times U_{n+1}$, considered in \cite{Zydor1, Zydor3}), and for a variant of the corresponding Jacquet--Rallis relative trace formula for linear groups \cite{Zydor2, Zydor3}. I expect that, in the unitary Gross--Prasad case that has been studied by Zydor, the resulting distribution coincides with his, but it would be an interesting exercise to show that. 

Let us check that the exponents are non-critical at the most singular points of the morphism $\mathcal X\to\cc$, where $\mathcal X$ denotes the pertinent stack in each case and $\cc$ denotes the invariant-theoretic quotient. In the Gross--Prasad cases, this is the point represented by the element $1$ in the presentations:
$$ \mathcal X = \SO_{n+1}/\SO_n\conj,\,\,\mbox{ resp.\ }\mathcal X = U_{n+1}/U_n\conj.$$
Here $\SO_{n+1}$, resp.\ $U_{n+1}$, denotes the special orthogonal, resp.\ unitary group of the quadratic/hermitian space obtained from that of $\SO_n$ by adding an orthogonal line (with a non-degenerate quadratic/hermitian form on it).
Let $H:=H_n := \SO_n$, resp.\ $U_n$. The linearization of the stack at $1$ is the adjoint representation of $H$ on the Lie algebra $\mathfrak h_{n+1}$. It decomposes as
$$\mathfrak h \oplus \std,$$
where $\mathfrak h$ implies the adjoint action of $H$ on its Lie algebra and $\std$ is the standard representation, plus a copy of the trivial representation in the unitary case (but this will not contribute anything to the weights).

We restrict this representation to a maximal split torus (inside of a chosen minimal parabolic subgroup), and use the recipe for constructing a fan and an exponent arrangement that was given in \S \ref{ssregularized} to produce a partition of the anti-dominant cone. If $m$ is the split rank of $H$, and we embed $\GL_m$ in $H$ into a Levi subgroup (let us assume that $m\ge 3$ in the even orthogonal case), and denote by $\epsilon_1, \dots, \epsilon_m$ the standard characters of the torus of diagonal elements in $\GL_m$, the anti-dominant cone with respect to the usual choice of Borel is given by the inequalities
$$ \epsilon_1\le \dots \le\epsilon_{n-1}\le \epsilon_n\le 0,$$
except in the even orthogonal case where it is given by
$$ \epsilon_1\le \dots \le \epsilon_{n-1}\le\epsilon_n \le -\epsilon_{n-1}.$$
The non-zero weights of the standard representation are
$$ \pm \epsilon_1, \pm\epsilon_2 ,\dots, \pm \epsilon_n,$$
and therefore they don't introduce new walls to the anti-dominant chamber (they maintain their sign), except in the even orthogonal case where they introduce a wall at $\epsilon_n =0$. 

Thus, the fan $\F$ is the one corresponding to the anti-dominant cone $\mathfrak a^+$, except in the even orthogonal case where it corresponds to the partition of this cone into two subcones along the hyperplane $\epsilon_n=0$. 

In either case, the exponent arrangment is obtained by adding to the modular character of the minimal parabolic (coming from $\mathfrak h$) the following character: 
\begin{itemize}
\item for the cone given by the inequalities
$$ \epsilon_1\le \dots \le\epsilon_{n-1}\le \epsilon_n\le 0,$$
and for all its faces, the character
$$\epsilon_1+\dots + \epsilon_{n-1} + \epsilon_n;$$
\item in the even orthogonal case, for the other cone (with $\epsilon_n\ge 0$) and its faces, the character
$$\epsilon_1+\dots + \epsilon_{n-1} - \epsilon_n.$$
\end{itemize}

In both cases, the character added is \emph{non-trivial}, unless all $\epsilon_i=0$, i.e., except at the point $0\in \mathfrak a$. Therefore, the exponents are never critical. This finishes the discussion of the Gross--Prasad quotients.

In the linear case, the relevant stack is
$$ \mathcal X = \GL_{n,E}\backslash \GL_{n,E}\times \GL_{n+1,E}/\GL_n \times \GL_{n+1},$$
where $\GL_n$, $\GL_{n+1}$ are defined over our global field $k$, and with an index $E$ we denote the restriction of scalars of the base change to a quadratic extension $E/k$. However, there is a quadratic character entering the Jacquet-Rallis relative trace formula (on $\GL_n$ or $\GL_{n+1}$, whichever is even-dimensional), which I have presently not included in the formalism. However, the linearization close to the most singular point of $\mathcal X\to\cc$ (again represented by the element $1$) has been described in \cite{Zydor2}, and has to do with $H:=H_n:=\GL_n$-orbital integrals, against a quadratic character, on the space of the representation $V = \{X \in \mathfrak h_{n+1, E} | X+\bar X = 0\}$. That is, the only difference from the setup of the present paper is that for the same pair $(H,V)$, one needs to put a quadratic character on $H$. 

We check again the asymptotic behavior of the functions of the form $\Sigma_V \Phi$, where $\Phi$ is a Schwartz function on $V(\adele)$. The representation $V$ is isomorphic to the adjoint action of $H=H_n$ on the Lie algebra $\mathfrak h_{n+1}$, and hence decomposes as
$$ \mathfrak h \oplus \std,$$
plus a copy of the trivial representation. By the same arguments as above, the fan given by the recipe of \S \ref{ssregularized} corresponds just to the anti-dominant cone and its faces, and the exponents differ from the modular character of the minimal parabolic. Therefore, the regularized integral can be defined (even with a quadratic character).
\end{example}

\appendix

\section{From algebraic to Nash stacks: a presentation-free approach} \label{app:Nash}

\subsection{} Let $\mathcal X$ be a smooth algebraic stack of finite type over a local field $F$. In this appendix I will construct out of $\mathcal X$ a stack $\mathfrak X$ over the (\'etale) site $\mathfrak N$ of Nash manifolds, that does not use a fixed presentation of $\mathcal X$ but, rather, a ``limit'' over all presentations. If there is an $F$-surjective presentation, this construction is equivalent to the construction by groupoids that was presented in \S \ref{algtoNash}.

I will construct the stack $\mathfrak X=\mathcal X(F)$ over $\mathfrak N$ as a ``limit'' of the Nash groupoids $[R_X(F) \rightrightarrows X(F)]$, where $X$ runs over all algebraic presentations $X\to \mathcal X$, and $R_X=X\times_{\mathcal X} X$.

This ``limit'' will be obtained by localization of a category with respect to smooth surjective morphisms: $X'\to X$ over $\mathcal X$. More precisely, we first define a pre-stack $\mathfrak X^\ps$, as a localization of a category $\mathfrak X^\naive$. The category $\mathfrak X^\naive$ is fibered not only over $\mathfrak N$, but also over $\PresX$, the category of smooth epimorphisms of algebraic stacks (``presentations''): $X\to \mathcal X$, where $X$ is a (necessarily smooth) scheme of finite type over $F$, with morphisms between $(X'\to \mathcal X)$ and $(X\to \mathcal X)$ being all smooth morphisms $X'\to X$ over $\mathcal X$. For $X\in \ob(\PresX)$, set $R_X = X\times_{\mathcal X} X$, with the ``source'' and  ``target'' maps $s$ and $t$:
$$ R_X \overset{s}{\underset{t}\rightrightarrows} X.$$

The fiber category of $\mathfrak X^\naive$ over $(X\to \mathcal X)\in \ob(\PresX)$ is the Nash stack associated to the Nash groupoid $[R_X(F)\rightrightarrows X(F)]$.

Let $X'\to X$ be a morphism in $\PresX$, and $(T\to W, T\to X(F))$ an object in $\mathfrak X^\naive$ over $(X\to \mathcal X)$. The fiber product
$$ T':= T \times_{X(F)} X'(F)$$
is still smooth over $W$, and an epimorphism iff the image of $X'(F)$ in $X(F)$ contains the image of $T$. The map $T'\to W$ is naturally an $R_{X'}$-torsor.

A morphism $(T'\to V, T'\to X'(F)) \to (T\to W, T\to X(F))$ in $\mathfrak X^\naive$ is a Cartesian square of Nash manifolds:
\begin{equation}\label{morphism} \xymatrix{
T' \ar[r]\ar[d] & T\ar[d] \\
V\times X'(F) \ar[r] & W\times X(F),
}\end{equation}
satisfying the following requirements: the bottom arrow is induced from a morphism $V\to W$ in $\mathfrak N$ and a morphism $X'\to X$ in $\PresX$, \emph{the image of $X'(F)\to X(F)$ contains the image of $T\to X(F)$}, and the induced diagram of groupoids commutes:
$$ \xymatrix{
[T'\times_V T'\rightrightarrows T'] \ar[r]\ar[d] & [T\times_W T \rightrightarrows T] \ar[d] \\
[R_{X'}(F)\rightrightarrows X'(F)] \ar[r] & [R_X(F)\rightrightarrows X(F)].
}
$$
For simplicity, we will usually represent such a morphism just by the arrow $T'\to T$, with all other arrows being implicit.

Now we construct the pre-stack $\mathfrak X^\ps$ by localizing $\mathfrak X^\naive$ over a class $S$ of morphisms, i.e., $\mathfrak X^\ps = \mathfrak X^\naive[S^{-1}]$. The class $S$ is the class of those morphisms of the form \eqref{morphism} where $V=W$ and the map $V\to W$ is the identity, i.e., $T'$ is obtained by base change in $\PresX$ only:
$$ \xymatrix{
T' \ar[r]\ar[d] & T\ar[d] \\
X'(F) \ar[r] & X(F)
}$$
(where, again, I remind the reader that the image of $T$ has to lie, set-theoretically, in the image of $X'(F)$). The morphisms in $S$ will be represented by double arrows $\Rightarrow$.

\begin{proposition}
The class $S$ allows a calculus of right fractions. This means that it verifies the right Ore conditions:
\begin{enumerate}[(a)]
\item for any $s\in S$ and any other morphism $f$ in $\mathfrak X^\naive$ with the same target, there is a commutative diagram
$$ \xymatrix{ T_1 \ar@{-->}[r]^{f'} \ar@{==>}[d]_{s'} & T_2 \ar@{=>}[d]^s \\
T_3 \ar[r]_f & T_4 \\
}$$
with $s'\in S$.

\item If 
$$\xymatrix{ T_1 \ar@<0.5ex>[r]^{f_1} \ar@<-0.5ex>[r]_{f_2} & T_2 \ar@{=>}[r]^t & T_3}$$ is such that $t\circ f_1 = t\circ f_2$, with $t\in S$, then there is an $s$ in $S$ such that $f_1\circ s = f_2\circ s$:
$$\xymatrix{T_4 \ar@{=>}[r]^s & T_1 \ar@<0.5ex>[r]^{f_1} \ar@<-0.5ex>[r]_{f_2} & T_2}.$$

\end{enumerate}
\end{proposition}

\begin{proof}
The first one is true, with the arrows induced by base change from the corresponding diagram:
$$ \xymatrix{ V\times X_1(F) \ar[r] \ar@{=>}[d] & W\times X_2(F) \ar@{=>}[d]^s \\
V\times X_3(F) \ar[r]_f & W\times X_4(F) \\
}$$
where $X_1 = X_2\times_{X_4} X_3$.

For the second one, consider the corresponding arrows:
$$ \xymatrix{
T_1 \ar[d]\ar@<0.5ex>[r]^{f_1} \ar@<-0.5ex>[r]_{f_2}  &  T_2 \ar@{=>}[r]^t\ar[d] & T_3 \ar[d]\\
V \times X_1(F) \ar@<0.5ex>[r]^{f_1} \ar@<-0.5ex>[r]_{f_2}  &  U \times X_2(F) \ar@{=>}[r]^t & U\times X_3(F),
}
$$
where we recall that corresponding squares are Cartesian. Since $t$ is, by definition of the class $S$, the identity on $U$, and $t\circ f_1 = t\circ f_2$ by assumption, it follows that $f_1|_V = f_2|_V$. Therefore, 
 the left squares can be completed to a Cartesian diagram:
$$ \xymatrix{
T_4\ar[d] \ar@{=>}[r]^s & T_1 \ar[d]\ar@<0.5ex>[r]^{f_1} \ar@<-0.5ex>[r]_{f_2}  &  T_2 \ar[d] \\
V \times (X_1 \times_{f_1, X_2, f_2} X_1)(F) \ar@{=>}[r] &
V \times X_1(F) \ar@<0.5ex>[r]^{f_1} \ar@<-0.5ex>[r]_{f_2}  &  U \times X_2(F).
}
$$
The square on the left corresponds to a morphism $s$ in $S$, and by construction we have $f_1\circ s = f_2 \circ s$.

\end{proof}

The fact that $S$ allows a calculus of right fractions means that morphisms $T_1\to T_2$ in $\mathfrak X^\ps = \mathfrak X^\naive[S^{-1}]$ can be described as equivalence classes of ``roofs'':
$$ \xymatrix{
& Z \ar@{=>}[dl] \ar[dr] \\
T_1 && T_2
}$$
where two roofs with sources $Z$ and $Z'$ are equivalent if they are dominated by a third one, i.e., there is a commutative diagram:
$$ \xymatrix{
& Z \ar@{=>}[dl]  \ar[dr] \\
T_1 & \ar@{=>}[l] Z'' \ar[u] \ar[d] \ar[r] & T_2. \\
& Z' \ar@{=>}[ul]  \ar[ur] 
}$$

This makes it easy to check that $\mathfrak X^\ps$ is indeed a pre-stack.

\begin{proposition}
$\mathfrak X^\ps = \mathfrak X^\naive[S^{-1}]$ is a pre-stack in the \'etale topology over $\mathfrak N$.
\end{proposition}

\begin{proof}
We first check that it is a category fibered in groupoids over $\mathfrak N$.

First, given an object $(T\to U, T\to X(F))$ in $\mathfrak X_U^\ps$ and a morphism $V\to U$ in $\mathfrak N$, we get by base change an object $(T\times_U V \to V, T\times U V\to X(F))$ in $\mathfrak X_V^\ps$ as required by the first axiom of ``category fibered in groupoids''. Notice that both objects live over the same $(X\to \mathfrak X)\in \ob(\PresX)$.

For the second axiom, let $W\to V\to U$ be a diagram in $\mathfrak N$ and $T_W\to T_U$, $T_V\to T_U$ morphisms in $\mathfrak X^\ps$ lying over $W\to U$, $V\to U$, respectively. If we denote by $X_W, X_V$ and $X_U$ the smooth $\mathcal X$-schemes over which they live, the morphisms between them mean correspond to equivalence classes of base changes: $T_W'\to X_W'(F)$, $T_V'\to X_V'(F)$ (induced from $X_W'\to X_W, \,\, X_V'\to X_V$) and morphisms: $X_W'\to X_U$, $X_V'\to X_U$, together with Cartesian diagrams:
$$ \xymatrix{
T_W' \ar[d]\ar[r] & T_U \ar[d] \\
X_W'(F) \ar[r] & X_U(F)
}$$
(and similarly for $T_V'$).

Thus, $T_W'\simeq T_U\times_{U\times X_U(F)} (W\times X_W'(F))$ (and similarly for $T_V'$) and the map $W\to V$ induces a unique map:
$$ T_W'' \to T_V',$$
where $T_W''=T_U\times_{U\times X_U(F)} (W\times (X_W'\times_{X_U} X_V')(F))$. The natural map $T_W''\to T_W'$ corresponds to a morphism in $S$, hence the arrow $W\to V$ lifts to an arrow $T_W\to T_V$ in $\mathfrak X^\ps$ such that the composition: $T_W\to T_V\to T_U$ equals the given morphism $T_W\to T_U$.

It is easy to see that any two roofs representing such a morphism $T_W\to T_V$ are dominated by a third one (constructed again by a fiber product), hence the lift is unique and the category is indeed fibered in groupoids over $\mathfrak N$.

Now we verify the pre-stack axiom: Consider two isomorphisms between two objects $T_1, T_2\in \ob(\mathfrak X^\ps_U)$. Suppose that these isomorphisms are identified over all elements of a covering family $\{U_i\to U\}_i$, that is: all elements of a finite cover of $U$ by \'etale maps. The isomorphisms are represented by ``roofs'' as before, and the fact that they are identified locally means that, locally, these roofs are dominated by other roofs. The latter lie over objects $(X_i\to \mathcal X)\in \PresX$; notice that \emph{each of them} is a presentation of $\mathcal X$, since the cover only concerns the $\mathfrak N$-parameter. Therefore, a fiber product of the $X_i$'s (which are finite in number) over $\mathcal X$ will provide the required isomorphism between the original morphisms. 

Vice versa, assume that $U=U_1\cup U_2$ (for simplicity), and let $T_{1,U_i}\to T_{2,U_i}$ be isomorphisms which ``agree'' over $U_1\cap U_2$. These morphisms are represented by morphisms in $\mathfrak X^\naive$: $T_{1,U_i}'\to T_{2,U_i}$, where $T_{1,U_i}'\to T_{1,U_i}$ is in $S$ (and we may by base change assume that for both $i=1,2$ the objects $T_{1,U_i}'$ live over the same object $(X'\to\mathcal X)\in\PresX$. The condition of their agreement on $U_1\cap U_2$ can be expressed with a further roof over $U_1\cap U_2$, and again by taking fiber products with the corresponding presentation we may in fact assume that the actual morphisms: $T_{1,U_i}'\to T_{2,U_i}$ agree over $U_1\cap U_2$. But then they can be glued to a morphism of Nash stacks: $T_1'\to T_2$ which together with the map $T_1'\to T_1$ represents the desired isomorphism: $T_1\to T_2 \in \mathfrak X^\ps$.

\end{proof}

\begin{definition}
Given a smooth algebraic stack $\mathcal X$ over $F$, we define a stack over $\mathfrak N$ as follows:
\begin{quote} $\mathfrak X := \mathcal X(F) := $ the stackification of $\mathfrak X^\ps$. 
\end{quote}
\end{definition}

This completes the definition, which poses a number of interesting questions, such as: Is the association $\mathcal X\to \mathfrak X$ functorial? Is $\mathfrak X$ a Nash stack? I will not address these questions in this paper, except for observing that the definition coincides with the one given in \S \ref{algtoNash} in the presence of $F$-surjective presentations:

\begin{proposition}\label{surjective}
Let $X\to \mathcal X$ be an $F$-surjective presentation of the smooth algebraic stacks $\mathcal X$ over $F$. 

Then $X(F)\to \mathfrak X:=\mathcal X(F)$ is a smooth epimorphism of stacks over $\mathfrak N$ and, in particular, $\mathfrak X$ is a Nash stack.

This Nash stack is equivalent to the stack defined by the groupoid object 
$$[(X\times_{\mathcal X} X)(F)\rightrightarrows X(F)].$$
\end{proposition}

\begin{proof}
The point is that the stack $[(X\times_{\mathcal X} X)(F)\rightrightarrows X(F)]$ is a final object, in some sense, in $\mathfrak X^\naive$. More precisely, let $X'\to \mathcal X$ be any other presentation of $\mathcal X$, and let $X''=X\times_{\mathcal X} X'$. Then we have obvious $1$-morphisms of stacks:
$$[R_{X''}(F)\rightrightarrows X''(F)] \to [R_X(F)\rightrightarrows X(F)],$$
$$[R_{X''}(F)\rightrightarrows X'(F)] \to [R_{X'}(F)\rightrightarrows X'(F)].$$

Since $X\to \mathcal X$ is $F$-surjective, the smooth morphism of Nash manifolds $X''(F)\to X'(F)$ is an epimorphism, and hence by Lemma \ref{epi} the second morphism above is an equivalence. It follows that for any presentation $X'\to \mathcal X$ we have a canonical (up to 2-isomorphism) $1$-morphism of stacks:
$$[R_{X'}(F)\rightrightarrows X'(F)] \to [R_X(F)\rightrightarrows X(F)],$$
and that the ``limit'' prestack $\mathfrak X^\ps$ (and hence its stackification $\mathfrak X$) is equivalent to the stack associated to $[R_X(F)\rightrightarrows X(F)]$. Of course, the $1$-morphism $X(F)\to \mathfrak X$ is automatically a smooth epimorphism in that case.

\end{proof}

\section{Homology and coshseaves in non-abelian categories}\label{app:homology}

\subsection{Exact structures} \label{ssexact}

Our basic references are \cite{Buehler,FS}.

\begin{definition} Let $\mathscr A$ be an additive category. A \emph{kernel-cokernel pair} $(f, g)$ in $\mathscr A$ is a pair of composable morphisms
$$X\xrightarrow{f} Y \xrightarrow{g} Z$$
in $\mathscr A$, such that $f$ is a kernel of $g$ and $g$ is a cokernel of $f$. If a class $\mathscr E$ of kernel-cokernel pairs on $\mathscr A$ is fixed, an \emph{admissible monomorphism} is a morphism $f$ for which there exists a morphism
$g$ such that $(f, g) \in \mathscr E$. \emph{Admissible epimorphisms} are defined dually. 

An exact structure on $\mathscr A$ is a class $\mathscr E$ of kernel-cokernel pairs which is closed under isomorphisms and satisfies the following axioms:
\begin{description}
\item[E$0$] For all objects $X$ in $\mathscr A$ , the identity morphism $\id_X: X\to X$ is an admissible monomorphism.
\item[E$0^\op$] For all objects $X$ in $\mathscr A$ , the identity morphism $\id_X: X\to X$ is an admissible epimorphism.
\item[E$1$] The class of admissible monomorphisms is closed under composition.
\item[E$1^\op$] The class of admissible epimorphisms is closed under composition.
\item[E$2$] The push-out of an admissible monomorphism along an arbitrary morphism exists and yields an admissible monomorphism.
\item[E$2^\op$] The pull-back of an admissible epimorphism along an arbitrary morphism exists and yields an admissible epimorphism.
\end{description}

An additive category $\mathscr A$ is called \emph{quasi-abelian} if:
\begin{enumerate}
\item every morphism has a kernel and a cokernel (such a category is called \emph{pre-abelian}), and
\item the class of kernels is stable under push-out along arbitrary morphisms and the class of cokernels is stable under pull-back along arbitrary morphisms.
\end{enumerate}
\end{definition}

In quasi-abelian categories, the class $\mathscr E_{\rm max}$ of all kernel-cokernel pairs is an exact structure. Vice versa, the property that all kernel-cokernel pairs form an exact structure characterizes quasi-abelian categories among pre-abelian ones.

For example, the categories of Banach spaces and Fr\'echet spaces (or nuclear Fr\'echet spaces) are quasi-abelian categories. The kernel-cokernel pairs $X\xrightarrow{f} Y\xrightarrow{g} Z$ in them are those pairs of mono- and epimorphisms which are \emph{strict}, i.e., with closed image. In other words, $f$ should be injective and $g$ is surjective with $\ker(g) = \im(f)$ (set-theoretically). 

Let $\mathscr A$ be an additive category equipped with an exact structure $\mathscr E$. A chain complex $\to A_n \to A_{n+1} \to \dots$ is called \emph{exact} or \emph{acyclic} if each differential factors $A_n\to B_n\to A_{n+1}$ in such a way that $B_n\to A_{n+1}\to B_{n+1}$ is in $\mathscr E$.
A chain map $f : A\to B$ is called a \emph{quasi-isomorphism} if its mapping
cone is homotopy equivalent to an acyclic complex; in pre-abelian categories one can directly say that the mapping cone is acyclic.

The \emph{derived category} is then defined in the usual way, by localization of the homotopy category with respect to quasi-isomorphisms.

\subsection{Cosheaves} Cosheaves are the notion that one obtains from sheaves by inverting the arrows in the source category. In other words, a cosheaf on a site $\mathscr C$ valued in a suitable category $\mathscr A$ is a functor
$$\mathcal F: \mathscr C \to \mathscr A$$
(a \emph{pre-cosheaf}) that takes covers to colimits, that is, for every covering family $(U_i\to U)_i$ we have:
\begin{equation}\label{coeq1} \mathcal F(U) = \underset{\to}\lim \left(\bigsqcup_{i,j} \mathcal F(U_i\times_U U_j) \rightrightarrows \bigsqcup_i \mathcal F(U_i)\right).
\end{equation}
Moreover, if $\mathscr A$ is an additive category endowed with an exact structure, we may impose a stricter condition on the cosheaf that is related to the exact structure.

In this paper, all cosheaves are valued either in complex vector spaces without topology (in the non-Archimedean case), or in Fr\'echet spaces. Moreover, the sites are such that any covering family can be replaced by a single element $U'\to U$. (We will tacitly assume this for all sites in the discussion that follows.) Under these simplifying conditions, \eqref{coeq1} becomes simply a co-equalizer diagram:
$$ \mathcal F(U'\times_U U') \rightrightarrows \mathcal F(U') \to \mathcal F(U).$$
The stricter condition implied above, when $\mathcal F$ is valued in Fr\'echet spaces, is that the co-equalizer is \emph{strict}, that is, if we denote by $s$ and $t$ the ``source'' and ``target'' arrows above, then the image of $s_!-t_!$ is closed. From now on, when talking about a cosheaf valued in Fr\'echet spaces, we will be assuming this condition. (For a morphism $s$ in $\mathscr C$, the induced morphism in $\mathscr A$ is denoted by $s_!$.)

If $\mathscr C$ has a final object ``$*$'', the \emph{global sections functor} is obtained by evaluating at that object. If it does not have a final object, we can try to define the global sections functor as a colimit:
$$ \mathcal F(*):= \lim_{\underset{U\in\Ob(\mathscr C)}{\to}} \mathcal F(U),$$
but it is not, in general, clear that one can take the colimit valued in the same target category (e.g., Fr\'echet spaces).

Here, I take a more ad-hoc point of view on global sections and homology, using only \v Cech homology for cosheaves. More precisely, for $U\in\Ob(C)$ and a covering $V\to U$, we define $\check H_\bullet ^V(U, \mathcal F)$:
to be the (strict) quasi-isomorphism class of the complex:
\begin{equation}\label{complex} \to \mathcal F([V]_U^{i+1}) \to \mathcal F([V]_U^i) \to \dots \to \mathcal F(V\times_U V) \to \mathcal F(V)\to 0,\end{equation}
where $[V]_U^i$ denotes the $i$-fold fiber product of $V$ over $U$, and the differentials are obtained by alternating sums of the morphisms that one gets from $[V]^{i+1}_U$ to $[V]^i_U$ by forgetting the $j$-th copy, $j=0,\dots, i$.

When $\mathscr C$ has a final object, we say that the cosheaf is \emph{acyclic} if the complex above is quasi-isomorphic to $ 0 \to \mathcal F(U) \to 0$ for every $U\in \Ob(\mathscr C)$ and every cover $V\to U$. When $\mathscr C$ does not have a final object, but satisfies the same property, I will call it \emph{locally acyclic}. (A more careful definition would require this condition only after passing to covers, but this will be enough for the purposes of this paper.) Here is the relation of this notion to the functor of global sections:

\begin{lemma}
Assume that $\mathcal F\xrightarrow{f} \mathcal G$ is a strict epimorphism of cosheaves of vector spaces or Fr\'echet spaces over a site $\mathscr C$, that is: for every $U\in \Ob(\mathscr C)$ the map: $\mathcal F(U) \to \mathcal G(U)$ is a strict epimorphism. Assume that $\mathcal G$ is locally acyclic. Then the kernel pre-cosheaf $\ker f(U):= \ker(\mathcal F(U)\to \mathcal G(U))$ is a cosheaf, the sequence $\ker f(U)\to \mathcal F(U) \to \mathcal G(U)$ is a kernel-cokernel pair in the category of vector/Fr\'echet spaces for every $U$ and, moreover, for any cover $V\to U$ we have a distinguished triangle in the derived category of vector/Fr\'echet spaces:
$$\to \check H_\bullet ^V(U, \mathcal G)[-1] \to \check H_\bullet ^V(U, \ker f) \to \check H_\bullet ^V(U, \mathcal F) \to \check H_\bullet ^V(U, \mathcal G) \to  $$
\end{lemma}

\begin{remark}
The statement about distinguished triangles holds in a much more general setting of kernel-cokernel pairs $\mathcal H\to \mathcal F \to \mathcal G$ of cosheaves; however, since there may be issues with cosheafification (which is needed to construct kernels) in the category of Fr\'echet spaces, I will avoid talking about exact structures of Fr\'echet cosheaves in more generality.
\end{remark}

\begin{proof}
The proof is as in the abelian case, by chasing arrows along the following diagram, and is left to the reader:

$$ \xymatrix{ & \vdots \ar[d] & \vdots \ar[d] & \vdots \ar[d] & \\
0 \ar[r] & \ker f(V\times_U V) \ar[d]\ar[r] & \mathcal F(V\times_U V) \ar[d]\ar[r] & \mathcal G(V\times_U V) \ar[d]\ar[r] & 0 \\
0\ar[r] & \ker f(V) \ar[d]\ar[r] &  \mathcal F(V)  \ar[d]\ar[r] & \mathcal G(V) \ar[d]\ar[r] & 0 \\
0\ar[r] & \ker f(U) \ar[d]\ar[r] &  \mathcal F(U)  \ar[d]\ar[r] & \mathcal G(U) \ar[d]\ar[r] & 0 \\
 & 0 & 0 & 0 &}$$
 
Notice that, without any assumptions on the acyclicity of $\mathcal G$, the fact that $f$ is a strict epimorphism implies that for any cover $V\to U$, the map $\ker f(V)\to \ker f(U)$ is a strict epimorphism. The acyclicity assumption is used to say that the kernel of this map is precisely the image of $\ker f(V\times_U V)$.

\end{proof}

Now let $X$ be a restricted topological space, and let $\mathscr C = X_{\rm Zar}$ be the ``Zariski site'' generated by its topology, i.e., its objects are open subsets of $X$, morphisms are inclusions, and covering families are finite families of morphisms which are jointly surjective. In fact, it is more convenient to enlarge $\mathscr C$ to include disjoint unions of open subsets of $X$ as its objects, with the obvious morphisms, so that coverings can be represented by a single morphism $V\to U$. 

A cosheaf $\mathcal F$ on $X_{\rm Zar}$ is called \emph{flabby} if for any inclusion of open subsets $U\to V$ the induced map $\mathcal F(U)\to \mathcal F(V)$ is a closed inclusion. 

\begin{lemma}\label{flabby}
Let $\mathcal F$ be a cosheaf of vector spaces or of Fr\'echet spaces on the Zariski site of a restricted topological space. If $\mathcal F$ is flabby, then it is acyclic.
\end{lemma}

\begin{proof}
If we were to form the complex \eqref{complex} in the category of cosheaves valued in vector spaces without topology, then the statement about acyclicity would follow from usual considerations (or from the even more common theory of sheaves, applied to the linear duals). We know, moreover, that the maps in \eqref{complex} are continuous, so in degree $\ge 1$, where the complex as abstract vector spaces is exact, the images of the differentials are closed. On the other hand, in degree zero we have by the cosheaf axioms that $\mathcal F(V\times_U V)\to \mathcal F(V) \to F(U)\to 0$ is strictly exact, which shows that \eqref{complex} is strictly quasi-isomorphic to the complex $0\to \mathcal F(U)\to 0$.
\end{proof}

\subsection{Homology for smooth $F$-representations}

This subsection aims to explain the homological meaning of the derived coinvariant functor of smooth $F$-representations of real algebraic (or Nash) groups, that I defined in an ad hoc way in \S \ref{Schwartz-quotient}. Recall that the category $\mathcal M_H$ of smooth $F$-representations of such a group $H$ is equivalent to the category of non-degenerate $A$-modules, where $A:=\mathcal S(H)$ is the convolution algebra of Schwartz measures on $H$. 

The formalism of homological algebra that we need is due to Deligne from SGA4 \cite{Deligne}, and presented very nicely in \cite{Buehler}. To apply it to our setting, we will use the results of Taylor \cite{Taylor} which, however, only hold for unital algebras. Therefore, we set
$$ \tilde A = A \oplus \CC,$$
with the algebra structure extending that of $A$ and making $1\in \CC$ the identity element. It is a  nuclear Fr\'echet algebra, and hence we have uniquely defined completed tensor products: $ \tilde A\hat \otimes V$ for any Fr\'echet space $V$. We let $\mathscr A$ be the category of Fr\'echet $\tilde A$-modules, and we endow it with the exact structure of all \emph{strongly exact} sequences as in \cite[\S 2]{Taylor}. This is a more restrictive condition than strict exactness; a sequence of Fr\'echet $\tilde A$-modules
$$ \dots \to V_{n+1} \to V_n \to V_{n-1}\to \dots$$
is strongly exact if it is strictly exact, and in addition every kernel and cokernel admit complements \emph{as topological vector spaces} (i.e., not necessarily as $\tilde A$-modules). 

On the other hand, the category $\mathscr F$ is quasi-abelian, and hence has a canonical (maximal) exact structure, that of strict exact sequences.

We assume given a continuous, non-zero homomorphism of algebras: $\int: A \to \CC$, or, equivalently a non-degenerate $A$-module structure on $\CC$ (where $a\in A$ acts as multiplication by $\int a$). For our Schwartz algebra this functional is the total integral, as the notation suggests. We extend it to $\tilde A$ as the identity on $\CC$. The \emph{rough} coinvariant functor corresponds to the association
$$ V \mapsto V/\mbox{\tiny the closure of the subspace spanned by vectors of the form } (av - \int a\cdot v), \,\, a\in \tilde A, v\in V.$$
It is a functor
$$C: \mathscr A \to \mathscr F,$$
and it coincides with our ``rough'' functor of $H$-coinvariants on non-degenerate modules of $\mathcal S(H)$.

A \emph{free} object in $\mathscr A$ is an object of the form $V = \tilde A\hat\otimes W$, where the action of $\tilde A$ is on the first factor by left multiplication. Free objects have the following properties:
\begin{itemize}
\item They are \emph{projective}, i.e., if $V$ is free, then the (vector space-valued) functor $\Hom_A (V, \bullet)$ is exact --- it turns strongly exact sequences in $\mathscr A$ into exact sequences of vector spaces.
\item We have $CV = \CC \hat\otimes_A V  = W$, where $\CC \hat\otimes_A$ denotes the quotient of $\CC \hat\otimes V$ by the image of the map: $A\hat\otimes V\to V$, $a\otimes v\mapsto av - \int a\cdot v$. The quotient is taken \emph{without} closure, and in particular the statement here means that the image is closed.
\end{itemize}
The first statement follows from the fact that $\Hom_{\tilde A} (V, Z) = \Hom (W, Z)$, cf.\ \cite[Proposition 1.3]{Taylor}. (The morphisms on the left are in $\mathscr A$, and those on the right in $\mathscr F$.) The second is \cite[Proposition 1.5]{Taylor}. Clearly, for both statements, it doesn't matter whether we use $A$ or $\tilde A$. The second statement easily extends to \emph{all} projective objects, since by \cite[Proposition 1.4]{Taylor} those are precisely the direct summands of free ones. Thus, the functor $C$ of coinvariants, restricted to the fully exact subcategory $\mathscr P$ of projective objects in $\mathscr A$, is exact.

As we will recall in a moment, every object in $\mathscr A$ admits a projective resolution. By \cite[Theorem 10.22]{Buehler} with arrows reversed (the second condition of this theorem is trivially satisfied by projectivity), this implies that the derived category $\mathcal D^- \mathscr P$ is equivalent to $\mathcal D^- \mathscr A$. By \cite[Lemma 10.26]{Buehler} (again reversing arrows), the functor $C$, when applied to a complex of projective objects representing a given object in $\mathcal D^- \mathscr A$, yields the \emph{total derived functor} $L_C$ of $C$:
$$L_C: \mathcal D^- \mathscr A \simeq \mathcal D^- \mathscr P \to \mathscr F,$$
which is characterized by the following universal property:

For every complex $A_\bullet$ (of objects  in $\mathscr A$) it
represents, in the derived category $\mathcal D^-\mathscr F$, the functor which assigns 
to each $K_\bullet \in \Ob(\mathcal D\mathscr F)$, the set of equivalence classes of pairs of diagrams:
$$\begin{cases}  C(A'_\bullet) \xrightarrow{f} K_\bullet, \\  A'_\bullet \xrightarrow{s} A_\bullet, \end{cases}$$
where $s$ is a (strong) quasi-isomorphism of complexes in $\mathscr A$. For more details, and the notion of equivalence, cf.\ \cite[\S 10.6]{Buehler}

We are left with showing that every object in $\mathscr A$ admits a projective resolution, thus establishing that $\mathcal D^- \mathscr P$ is equivalent to $\mathcal D^- \mathscr A$.
For notational simplicity, we set $\tilde A_n =$ the completed tensor product of $n$-copies of $\tilde A$, and analogously for $A$. By \cite[\S 2]{Taylor}, for every $V\in \Ob(\mathscr A)$ the sequence
\begin{equation}\label{res} \dots \to \tilde A_{n+1} \hat\otimes V \xrightarrow{d} \tilde A_n \hat\otimes V \to \dots \to \tilde A\hat\otimes V \to V \to 0\end{equation}
is a strongly exact resolution by free modules. By definition, if we number the copies of $\tilde A_{n+1}$ from $0$ to $n$, the action of $\tilde A$ is only on the left copy, and the boundary maps are 
\begin{equation}\label{dboundary}d(a_0\otimes \dots \otimes a_n\otimes v ) =  \sum_{i=0}^{n-1} (-1)^i a_0\otimes \dots\otimes a_i a_{i+1} \otimes \dots \otimes v + (-1)^n a_0 \otimes \dots \otimes a_{n-1} \otimes a_n v.\end{equation}

Hence, by \cite[Theorem 10.25]{Buehler}, for every strongly exact sequence of objects in $\mathscr A$:
$$ 0\to V_1 \to V_2 \to V_3\to 0,$$
we get a distinguished triangle in $\mathcal D^- \mathscr F$:
$$ L_C V_3[-1] \to  L_C V_1 \to L_C V_2 \to L_C V_3.$$

With all this material obtained from the literature, we would now like to prove that, for the computation of $L_C$, one may replace in \eqref{res} all copies of $\tilde A$ by $A$, thus arriving at the description of derived coinvariants of \S \ref{Schwartz-quotient}.

\begin{proposition}
The natural inclusion from the complex
\begin{equation}\label{nontildecomplex}\CC \hat\otimes_A \left(\dots \to A_{n+1} \hat\otimes V \xrightarrow{d} A_n \hat\otimes V \to \dots \to A\hat\otimes V\to 0 \right)\end{equation}
into the complex
\begin{equation}\label{tildecomplex}\CC \hat\otimes_A \left(\dots \to \tilde A_{n+1} \hat\otimes V\xrightarrow{d} \tilde A_n \hat\otimes V \to \dots \to \tilde A\hat\otimes V\to 0\right) \end{equation}
is a (strict) quasi-isomorphism. 
\end{proposition}
(The application of $\CC \hat\otimes_A $ to the entire complex just means application to each element or, by the above, applying the functor of coinvariants.)

\begin{proof}
Notice that $\CC \hat\otimes_A A_{n+1} \hat\otimes V \simeq A_n \hat\otimes V$. We explicate the boundary maps, which we will denote by $\delta$ in order to distinguish them from the boundary maps $d$ of \eqref{dboundary}:
\begin{equation}\label{deltaboundary} \delta(a_1\otimes \dots \otimes a_n\otimes v) = (\int a_1)\cdot a_2 \otimes \dots \otimes a_n\otimes v + \end{equation}
$$+ \sum_{i=1}^{n-1} (-1)^i a_1\otimes \dots \otimes a_i a_{i+1} \otimes \dots \otimes v + (-1)^n a_1 \otimes \dots \otimes a_{n-1}\otimes a_n v.$$

Now, for every $0\le i \le n$, let
$$ B_n^i:= \tilde A_{i-1} \hat\otimes \CC \hat\otimes A_{n-i}\hat\otimes V.$$
The convention is that $B_n^0 = A_n \hat\otimes V$. It is then clear from the definitions:
$$ \tilde A_n \hat\otimes V = \bigoplus_{i=0}^n B_n^i.$$

Our goal is to show that the summands with $i\ne 0$ do not contribute anything to the homology of \eqref{tildecomplex}.

We claim that, for $i\ge 1$, $\delta(B_n^i) \subset B_{n-1}^{i-1} \oplus B_{n-1}^i$, with the first summand not appearing when $i=1$ and the second not appearing when $i=n$. More precisely, as can easily be seen from \eqref{deltaboundary}, we have
\begin{equation} \delta(a_1\otimes \dots a_{i-1} \otimes 1 \otimes a_{i+1} \otimes \dots \otimes a_n\otimes v) = \end{equation}
$$=\delta(a_1\otimes \dots a_{i-1}) \otimes 1 \otimes (a_{i+1} \otimes \dots \otimes a_n\otimes v )+ $$
$$+ (-1)^{i+1} (a_1\otimes \dots a_{i-1}) \otimes 1 \otimes d(a_{i+1} \otimes \dots \otimes a_n\otimes v).$$

Here by $\delta$ we denote the morphism given by the \emph{same} formula as \eqref{deltaboundary} when $V=\tilde A$, and $d$ denotes the morphism \eqref{dboundary}. When $i=1$, the convention is that $\delta=0$, and when $i=n$, then $d=0$. Hence, omitting the summands with $i=0$ from \eqref{tildecomplex} (which correspond to \eqref{nontildecomplex}), we still get a complex
\begin{equation}\label{newcomplex} \bigoplus_{i=1}^{n+1} B_{n+1}^i \to \bigoplus_{i=1}^n B_n^i \to \dots \to \CC\otimes V \to 0.
\end{equation}
Moreover, from the above formulas it is clear that \eqref{newcomplex} is the tensor product of the complexes
$$ \tilde A_{n+1} \xrightarrow{\delta} \tilde A_n \to \dots \to \tilde A \to 0$$
(with $\tilde A$ in degree $1$) and
$$ A_{n+1}\hat\otimes V \xrightarrow{d} A_n \hat\otimes V \to \dots \to A\otimes V\to V \to 0$$
(with $V$ in degree $0$).

Both complexes are acyclic (strictly exact): the first computes the homology of $\tilde A$ as an $\tilde A$-module, which is free and hence acyclic, and the second is the resolution \eqref{res}, with $\tilde A$ replaced by $A$ (and easily seen to be a strictly exact resolution, again, though not necessarily strongly exact). Thus, \eqref{newcomplex} is acyclic, and \eqref{tildecomplex} is (strictly) quasi-isomorphic to \eqref{nontildecomplex}.
\end{proof}

\section{Asymptotically finite functions} \label{app:finite}

\subsection{Derivative arrangements and a criterion for asymptotic finiteness} \label{criterion}

Let us consider a toric variety $Y$ for a real torus $T$ as in \S \ref{finite-toric}. In particular, $Y$ is considered as a normal embedding for a torus quotient $T'$ of $T$.

Let $E$ be an exponent arrangement for $Y$; it was used in \S \ref{finite-toric} to describe a cosheaf on $\mathcal F_E$, consisting of functions on $T(\RR)$. We will describe an efficient way of checking whether a function on $T(\RR)$ is a section of $\mathcal F_E$ over $Y(\RR)$. The discussion here carries over, essentially verbatim, and will be applied to the case of the equivariant toroidal compactifications of the automorphic quotient $[H]=H(k)\backslash H(\adele)$. For notational simplicity we present the case of toric varieties.

To describe such a criterion, recall that the toric variety $Y$ is described by a fan $\F$ of strictly convex rational polyhedral cones on the vector space
$$\mathfrak t' = \Hom(\Gm,T')\otimes\RR,$$
and that the cones in the fan are in bijection $C\leftrightarrow Z_C$ with the geometric orbits of $T$ on $Y$, in such a way that the relative interior of $C$ consists precisely of those cocharacters $\lambda$ into $T'$ such that $\lim_{\lambda\to 0}\lambda(t)\in Z_C$. The exponent arrangement $Z\mapsto E(Z)$ can be considered as an arrangement for the fan: $C\mapsto E(C):= E(Z_C)$. We recall that $E(Z_C)$ consists of characters of the $\RR$-points of the stabilizer $T_C$ of points on $Z_C$.

The sub-fan consisting of the cone $C$ and all its faces  corresponds to the affine, open subvariety $Y_C\subset Y$, previously denoted as $Y_{Z_C}$. To describe the behavior of sections locally, in a neighborhood of a point $z\in Z_C(\RR)$, we may replace $Y$ by this affine open subvariety described by the fan of the cone $C$.

Let $D$ be a cone in the fan. (We use a different letter now, because the following will be used for every face $D$ of $C$ to describe the behavior in a neighborhood of a point of $Z_C$.)

We define a \emph{derivative arrangement} $(\F_D', E)$ as follows:

\begin{itemize}
\item The new fan $\F'_D$ will live on the vector space
$$\mathfrak t'_D:= \mathfrak t/\mathfrak t_D,$$
where $\mathfrak t_D$ is the subspace spanned by cocharacters into $T_D$; thus, the new fan will correspond to an embedding of the torus $T'_D:= T/T_D$.

To define the fan $\F'_D$, consider the set of orbits which are contained in the closure of $Z_D$; they correspond to all cones $D'\in \F$ which contain $D$. The fan $\F'_D$ will consist of the images of all those cones $D'$ in $\mathfrak t'_D$, which are strictly convex, as required. 

\item It inherits the multisets $E(D')$ of exponents from the original exponent arrangement. Thus, we do not need a new symbol for the exponents.
\end{itemize}

The pair $(\F'_D, E)$ gives rise to a toric embedding $Y'_D$ of $T'_D$, and a cosheaf over it which we will denote by $\mathcal F_{D, E}$; its restriction to the open stratum $Z_D$ is what in the notation of \S \ref{finite-toric} was $\mathcal F_{Z_D, E(Z_D)}$. 

Now we describe how to detect if a function on $T(\RR)$ coincides with an element of $\mathcal F_E(Y)$ in a (semi-algebraic) neighborhood $V$ of a point $z\in Z_C(\RR)$. The neighborhood $V$ will be taken in $Y_C$, according to Remark \ref{contraction}. It will also be taken to be compact --- and in particular with compact image in $Z_C(\RR)$ under the contraction map.

\begin{lemma}\label{lemmacriterion} 
Let $V= \bigcup_D U_D$ be a decomposition into subsets, one for each non-trivial face $D$ of $C$ (i.e., each non-open orbit containing $Z_C$ in its closure), such that $U_D$ is bounded away\footnote{By ``bounded away'' from an orbit $Z$ we will mean that it is disjoint from a semi-algebraic neighborhood of $Z(\RR)$.} from every orbit not contained in the closure of $Z_D$ (i.e., from any $Z_E$ with $D \not\subset E$). 

Then a smooth function $f$ on $T(\RR)$ coincides with an element of $\mathcal F_E(Y)$ on $V$ if and only if on $U_D$ it coincides, up to an element of $\mathcal F_E(T')$ (i.e., a rapidly decaying section), with a section of $\mathcal F_{D,E}$ over the closure of $Z_D(\RR)$.
\end{lemma}

The picture here shows such a partition of a neighborhood:
\begin{center}
\includegraphics[scale=0.2]{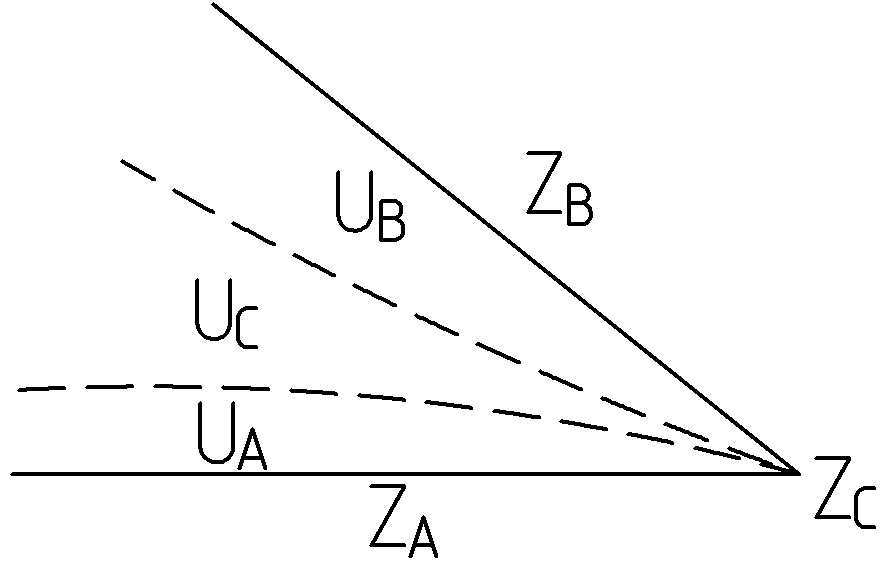}
\end{center}
I will describe a way to get such a partition after the proof; the idea is, essentially, to start by removing from $V$ neighborhoods of the maximal orbits (= one-dimensional edges of $C$), then remove from the remainder neighborhoods of the orbits of codimension two, etc. The proof also has such an inductive structure.

\begin{proof}
Let us say that two orbits are not comparable if one is not contained in the closure of the other. In that case, they can be separated by (semi-algebraic) open neighborhoods.

Let $Z_D$ be a non-open orbit of \emph{maximal} dimension in $Y_C$, that is: $D$ is a face of $C$ of dimension one. I claim:

\begin{quote}
In a neighborhood $U$ of $Z_D$, which is bounded away from any orbit not comparable to $Z_D$, restrictions to $U$ of elements of $\mathcal F_E(Y(\RR))$ are precisely the restrictions of those functions on $T(\RR)$ which differ by elements of $\mathcal F_E(T'(\RR))$ (i.e., sections of rapid decay on $T'(\RR)$) from sections of $\mathcal F_{D,E}$ over the closure of $Z_D$.
\begin{equation}\label{simplequote} \end{equation}
\end{quote}

This is immediately true by definition for those sections whose germs at orbits in the closure of $Z_D$ are zero (i.e., which are sections of $\mathcal F_E$ on an open set not containing any orbits smaller than $Z_D$). We now check it inductively on $\dim Z_D - \dim Z_{D'}$, where $Z_{D'}$ is an orbit in the closure of $Z_D$, of smallest dimension such that the germ of a given section $f$ at $Z_{D'}(\RR)$ is non-zero. Using the notation of the definition, in a neighborhood of $Z_{D'}(\RR)$ the section is equal to $f_\sigma$ plus a section whose germ at $Z_{D'}(\RR)$ is zero. Since $f_\sigma$ is a section of $\mathcal F_{D', E}$  and hence coincides \emph{a fortiori} with a section of $\mathcal F_{D,E}$ in that neighborhood, by the inductive hypothesis we are done.

We started with \eqref{simplequote} because it is easier to state, but in fact it is a special case of the following, which is proved by exactly the same argument:

\begin{quote}
If $Z_D$ is \emph{any} orbit in $Y$, $U$ is a neighborhood $U$ of $Z_D$ which is bounded away from any orbit not comparable to $Z_D$, and $U'$ is obtained from $U$ by removing neighborhoods of all orbits of larger dimension, then restrictions to $U'$ of elements of $\mathcal F_E(Y(\RR))$ are precisely the restrictions of those functions on $T(\RR)$ which differ by elements of $\mathcal F_E(T'(\RR))$ (i.e., sections of rapid decay on $T'(\RR)$) from sections of $\mathcal F_{D,E}$ over the closure of $Z_D$.
\begin{equation}\label{involvedquote} \end{equation}
\end{quote}

And this implies the lemma.
\end{proof}

I explain how to get a partition into subsets $U_D$ as above. We can choose a semi-algebraic lift of the contraction map:
$s:Z_C(\RR)\to T'(\RR),$
and then assume that the given neighborhood of $z\in Z_C(\RR)$ is a compact subset of the set
$$ \{ s(z')\cdot t\,\, | \,\, z' \in Z_C(\RR), t\in T_C(\RR), \log(t)\in C\}.$$
(Recall that under the log map: $T(\RR)\to \mathfrak t_\RR \to \mathfrak t'_\RR$, elements of $T_C$ have image in the linear span of $C$.)

We can then describe a partition by describing a partition of the cone $C$ and pulling it back via the log map.

We do this by ``moving'' the maximal faces of $C$ into the strict interior; that is, if we take a cross-section of the cone with an affine hyperplane $A$ in the linear span of $C$, meeting all of its non-trivial faces, then the cone $C$ is defined by the ``positive'' sides of a set of hyperplanes $H_i\subset A$, and we can consider a small parallel translation $H_i\to H_i'$ of each of those hyperplanes towards the interior of $C$. For a given face $D$, defined as the intersection of the hyperplanes $H_i$, $i\in I_D$, the subset $U_D$ will be determined by the set $D'$ of those elements of the cone $C$ which lie on the ``negative'' side of the corresponding translated hyperplanes $H_i'$, $i\in I_D$. We observe that the set $D'$ only meets the relative interiors of the cones containing $D$ in their closure, and therefore the corresponding set $U_D$ will be bounded away from orbits that are not in the closure of $Z_D$.

\subsection{Completion of the proof of Theorem \ref{isfinite}}

Recall that we have an equivariant toroidal embedding $[H]^\F$ of $[H]=H(k)\backslash H(\adele)$, determined by the fan $\F$ on the anti-dominant Weyl chamber $\mathfrak a^+$, which in turn is determined by the weights of the representation $V$ of $H$.

Let us fix a minimal parabolic $P_0$ with a Levi subgroup $M_0$ as before, and consider all cocharacters into the universal maximal split torus $A$ as cocharacters into $H$ via the embedding $A\to M_0$ determined by $P_0$.

For any cone $C\in \F$ we define
\begin{equation}\label{fC}
f_C (h) := \sum_{\gamma\in V_{C,0}} \int_{V_{C,+}(\adele)} f((\gamma+v) h) dv,
\end{equation}
where the spaces $V_{C,0}$, $V_{C,+}$, $V_{C,-}$ are as in \eqref{decompC}.

It is easy to see:
\begin{lemma}
The function $f_C$ is a smooth function on $[H]_{P_C}$, not depending on the choice of $(P_0,M_0)$, and an $A_C(\adele)$-eigenfunction with eigencharacter $\chi_C$.
\end{lemma}

\begin{proof}
Indeed, if $P_C, M_C$ are the standard parabolic and Levi in the class associated to $C$, induced from the chosen pair $(P_0,M_0)$, the subspaces $V_{C,0}$, $V_{C,+}$ are $M_C$-stable, since $A_C$ lies in the center of $M$; moreover, although $V_{C,0}$ is not necessarily stable under the unipotent radical $\mathcal U(P_C)$ of $P_C$, any affine subspace of the form $\gamma+V_{C,+}$ is, because it can be characterized as the subspace of those elements for which $\lim_{t\to 0} v\cdot \lambda(t) = \gamma$ for $\lambda$ a cocharacter in the relative interior of $C$, and $\lim_{t\to 0} \lambda(t)^{-1} u \lambda(t) =1$ for all such $\lambda$ and $u\in \mathcal U(P_C)$. The action of $\mathcal U(P_C)$ is by affine automorphisms on those affine subspaces, and since $\mathcal U(P_C)$ has trivial character group, it has to preserve Haar measure.
\end{proof}

Now we will use the derivative arrangements $(\F_C',E)$ defined in \S \ref{criterion}.

By decreasing induction on the dimension of $C$ (s.\ the following remarks), we may assume that we have proven:

\begin{quote}
The map $f\mapsto f_C$ is a continuous map from $\mathcal F(V(\adele))$ to $\Ff_{C, E}(A_C(\RR)\backslash [H]_{P_C}^{\F'_C})$,
 the space of asymptotically finite functions on $[H]_{P_C}$ which are $(A_C(\RR),\chi_C)$-eigenfunctions, 
with fan $\F'_C$ and exponent arrangement as in \S \ref{criterion}, that is: the fan $\F'_C$ on the vector space $\mathfrak a/\mathfrak a_C$ consists of the images of the cones of $\F$ containing $C$, and the arrangement is simply the restriction of $E$ to these cones.
\begin{equation}\label{quote}\end{equation}
\end{quote}

We remark the following:
\begin{itemize}
\item This claim is obtained by inductively applying the $\Sigma_V$-version (not the $\Sigma_N$-version) of Theorem \ref{isfinite} to the Schwartz function 
$$f':\gamma\mapsto \int_{V_{C,+}(\adele)} f((\gamma+v)) dv$$ on the $M(\adele)$-stable vector space $V_{C,0}(\adele)$. Notice that 
\begin{equation}\label{fC2} f_C(m) = \chi_C(m) \Sigma_{V_{C,0}}f'(m)\end{equation}
on the subspace $[M]\subset [H]_{P_C}$, where $\chi_C$ is the (unique) character corresponding to $C$ of the exponent arrangment $E$. (This embedding depends on the choice of parabolic $P_C$, but so does the subspace $V_{C,0}$.)

\item For any cone $D$ which contains $C$ as a face we have $V_{C,0}\supset V_{D,0}$, $V_{C,\pm}\subset V_{D,\pm}$. Thus, the exponent arrangement $E$ restricted to the cones containing $D$ in their closure is \emph{the same} as the exponent arrangement on $\mathfrak a$ obtained by the above recipe from the weights $\Phi(V_{C,0})$ of the representation $V_{C,0}$, \emph{multiplied by $\chi_C$}.

\item The basis of the induction is the case when $C$ is of full dimension, in which case $A_C=A$, $P_C = P_0$ and the quotient $A_C(\RR)\backslash [H]_{P_C}$ is compact, and it is immediately clear that \eqref{fC2} represents a continuous map to the space of smooth $A_C(\RR)$-eigenfunctions with eigencharacter $\chi_C$ on $[M]$.
\end{itemize}

Now we will check the asymptotic finiteness of the function $\Sigma_V f$ in the vicinity of a point $z\in Z_C$ by adapting the criterion of Lemma \ref{lemmacriterion}. Namely, we restrict our attention to a sufficiently small neighborhood $U$, namely a neighborhood satisfying the following conditions:
\begin{itemize}
\item it belongs to a small neighborhood of the $P_C$-cusp which is isomorphic to a neighborhood of the cusp in $P_C(k)\backslash H(\adele)$ under the map \eqref{leftright}; 
\item let $U'$ be the homeomorphic preimage of $U$ in a neighborhood of the cusp in $P_C(k)\backslash H(\adele)$, then $U'$ is the preimage of a neighborhood $U''$ of $z$ in $[H]_{P_C}^\F$ as described before Lemma \ref{lemmacriterion}, namely: it is compact and belongs to orbits which contain $Z_C$ in their closure (where $Z_C$ is, by definition, a stratum both in $[H]$ and in $[H]_{P_C}$.
\end{itemize}

We can then partition $U$ into subsets $U_D$, one for each non-zero face $D$ of $C$, such that $U_D$ is bounded away from all orbits not in the closure of $Z_D$. We may assume that the original neighborhood is sufficiently close to the $P_C$-cusp, and hence also to the $P_C$-cusp for every face $D$ of $C$.

Then I claim:

\begin{lemma}
The restriction of $\Sigma_V f$ to $U_D$ coincides, up to the restriction of a rapidly decaying function on $[H]$ (depending continuously on $f$), with $f_D$.
\end{lemma}

By a straightforward adaptation of Lemma \ref{lemmacriterion} and by \eqref{quote}, this is enough to prove the theorem. 

This lemma is proven by a variant of the simplified argument that we used in \S \ref{generalcase} for a single cocharacter $\lambda$. 
By the assumptions on $U''$, it is contained in a set of the form \begin{equation}\label{nbhd} M_C(k)\mathcal U(P_C)(\adele) A_C^+ K \subset [H]_{P_C},
\end{equation}
where $A_C^+$ is the preimage of cone $C$ under the log map
$$ A_C(\RR)\mapsto \mathfrak a_C \otimes_\QQ \RR,$$
and $K$ is a compact subset in $H(\adele)$. We have implicitly chosen a parabolic in class of $P_C$ in order to write the set on the right-hand side of \eqref{nbhd}, and although its Levi subgroup $M_C$ is not needed in order to make sense of the above expression, let us now choose such a decomposition, in order to consider elements of $A_C$ as elements in the center of $M_C$, and hence as elements of $H$. Thus, we have decompositions of the vector space $V$ as in \eqref{decompC}, and similarly when $C$ is replaced by any face $D$ of it, by the choice of Levi $M_D$ induced from $M_C$.

From \eqref{nbhd} it follows that $U\subset[H]$ is covered by the $H(k)$-coset of a set of the form $\Omega A_C^+ K \subset [H]_{P_C}$, where $\Omega$ is a compact subset of $\mathcal U(P_C)(\adele)$. Since the subset $U_D$ is bounded away (by a semi-algebraic neighborhood) from any stratum not in the closure of $Z_D$, it is covered by a subset of the form
$$\Omega A_C^D K \subset [H]_{P_C},$$
where $A_C^D$ is the preimage under the log map of a subcone $D'\subset C$ which is bounded away from any face of $C$ not containing $D$ in its closure. (The reader can compare with the suggested construction of the subsets $U_D$ at the end of \S \ref{criterion}.)

We can now attempt to estimate the restriction of $\Sigma_V f$ to a set of the form $\Omega A_C^D K$ where, we recall, the set $A_C^D$ is considered as a subset of $H(\adele)$ by the choice of specific Levi for the class of $P_C$.

Decomposing the sum over $V(k)$ as in \eqref{lambda} we will get terms that decay rapidly in $U_D$, except for the last one, which is equal to $f_D$. More precisely, consider the term
\begin{equation}\label{exp} \sum_{\gamma_-\in V_{D,-}(k)\smallsetminus\{0\}} \left( \sum_{\gamma_1\in (V_{D,0}+V_{D,+})(k)} f((\gamma_-+\gamma_1)h)\right),\end{equation}
and let $h\in A_C^D$. We can identify $A_C^D$ modulo a compact subgroup with its image $\mathfrak a_C^D$ in $\mathfrak a_C\otimes \RR$. The fact that the latter is bounded away from faces whose closure does not contain $D$ means that for every $A_C$-weight $\chi$ of the representation $ V_{D,-}$, viewed as a functional on $\mathfrak a_C$, the set $\mathfrak a_C^D$ lies in a belonging strictly in the negative half-space of $\chi$. From this, it easily follows that \eqref{exp} is of rapid decay on $A_C^D$, in a way that depends continuously on $f$. To incorporate the whole set $\Omega A_C^D K$ we just need to translate $f$ by $k\in K$ and change our choice of Levi $M_C$ by $\omega M_C \omega^{-1}$, for $\omega\in \Omega$ (and, correspondingly, the decomposition \eqref{decompC} for the face $D$). A way to encode this uniformly is to think of the spaces $V_{D, \pm}, V_{D,0}$ as abstract $k$-vector spaces, and let $\Omega$ parametrize their embeddings into $V$. These abstract spaces come with an action of the abstract torus $A_C$, and as $\omega$ varies in $\Omega$ and $k$ varies in $\Omega K$ we apply the sum \eqref{exp} to the $\omega$-pullback of the $k$-translate of $f$ to these spaces, with $h\in A_C^D$. It is then clear that the above estimates are uniform in the parameters $(\omega, k)$, and hence the expression \eqref{exp} is of rapid decay in $U_D$, in a way that depends continuously on $f$.

The same arguments apply to the second line of \eqref{lambda}, with $f$ replaced by a Fourier transform, and thus the above lemma, and the theorem, have been proved.

\bibliographystyle{plainurl}
\bibliography{stacks}

\articleend

\newpage

\title{Erratum to: The Schwartz space of a smooth semi-algebraic stack}

\author{Yiannis Sakellaridis}
\email{sakellar@rutgers.edu}
\address{Department of Mathematics and Computer Science, Rutgers University -- Newark, 101 Warren Street, Smith Hall 216, Newark, NJ 07102, USA.}

\maketitle

\setcounter{section}{4}

The purpose of this note is to fix two gaps in the construction of Schwartz spaces of semi-algebraic stacks in \cite{SaStacks-appendix}, and to strenghen some statements, replacing quasi-isomorphisms by homotopy equivalences. I am grateful to Avraham Aizenbud, Shachar Carmeli, and Dmitry Gourevitch for pointing out the gaps, and suggesting the stronger statements.

The first gap is in the proofs of Propositions 3.1.2 and 3.1.4, where I misquote \cite[Theorem A.1.1]{AGKl-appendix} and write a Schwartz function as a product of two Schwartz functions. There is also an obvious typo in the statement of Proposition 3.1.4: the sequence appearing should end with $\xrightarrow{\partial_0} \mathcal S(Y)\to 0$. 
Moreover, with this gap corrected, a stronger statement is actually proven in these two propositions than claimed. Namely, the sequence of Proposition 3.1.4 (with the aforementioned typo corrected) is not just strictly exact, but homotopic to zero. I formulate this here as a proposition, which supersedes both of Propositions 3.1.2 and 3.1.4 in the paper, and indicate the corrections needed for a complete proof.

\begin{proposition}\label{theprop}
 Let $\pi: X\to Y$ be a smooth surjective morphism of Nash manifolds. Let $[X]_Y^n=$ the fiber product of $n$ copies of $X$ over $Y$ (whose projection map to $Y$ is still denoted by $\pi$), and consider the complex 
 $$ (\mathcal S([X]_Y^n))_n: \cdots \to \mathcal S([X]_Y^3) \to \mathcal S([X]_Y^{2}) \to \mathcal S(X) \to 0,$$
 with differentials $\partial_n: \mathcal S([X]_Y^{n+1}) \to \mathcal S([X]_Y^{n})$ induced from the alternating sum of push-forwards when a copy of $X$ is deleted, as in \cite[Proposition 3.1.4]{SaStacks-appendix}. Consider $\mathcal S(Y)$ as a complex in degree zero, and the morphism of complexes 
 $$ \pi_!: (\mathcal S([X]_Y^n))_n \to \mathcal S(Y)$$ 
 induced by the push-forward $\pi_!: \mathcal S(X)\to \mathcal S(Y)$. This morphism is a homotopy equivalence.
\end{proposition}

\begin{proof}
\cite[Theorem A.1.1]{AGKl-appendix} states that any $f\in \mathcal S([X]_Y^n)$ can be written as a finite sum 
$$ f(\x) = \sum_{i=1}^m \phi_i(\pi(\x)) f_i(\x),$$
where $\phi_i$ is a Schwartz function on $Y$ and $f_i \in \mathcal S([X]_Y^n)$. Over an Archimedean field it is \emph{not} true, in general, that it can be written as a product $\phi\cdot f_0$, as claimed in \cite{SaStacks-appendix}.

But now, assuming, as in the proofs of \cite[Propositions 3.1.2 and 3.1.4]{SaStacks-appendix}, that the differential $\partial_{n-1} f \in \mathcal S([X]_Y^{n-1})$ vanishes, an extra complication arises, because only the sum $\sum_{i=1}^m \phi_i(\pi(\x)) \partial_{n-1} f_i(\x)$ vanishes, not each term $\partial_{n-1} f_i$ individually. The next step in the proofs is to disintegrate $\phi$ to an element $h$ of some space of ``relative Schwartz measures'' $\mathcal S'(X)$ (whose push-forwards to $Y$ are Schwartz functions --- see the proof of \cite[Proposition 3.1.2]{SaStacks-appendix} for details). For the argument to go through as stated, we need to do this compatibly for all $\phi_i$'s. Namely, let us assume that the base field is $F=\mathbb R$ (because in the non-Archimedean case there is no issue, and in the complex case we may work by restriction of scalars over $\mathbb R$ without changing the final statement). 

Let us first discuss the special case where the morphism $X\to Y$ admits a Nash section $\sigma: Y\to X$, which, in addition, extends to a \emph{tubular neighborhood} $\iota: Y\times B_r \hookrightarrow X$, where $r$ is the relative dimension of the map $\pi$, and $B_r$ is the open unit ball in $\mathbb R^r$. Then, choosing a Schwartz measure $\mu \in \mathcal S(B_r)$ with total mass $1$, we can set $h_i := \iota_! (\phi_i\otimes \mu) \in \mathcal S'(X)$. Then $\pi_! h_i = \phi_i$, and $\sum_i h_i(x_0) \partial_{n-1} f_i(x_1, \dots, x_{n-1}) =0$; the proofs of the two Propositions now go through as stated. Moreover, the tubular neighborhood gives rise to an embedding, again to be denoted by the same letter:
$$ \iota: [X]_Y^n \times B_r \hookrightarrow [X]_Y^{n+1}$$
(with $B_r$ determining the last coordinate), and a choice of $\mu$ as above allows us to define linear maps
$$H_n: \mathcal S([X]_Y^n) \to \mathcal S([X]_Y^{n+1})$$
by $H_n(f)= \iota_!(\mu \otimes f)$. This includes the case of $H_0: \mathcal S(Y) \to \mathcal S(X)$, which is a section for the push-forward map. One then easily checks that $H_n$ is a homotopy between $H_0 \circ \pi_!$ and the identity on the complex $(\mathcal S([X]_Y^n))_n$; in other words, $\pi_!: (\mathcal S([X]_Y^n))_n \to \mathcal S(Y)$ is a homotopy equivalence.

We have up to now assumed that the morphism $X\to Y$ admitted a section with a tubular neighborhood. Such tubular neighborhoods exist locally over $Y$ \cite[2.4.3]{AGdeRham-appendix}, \cite[Theorem 3.6.2]{AGSchwartz-appendix}. The last step to correct the proof is to show that all statements are local over $Y$ (in the semi-algebraic topology). For this, given a (finite) semi-algebraic upen cover $Y = \cup_j Y_j$, we use the ``Schwartz partition of unity'' of \cite[Theorem 4.4.1]{AGSchwartz-appendix}, which is a collection of tempered functions $u_j$, with $u_j$ supported on $Y_j$, $\sum_j u_j=1$, and the property that multiplication by $u_j$ turns a Schwartz function (or measure) on $Y$ to a Schwartz function (or measure) on $Y_j$. Of course, multiplication by $u_j\circ\pi$ will not change the property $\partial_{n-1} f=0$ (of $f\in \mathcal S([X]_Y^n)$), so we are reduced to the case where a section with a tubular neighborhood exists.
\end{proof}

The second gap is in the proof of functoriality in Theorem 3.3.1. Again, once the gap is fixed a stronger statement is actually proven:

\begin{theorem}
Let $\mathfrak X$ be a Nash stack. For any two presentations $X_1\to \mathfrak X$, $X_2\to \mathfrak X$, the Schwartz complexes 
 $$ (\mathcal S([X_i]_{\mathfrak X}^n))_n: \cdots \to \mathcal S([X_i]_{\mathfrak X}^3) \to \mathcal S([X_i]_{\mathfrak X}^{2}) \to \mathcal S(X_i) \to 0$$
 ($i=1,2$) are canonically homotopy equivalent;  and hence can be denoted by $\mathcal S_\bullet(\mathfrak X)$. 

The association $\mathfrak X\mapsto \mathcal S_\bullet(\mathfrak X)$ is functorial with respect to smooth 1-morphisms
of Nash stacks, up to homotopy.
\end{theorem}

\begin{proof}
For notational simplicity, let us in the proof denote $X_1$ by $X$, $X_2$ by $Y$, and $X^{(i)}Y^{(j)}:= [X]_{\mathfrak X}^i \times_{\mathfrak X} [Y]_{\mathfrak X}^j$. We also denote $X^{(1)}Y^{(1)}$ by $R_{XY}$, $X^{(2)}$ by $R_X$, and $Y^{(2)}$ by $R_Y$; hence, for $i, j\ge 1$ we have 
\begin{equation}\label{asproduct}X^{(i)}Y^{(j)} = [R_X]_X^{i-1} \times_X R_{XY} \times_Y [R_Y]_Y^{j-1}.
\end{equation}
This is a unique Nash manifold up to unique isomorphism, once the Nash manifolds $R_{XY}, R_X, R_Y$ (with their morphisms to $X, Y$) have been fixed. The ``presentations'' $X\to \mathfrak X$, $Y\to \mathfrak X$ implicitly include the groupoids $R_X\rightrightarrows X$, $R_Y\rightrightarrows Y$, but the gap in the proof of Theorem 3.3.1 is that it is not taken into account that $R_{XY}$ is only defined up to automorphisms over $X\times Y$. Thus, we need to make sure that, in the proof of Theorem 3.3.1, composition with automorphisms $\tau: R_{XY}\to R_{XY}$ over $X\times Y$ does not change the homotopy class of the equivalence $(\mathcal S(X^{(n)}))_n \xrightarrow\sim (\mathcal S(Y^{(n)}))_n$.

We revisit the proof, in order also to explain that it can be strengthened to a homotopy equivalence. The essential statement is that, if we consider the total complex $\mathbb T_{XY}$ associated to the bicomplex $(\mathcal S(X^{(i)} Y^{(j)}))_{i, j \ge 1}$, its natural push-forward maps to the complexes $(\mathcal S(X^{(i)}))_{i\ge 1}$, $(\mathcal S(Y^{(j)}))_{j\ge 1}$ are homotopy equivalences. 

Let us briefly see why: Applying Proposition \ref{theprop} above, for any $i$, the natural push-forward 
$$ (\mathcal S(X^{(i)}Y^{(j)}))_j \longrightarrow (0 \to \mathcal S(X^{(i)})\to 0)$$
is a homotopy equivalence. The construction of a homotopy inverse relied on choosing, locally, a section with a tubular neighborhood:
$$ X^{(i)} \times B_r \hookrightarrow X^{(i)}Y^{(1)}.$$

In our setting, we can choose once and for all a section with a tubular neighborhood 
\begin{equation}\label{tubular}
\iota:  X\times B_r \hookrightarrow R_{XY},
\end{equation}
 at least locally on $X$. (It is clearly enough to work locally over $X$ here, in order to prove the homotopy equivalence of the total complex with $(\mathcal S(X^{(i)}))_i$.) This \emph{induces} sections 
$$X^{(i)} \times B_r = [R_X]_X^{(i-1)} \times_X  X \times B_r \hookrightarrow X^{(i)}Y^{(1)}  = [R_X]_X^{(i-1)} \times_X R_{XY}$$
for all $i\ge 1$; the resulting homotopy inverses constructed in the proof of Proposition \ref{theprop} (denoted by $H_0$ there) will now be, by construction, \emph{chain maps of complexes}:
$$ H_0: (\mathcal S(X^{(i)}))_{i\ge 1} \to (\mathcal S(X^{(i)}Y^{(1)}))_{i\ge 1} \hookrightarrow \mathbb T_{XY}.$$

Thus, the statement of Proposition \ref{theprop} extends to the morphism of complexes $\mathbb T_{XY} \to (\mathcal S(X^{(i)})_{i\ge 1}$, and shows that it is a homotopy equivalence.

Let us now assume that $\tau: R_{XY}\to R_{XY}$ is an automorphism over $X\times Y$; by \eqref{asproduct}, it induces automorphisms $\tau_!$ of all Schwartz spaces $\mathcal S(X^{(i)} Y^{(j)})$, $i,j\ge 1$. I claim that the composition $\tau_!\circ H_0$ is still a homotopy inverse to the push-forward map $p_X:\mathbb T_{XY} \to (\mathcal S(X^{(i)}))_{i\ge 1}$. Indeed, in the construction of this homotopy inverse, $\tau$ just modifies the tubular neighborhood \eqref{tubular}, that is, $\tau_! H_0$ is obtained by the same construction, using the tubular neighborhood $\tau \circ \iota$. Thus, it is homotopy inverse to the canonical push-forward map $p_X$.

But $\tau_!\circ H_0$ is also homotopy inverse to the composition $p_X\circ \tau^{-1}_!$. We deduce that the morphisms $p_X$ and $p_X\circ \tau^{-1}_!$ are homotopic. The same holds for the analogous morphisms $p_Y$ and $p_Y\circ \tau^{-1}_!$, thus the homotopy class of the equivalence of Schwartz complexes
$$(\mathcal S(X^{(i)}))_{i\ge 1} \xrightarrow\sim (\mathcal S(Y^{(j)}))_{j\ge 1}$$
is fixed under automorphisms of $X\times_\mathfrak X Y$.

\end{proof}

\end{document}